%% file: Final.tex
\documentclass[a4paper,12pt]{article}

\usepackage{mathrsfs,amsthm,color,verbatim,bbm,amsmath,amsfonts,amssymb,newclude,nicefrac,amsfonts,graphicx,color,enumerate,hyperref,bm,geometry,dsfont}

\usepackage[nocompress]{cite}

\usepackage[latin1]{inputenc}

\usepackage[nosort,capitalise]{cleveref}

\usepackage{xcolor}
\hypersetup{
	colorlinks,
	linkcolor={red!60!black},
	citecolor={green!60!black},
	urlcolor={blue!60!black}
}

\newcommand{\R}{\mathbb{R}}
\newcommand{\N}{\mathbb{N}}
\newcommand{\E}{\mathbb{E}}
\newcommand{\id}{\operatorname{id}}
\renewcommand{\P}{\mathbb{P}}

\newcommand\numberthis{\addtocounter{equation}{1}\tag{\theequation}}

\newtheorem{theorem}{Theorem}[section]
\newtheorem{lemma}[theorem]{Lemma}

\newtheorem{cor}[theorem]{Corollary}
\newtheorem{definition}[theorem]{Definition}

\newtheorem{prop}[theorem]{Proposition}

\newtheorem{example}[theorem]{Example}

\newcommand{\smallsum}{\textstyle\sum}

\newcommand{\ANNs}{\mathbf{N}}
\newcommand{\activation}{a}
\newcommand{\activationDim}[1]{\mathfrak{M}_{\activation,#1}}
\newcommand{\functionANNgeneral}{\mathcal{R}_{\activation}}
\newcommand{\functionANN}{\mathcal{R}_{\mathfrak{r}}}
\newcommand{\paramANN}{\mathcal{P}}

\newcommand{\lengthANN}{\mathcal{L}}
\newcommand{\inDimANN}{\mathcal{I}}
\newcommand{\compANN}[2]{{#1 \bullet #2}}
\newcommand{\concANN}[2]{{#1 \odot_\mathbb{I} #2}}
\newcommand{\concPsiANN}[2]{{#1 \odot_{\Psi} #2}}

\newcommand{\outDimANN}{\mathcal{O}}

\newcommand{\dims}{\mathcal{D}}

\newcommand{\deltaIndex}{\min\{(T/N)^{1/2},1\}}
\newcommand{\deltaIndexProof}{\min\{\delta,1\}}

\newcommand{\downFixed}[1]{\lfloor #1\rfloor}
\newcommand{\upFixed}[1]{\lceil #1\rceil}

\newcommand{\indicator}[1]{\mathds{1}_{#1}}

\newcommand{\affineProcess}{Y}
\newcommand{\zigZagProcess}{\mathcal{Y}}

\newcommand{\covariance}{\operatorname{Cov}\!}
\newcommand{\cov}{\operatorname{Cov}}
\newcommand{\norm}[1]{ \left\| #1 \right\| }
\newcommand{\Norm}[1]{ \| #1 \| }

\newcommand{\qandq}{\qquad\text{and}\qquad}
\newcommand{\andq}{\text{and}\qquad}

\newcommand{\eps}{\varepsilon}

\newcommand{\LogBin}{\log_2}

\newcommand{\is}{\leftarrow}

\newcommand{\timeGrid}{\tau}

\newcommand{\dnnFunction}{F}

\newcommand{\drift}{\mathfrak{f}^1}
\newcommand{\initial}{\mathfrak{f}^0}

\newcommand{\vertiii}[1]{{\left\vert\kern-0.25ex\left\vert\kern-0.25ex\left\vert {#1} 
		\right\vert\kern-0.25ex\right\vert\kern-0.25ex\right\vert}}

\usepackage{tikz}
\usetikzlibrary{matrix,chains,positioning,decorations.pathreplacing,arrows}
\usetikzlibrary{shapes,fadings,arrows.meta}
\tikzset{
	font={\fontsize{9pt}{12}\selectfont}}
\usepackage{adjustbox}
\usepackage{float}

\crefname{figure}{Figure}{Figures}

\begin{document}
	\title{Space-time deep neural network\\ approximations for high-dimensional\\ partial differential equations}

\author{Fabian Hornung$^{1}$,
	Arnulf Jentzen$^{2,3,4}$, \\ and Diyora Salimova$^{4,5,6}$
	\bigskip
	\\
	\small{$^1$Faculty of Mathematics, Karlsruhe Institute of Technology, }\\
	\small{Germany, e-mail: fabianhornung89@gmail.com} 
	\\
	\small{$^2$School of Data Science and Shenzhen Research Institute of Big Data,}\\
	\small{The Chinese University of Hong Kong, Shenzhen (CUHK-Shenzhen),}\\
		\small{China,  e-mail: ajentzen@cuhk.edu.cn}
	\\
	\small{$^3$Applied Mathematics: Institute for Analysis and Numerics,}\\
	\small{Faculty of Mathematics and Computer Science, University of M\"unster, }\\
	\small{Germany, e-mail: ajentzen@uni-muenster.de}
	\\
	\small{$^4$Seminar for Applied Mathematics, Department}\\
	\small{ of Mathematics, ETH Zurich, Switzerland}
	\\
	\small{$^5$Department of Information Technology and}\\
    \small{Electrical Engineering, ETH Zurich, Switzerland}
  \\
\small{$^6$Department for Applied Mathematics, University of Freiburg,}\\
\small{Germany, e-mail: diyora.salimova@mathematik.uni-freiburg.de}
}

\maketitle

\begin{abstract}
It is one of the most challenging issues in applied mathematics to approximately solve high-dimensional partial differential equations (PDEs) and most of the numerical approximation methods for PDEs in the scientific literature suffer from the so-called curse of dimensionality in the sense that the number of computational operations employed in the corresponding approximation scheme to obtain an  approximation precision $\varepsilon >0$ grows exponentially in the PDE dimension and/or the reciprocal of $\varepsilon$. Recently, certain deep learning based  methods for PDEs have been proposed  and various numerical simulations for such methods suggest that deep  artificial neural network (ANN) approximations might have the capacity to indeed overcome the curse of dimensionality in the sense that  the number of real parameters used to describe the approximating deep ANNs  grows at most polynomially in both the PDE dimension $d \in \N$ and the reciprocal of the prescribed approximation accuracy $\varepsilon >0$. There are now also a few
rigorous mathematical results in the scientific literature which  substantiate this conjecture by proving that   deep ANNs overcome the curse of dimensionality in approximating solutions of PDEs.  Each of these results establishes that deep ANNs overcome the curse of dimensionality in approximating suitable PDE solutions 
at a fixed time point $T >0$ and on a compact cube $[a, b]^d$ in space but none of these results provides an answer to the question whether the entire PDE solution on $[0, T] \times [a, b]^d$ can be approximated by deep ANNs without the curse of  dimensionality. 
It is precisely the subject of this article to overcome this issue. More specifically, the main result of this work in particular proves for every $a \in \R$, $ b \in (a, \infty)$ that solutions of  certain Kolmogorov PDEs can be approximated by deep ANNs on the space-time region $[0, T] \times [a, b]^d$ without the curse of dimensionality. 
\end{abstract}

\begin{center}
\end{center}

\begin{center}
\\ 	approximation, partial differential equation, PDE, stochastic differential\\ equation, SDE, Monte Carlo Euler,  Feynman--Kac formula, ANN
\end{center}

\tableofcontents

\section{Introduction}
\label{sec:intro}

It is one of the most challenging issues in applied mathematics to approximately solve high-dimensional partial differential equations (PDEs) and most of the numerical approximation methods for PDEs in the scientific literature suffer from the so-called \emph{curse of dimensionality} in the sense that the number of computational operations  employed in the corresponding approximation scheme to obtain an  approximation precision $\varepsilon >0$ grows exponentially in the PDE dimension and/or the reciprocal of $\varepsilon$ (for such concepts cf., e.g., Bellman~\cite{Bellman1957}, Novak \& Ritter~\cite{NovakRitter1997}, Novak \& Wo\'{z}niakowski~\cite[Chapter~1]{Novak2008} and Novak \& Wo\'{z}niakowski~\cite[Chapter~9]{Novak2010}  and
for methods which do not suffer from the curse of dimensionality in the case of some special classes of nonlinear PDEs
cf., e.g., \cite{Henry-Labordere2014,Henry-Labordere2012,HenryOudjaneEtAl2019,EHutzenthaler2021,EHutzenthaler2019,HutzenthalerJentzenKruse2020},
\cite[Section~4]{BeckHutzenthalerEtAL2023},
\cite[Sections~2 and 6]{EHanJentzen2022}, and the references therein).

Recently, certain artificial neural networks (ANNs) based approximation methods for PDEs have been proposed  and various numerical simulations for such methods suggest (cf., e.g., \cite{weinan2017deep,Han2018PNAS,Raissi2018DeepHP,PhamWarin2021,Magill2018NeuralNT,LyeMishraRay2020,LongLuMaDong2018,JacquierOumgari2023,HurePhamWarin2020,HanLong2020,GoudenegeEtAl2020,FujiiTakahashi2019,Berg2018AUD,ChanMikaelWarin2019, weinan2018deep,Farahmand2017DeepRL,   henry2017deep,Sirignano2018,NueskenRichter2021,KhooLuYing2020,
	DissanayakePhan1994,Lagaris1998ArtificialNN,LiLuo2003,LiYing2022,ZhouHanLu2021} and the references mentioned therein) that deep ANNs  might have the capacity to indeed overcome the curse of dimensionality in the sense that  the number of real parameters used to describe the approximating deep ANNs  grows at most polynomially in both the PDE dimension $d \in \N = \{1, 2, \ldots\}$ and the reciprocal of the prescribed approximation accuracy $\varepsilon >0$.

 There are now also a few
 rigorous mathematical results in the scientific literature which substantiate this conjecture by proving that   deep ANNs overcome the curse of dimensionality in approximating solutions of PDEs; cf., e.g., \cite{ElbraechterEtAl2022,GrohsWurstemberger2023,JentzenSalimovaWelti2018,KutyniokEtAl2022,ReisingerZhang2020,GrohsJentzenSalimova2022,GononGrohsEtAl2022}.
Each of the references mentioned in the previous sentence establishes that deep ANNs overcome the curse of dimensionality in approximating suitable PDE solutions 
at a fixed time point $T >0$ and on a compact cube $[a, b]^d$ in space but none of the results in these references provides an answer to the question whether the entire PDE solution on $[0, T] \times [a, b]^d$ can be approximated by deep ANNs without the curse of  dimensionality.

It is precisely the subject of this article to overcome this issue. More specifically, the main result of this work, \Cref{theorem:DNNerrorEstimate} in Subsection~\ref{sub:cost} below, in particular proves for every $a \in \R$, $ b \in (a, \infty)$ that solutions of  certain Kolmogorov PDEs can be approximated by deep ANNs on the space-time region $[0, T] \times [a, b]^d$ without the curse of dimensionality. To illustrate the findings of this work in more details we now present in \Cref{thm:DNNerrorEstimateLaplace} below a special case of \Cref{theorem:DNNerrorEstimate}.

\begin{theorem}
	\label{thm:DNNerrorEstimateLaplace}
	Let 
	$	\ANNs
	=
	\cup_{L \in \N}
	\cup_{ l_0,l_1,\ldots, l_L \in \N }
	(
	\times_{k = 1}^L (\R^{l_k \times l_{k-1}} \times \R^{l_k})
)$,  let 	$A \colon (\cup_{d \in \N} \R^d) \to (\cup_{d \in \N} \R^d)$ satisfy   for all $d \in \N$, $x = (x_1,  \ldots, x_d) \in \R^d$ that
\begin{align}
	\label{eq:intro:activ}
A(x) = (\max\{x_1, 0\},  \ldots, \max\{x_d, 0\}),
\end{align}
let $\paramANN \colon \ANNs \to \N$  and $\mathcal{R}  \colon \ANNs \to( \cup_{k,l\in\N}\,C(\R^k,\R^l))$ 
satisfy
for all $ L\in\N$, $l_0,l_1,\ldots, l_L \in \N$, 
$
\Phi  
=
((W_1, B_1),\allowbreak \ldots, (W_L,\allowbreak B_L))
\in  \allowbreak
( \times_{k = 1}^L\allowbreak(\R^{l_k \times l_{k-1}} \times \R^{l_k}))
$,
$x_0 \in \R^{l_0}, x_1 \in \R^{l_1}, \ldots, x_{L} \in \R^{l_{L}}$ 
with $\forall \, k \in \N \cap (0,L) \colon x_k =A (W_k x_{k-1} + B_k)$  
that 
\begin{align}
\textstyle	\paramANN (\Phi)
	=
	\smallsum_{k = 1}^L l_k(l_{k-1} + 1), \quad \mathcal{R}(\Phi) \in C(\R^{l_0},\R^{l_L}), \quad \text{and} \quad 
(\mathcal{R}(\Phi)) (x_0) = W_L x_{L-1} + B_L,
\end{align}
for every $d \in \N$ let
	$ f_d \colon \R^d \to \R^d $  and 
	$ g_d \colon \R^d \to \R$
	be functions,
	let 
	$ T, \kappa, p \in (0,\infty) $,   \allowbreak
	assume for all
	$ d \in \N $, 
	$ \varepsilon \in (0,1] $ that there exist $\mathfrak{f}, \mathfrak{g} \in \ANNs$ such that for all 
	$ 
	x,y \in \R^d
	$ it holds
	that 
	\vspace{-1ex}
	\begin{gather}
	\label{eq:intro:ass:1}
	  	\mathcal{R}( \mathfrak{f})
	\in 
	C( \R^d,  \R^{ d }), \qquad  \mathcal{R}( \mathfrak{g})
	\in 
	C( \R^d,  \R),  \qquad 	\| 
	f_d( x ) 
	- 
	f_d( y )
	\|
	\leq 
	\kappa 
	\| x - y \|
	, 
		\\[0.5ex]
	\label{thm_intro:diff}
	\varepsilon|
	g_d( x )
	| 
	+\| f_d(x) 
	- 
	(\mathcal{R}(\mathfrak{f})) (x)
	\|
	+| g_d(x) 
	- 
	(\mathcal{R}(\mathfrak{g})) (x)
	|
	\leq 
	\varepsilon \kappa d^{ \kappa }
	(
	1 + \| x \|^{ \kappa }
	), \\[0.5ex]
	\label{eq:intro:ass:last}
	| (\mathcal{R}(\mathfrak{g})) (x) - (\mathcal{R}(\mathfrak{g})) (y)| \leq \kappa d^{\kappa} (1 + \|x\|^{\kappa} + \|y \|^{\kappa})\|x-y\|, \\[0.5ex]
	\text{and} \qquad 	\varepsilon^{\kappa} [\mathcal{P}( \mathfrak{f}) + \mathcal{P}( \mathfrak{g} ) ]+
	\|
	(\mathcal{R}(\mathfrak{f}))(x)    
	\|	
	\leq 
	\kappa ( d^{ \kappa } + \| x \| ),
	\end{gather}
and	for every $d \in \N$	let
$u_d \in C([0, T] \times \R^d, \R) $ be an at most polynomially growing viscosity solution of
	\begin{equation}
	\label{eq:intro:PDE}
	\begin{split}
	( \tfrac{ \partial }{\partial t} u_d )( t, x ) 
	& = 
	(\Delta_x u_d )( t, x ) +
	( \tfrac{ \partial }{\partial x} u_d )( t, x )
	\,
	f_d( x )
	\end{split}
	\end{equation}
	with 	$ u_d( 0, x ) = g_d( x )$
	for $ ( t, x ) \in (0,T) \times \R^d $.
	Then  there exists $c\in \R$ such that 	for all 
	$d\in\N$, $\varepsilon\in (0,1]$  there exists $\mathfrak{u} \in  \ANNs$ such that
	$\mathcal{R}(\mathfrak{u})\in C(\R^{d+1},\R)$,
	$\paramANN(\mathfrak{u})\le c \varepsilon^{-c}  d^{c}$, and 
	\begin{equation}
	\label{eq:intro:result}
	\begin{split}
	& \textstyle \left[
	\int_{[0,T]\times[0,1]^d}
	\left\vert
	u_d(y) 
	- 
	(\mathcal{R}(\mathfrak{u}))(y)
	\right\vert^p
	d y
	\right]^{\!\nicefrac{1}{p}} 
	\le
	\varepsilon.
	\end{split}
	\end{equation}
\end{theorem}

 \Cref{thm:DNNerrorEstimateLaplace} follows from \Cref{cor:DNNerrorEstimateLaplace} in Subsection~\ref{sub:cost} below. \Cref{cor:DNNerrorEstimateLaplace}, in turn, is a consequence of \Cref{theorem:DNNerrorEstimate} which is the main result of this article. 
In the following we add a few comments on some of the mathematical objects appearing in \Cref{thm:DNNerrorEstimateLaplace} above. 

Note in \Cref{thm:DNNerrorEstimateLaplace} that $\norm{\cdot} \colon (\cup_{d\in\N} \R^d) \to [0,\infty)$ is the function which satisfies for all $d \in \N$, $x = (x_1, \ldots, x_d) \in \R^d$ that
\begin{align}
\textstyle \norm{x} = \big[ \!\sum_{j=1}^d |x_j|^2\big]^{\nicefrac{1}{2}}
\end{align}
(standard norm, cf.\ \Cref{Def:euclideanNorm} below). In \Cref{thm:DNNerrorEstimateLaplace}  we approximate the solution functions $u_d \colon [0, T] \times \R^d \to \R$, $d \in \N$, of the PDEs in \eqref{eq:intro:PDE} by deep ANNs. We assume that the solution functions $u_d \colon [0, T] \times \R^d \to \R$, $d \in \N$, are at most polynomially growing which means that for every $d \in \N$ there exists $q\in (0, \infty)$ such that for all $t\in [0, T]$, $x \in   \R^d$ it holds that
\begin{align}
|u_d(t, x)| \leq q(1 + \|x\|^q).
\end{align}
This polynomial growth assumption ensures uniqueness of the solution functions $u_d \colon [0, T] \times \R^d \to \R$, $d \in \N$, of  the PDEs in \eqref{eq:intro:PDE}.

The set $\ANNs$ in \Cref{thm:DNNerrorEstimateLaplace} is a set of tuples of  pairs of real matrices and real vectors and we think of $\ANNs$  as the set of all  ANNs (cf.~\Cref{Def:ANN} below). Observe that \Cref{thm:DNNerrorEstimateLaplace} is an approximation result for  ANNs with the rectifier activation function and the corresponding rectifier functions are described through the function $A \colon (\cup_{d \in \N} \R^d) \to (\cup_{d \in \N} \R^d)$  appearing in \Cref{thm:DNNerrorEstimateLaplace}.

For every $\Phi \in \ANNs$ the number $\paramANN(\Phi) \in \N$ in \Cref{thm:DNNerrorEstimateLaplace}  corresponds to the number of real parameters employed to describe the ANN $\Phi$ (cf.~\Cref{Def:ANN} below).
For every $\Phi \in \ANNs$ 
the function 
\begin{align}
\mathcal{R}(\Phi) \in( \cup_{k,l\in\N}\,C(\R^k,\R^l))
\end{align}
corresponds to the realization function associated to the ANN  $\Phi$ (cf.~\Cref{Definition:ANNrealization} below). We also refer to \Cref{figure_intro}  for a graphical illustration of the architecture of the ANN $\Phi$ and its realization function $\mathcal{R}(\Phi)$. The functions 	$ f_d \colon \R^d \to \R^d $, $d \in \N $, describe the drift coefficient functions and the functions 	$ g_d \colon \R^d \to \R$, $d \in \N $,  describe the initial value functions  of the PDEs whose solutions we intend to approximate in \Cref{thm:DNNerrorEstimateLaplace} (see \eqref{eq:intro:PDE} above).

\def\layersep{2.5cm}
\begin{figure}[h]
	\centering
	\begin{adjustbox}{width=\textwidth}
		\begin{tikzpicture}[shorten >=1pt,->,draw=black!50, node distance=\layersep]
			\tikzstyle{every pin edge}=[<-,shorten <=1pt]
			\tikzstyle{output neuron}=[very thick, circle,draw=black, fill=red!30, minimum size=40pt,inner sep=0pt]
			\tikzstyle{input neuron}=[very thick, circle, draw=black,fill=-red!80,minimum size=40pt,inner sep=0pt]
			\tikzstyle{hidden neuron}=[very thick, circle,draw=black,fill={rgb:black,1;white,5},minimum size=30pt,inner sep=0pt]
			\tikzstyle{annot} = [text width=9em, text centered]
			\tikzstyle{annot2} = [text width=4em, text centered]
			
			\node[input neuron] (I-1) at (-3,-0.cm) {\large $1$};
			\node[input neuron] (I-2) at (-3,-2.7cm) {\large $2$};
			\node(I-dots) at (-3,-4.2cm) {\vdots};
			\node[input neuron] (I-3) at (-3,-5.8cm) {\large $l_0$};

			\path[yshift = 2cm]
			node[hidden neuron](H0-1) at (0*\layersep, -1 cm) { $1$};
			\path[yshift = 2cm]
			node[hidden neuron](H0-2) at (0*\layersep, -3.5 cm) { $2$};
			\path[yshift = 2cm]
			node[hidden neuron](H0-3) at (0*\layersep, -6 cm) { $3$};
			\path[yshift = 2cm]
			node(H0-dots) at (0*\layersep, -7.4 cm) {\vdots};
			\path[yshift = 2cm]
			node[hidden neuron](H0-4) at (0*\layersep, -9 cm) { $l_1$};
%
			\path[yshift = 2cm]
			node[hidden neuron](H1-1) at (1.5*\layersep, -1 cm) {$1$};
			\path[yshift = 2cm]
			node[hidden neuron](H1-2) at (1.5*\layersep, -3.5 cm) {$2$};
			\path[yshift = 2cm]
			node[hidden neuron](H1-3) at (1.5*\layersep, -6 cm) {$3$};
			\path[yshift = 2cm]
			node(H1-dots) at (1.5*\layersep, -7.4 cm) {\vdots};
			\path[yshift = 2cm]
			node[hidden neuron](H1-4) at (1.5*\layersep, -9 cm) {$l_2$};
			
	\path[yshift = 1cm]
	node(Hdot-1) at (2.7*\layersep, -1 cm) {$\cdots$};
	\path[yshift = 0.5cm]
	node(Hdot-2) at (2.7*\layersep, -3.4 cm) {$\cdots$};
	\path[yshift = 0.5cm]
	node(Hdot-dots) at (2.7*\layersep, -4.7 cm) {$\vdots$};
	\path[yshift = 0.5cm]
	node(Hdot-3) at (2.7*\layersep, -6.5 cm) {$\cdots$};
			
			\path[yshift = 2cm]
			node[hidden neuron](H2-1) at (3.9*\layersep, -1 cm) {$1$};
			\path[yshift = 2cm]
			node[hidden neuron](H2-2) at (3.9*\layersep, -3.5 cm) {$2$};
			\path[yshift = 2cm]
			node[hidden neuron](H2-3) at (3.9*\layersep, -6 cm) {$3$};
			\path[yshift = 2cm]
			node(H2-dots) at (3.9*\layersep, -7.4 cm) {\vdots};
			\path[yshift = 2cm]
			node[hidden neuron](H2-4) at (3.9*\layersep, -9 cm) {$l_{L-1}$};

			\path[yshift = 1.5cm]
			node[output neuron](O-1) at (5.4*\layersep,-1.5 cm) {\large $1$}; 
			\path[yshift = 1.5cm]
			node[output neuron](O-2) at (5.4*\layersep,-4.2 cm) {\large $2$}; 
			\path[yshift = 1.5cm]
			node(O-dots) at (5.4*\layersep, -5.7 cm) {\vdots};
			\path[yshift = 1.5cm]
			node[output neuron](O-3) at (5.4*\layersep,-7.2 cm) {\large $l_L$};
			\foreach \source in {1,2,3}
			\foreach \dest in {1,2,3,4}
			\path[-{Latex[length=2mm, width=4mm]}, line width = 0.8, draw=black] (I-\source) -- (H0-\dest);

			\foreach \source in {1,2,3,4}
			\foreach \dest in {1,2,3,4}
			\path[-{Latex[length=2mm, width=4mm]}, line width = 0.8, draw=black] (H0-\source) -- (H1-\dest);

			\foreach \source in {1,2,3,4}
			\foreach \dest in {1,2,3}
			\draw[-{Latex[length=2mm, width=4mm]}, line width = 0.8, draw=black] (H1-\source) -- (Hdot-\dest);
			
			\foreach \source in {1,2,3}
			\foreach \dest in {1,2,3,4}
			\draw[-{Latex[length=2mm, width=4mm]}, line width = 0.8, draw=black] (Hdot-\source) -- (H2-\dest);
			
			\foreach \source in {1,2,3,4}
			\foreach \dest in {1,2,3}
			\path[-{Latex[length=2mm, width=4mm]}, line width = 0.8, draw=black] (H2-\source) -- (O-\dest);

			\node[annot,above of=H0-1, node distance=1.4cm, align=center] (hl) {$1^{\text{st}}$ hidden layer };
			\node[annot,above of=H1-1, node distance=1.4cm, align=center] (hl) {$2^{\text{nd}}$ hidden layer};
			\node[annot,above of=H2-1, node distance=1.4cm, align=center] (hl2) {${(L - 1)}^{\text{th}}$ hidden layer};
			\node[annot,above of=I-1, node distance=1.2cm, align=center] {Input layer};
			\node[annot,above of=O-1, node distance=1.4cm, align=center] {Output layer};
			
			\node[annot,below of=H0-4, node distance=1.4cm, align=center] (sl) {$x_1 = A(W_1 x_0$\\ 
				$ +B_1) \in \R^{l_1}$};
			\node[annot,below of=H1-4, node distance=1.4cm, align=center] (sl) { $x_2= A(W_2 x_1 $\\
				$+B_2) \in \R^{l_2}$};
			\node[annot,below of=H2-4, node distance=1.4cm, align=center] (sl2) {$x_{L-1}= A(W_{L-1} x_{L-2}$\\$ +B_{L-1}) \in \R^{l_{L-1}}$};
			\node[annot2,below of=I-3, node distance=1.2cm, align=center] {$x_0 \in \R^{l_0}$};
			\node[annot,below of=O-3, node distance=1.7cm, align=center] {$(\mathcal{R}(\Phi)) (x_0)$ \\$= W_L x_{L-1}$ \\$+ B_L
				\in \R^{l_L}$};
		\end{tikzpicture}
	\end{adjustbox}
	\caption{Graphical illustration for the realization function and the architecture of an ANN $ \Phi=
		((W_1, B_1),\allowbreak \ldots, (W_L,\allowbreak B_L))
		\in  \allowbreak
		( \times_{k = 1}^L\allowbreak(\R^{l_k \times l_{k-1}} \times \R^{l_k})) \subseteq \ANNs
		$ (see \Cref{thm:DNNerrorEstimateLaplace}) where $ L\in\N$ describes the number of affine linear transformations, where $l_0,l_1,\ldots, l_L \in \N$ describe the dimensions of the layers of the ANN, and where the function $A  \colon (\cup_{d \in \N} \R^d) \to (\cup_{d \in \N} \R^d)$ represents the activation function (see \eqref{eq:intro:activ}).}
	\label{figure_intro}
\end{figure}

The real number $T \in (0, \infty)$ denotes the time horizon of the PDEs whose solutions we intend to approximate.
The real number $\kappa \in (0, \infty)$  is a constant which we employ to formulate the assumptions on the drift coefficient functions $ f_d \colon \R^d \to \R^d $, $d \in \N $,  and the initial value functions 	$ g_d \colon \R^d \to \R$, $d \in \N $,   of the PDEs whose solutions we intend to approximate in \Cref{thm:DNNerrorEstimateLaplace} (see \eqref{eq:intro:ass:1}--\eqref{eq:intro:ass:last} above). The real number $p \in (0, \infty)$ is used to describe the way how we measure the error between the exact solutions of the PDEs in~\eqref{eq:intro:PDE} and the corresponding deep ANN approximations in the sense that we measure the error in the strong  $L^p$-sense (see~\eqref{eq:intro:result} above).

 We assume in \Cref{thm:DNNerrorEstimateLaplace}  that the drift coefficient functions $ f_d \colon \R^d \to \R^d $, $d \in \N $,  and the initial value functions 	$ g_d \colon \R^d \to \R$, $d \in \N $, of the PDEs whose solutions we intend to approximate 
can be approximated by ANNs  
without the curse of dimensionality
 (see \eqref{eq:intro:ass:1} and \eqref{thm_intro:diff} above). We note that according to \emph{the universal approximation type theorems}  for every $d \in \N$ and every compact set $K \subseteq \R^d$  one can uniformly approximate the functions $ f_d \colon \R^d \to \R^d $ and 	$ g_d \colon \R^d \to \R$ on the set $K$ through ANNs with  an arbitrary prescribed  precision $\varepsilon >0$ (see, e.g., Kidger \& Lyons~\cite[Theorem~3.2]{KidgerLyons2020}). However, the universal approximation type theorems do not guarantee that the number of  parameters of the approximating ANNs  grows at most polynomially in both the PDE dimension $d \in \N$ and the reciprocal of the prescribed approximation accuracy $\varepsilon >0$, i.e., the universal approximation type theorems do not guarantee that
 the approximating ANNs do not suffer from the curse of dimensionality.
 
The functions  $u_d\colon [0,T] \times \R^{d} \to \R$, $d\in\N$,  in \Cref{thm:DNNerrorEstimateLaplace}  denote the PDE solutions which we intend to approximate by means of deep ANNs.  In particular, in \eqref{eq:intro:result} in \Cref{thm:DNNerrorEstimateLaplace} we show that there exists a constant $c \in \R$ such that for any dimension $d \in \N$ and any approximation accuracy $\varepsilon \in (0, 1]$ there exists an  ANN $\mathfrak{u} \in  \ANNs$ such that the number of parameters $\paramANN(\mathfrak{u})$ of the ANN is bounded by  $c \varepsilon^{-c}  d^{c}$ and such that the	 realization function $\mathcal{R}(\mathfrak{u}) \colon \R^{d+1} \to \R$ of the ANN approximates the PDE solution $u_d \colon [0,T] \times \R^{d} \to \R$  in the $L^p([0, T] \times [0,1]^d; \R)$-sense  with the precision $\varepsilon$.  

Our proofs of  \Cref{thm:DNNerrorEstimateLaplace} above and \Cref{theorem:DNNerrorEstimate} below, respectively, are based on an application of Proposition~3.10 in Grohs et al.~\cite{GrohsHornungEtAl2023} (see   \eqref{ProofApproxOfEulerWithGronwall:Function}--\eqref{ProofApproxOfEulerWithGronwall:Adaptedness} in the proof of \Cref{Thm:ApproxOfEulerWithGronwall} in Subsection~\ref{subsec:DNN:MCE} below for details). More specifically, Proposition~3.10 in Grohs et al.~\cite{GrohsHornungEtAl2023}   allows us to obtain space-time deep ANN approximations for Euler approximations of deterministic equations and Monte-Carlo Euler approximations of stochastic differential equations (SDEs) with desired complexity bounds. 
Combing this approximation result with the famous \emph{Feynman--Kac theorem} (see \Cref{thm:feynman} below for a special case) and the approximation error estimates for  Monte Carlo Euler approximations in \Cref{prop:PDE_approx_Lp} below
 enables us construct space-time deep ANN approximations of certain Kolmogorov PDEs with desired approximation capabilities.

In the following we present concrete  examples of PDEs whose coefficient functions  satisfy the assumptions of \Cref{thm:DNNerrorEstimateLaplace} above. For further families of coefficient functions satisfying such kind of regularity assumptions we refer to the arguments, e.g., in Bach~\cite{bach2017breaking}, E \& Wang \cite{EWang2018}, Cheridito et al.~\cite{CheriditoEtAl2022}, Beneventano et al.~\cite{BeneventanoEtAl2021}, and the references  therein.
\begin{example}
For every $d \in \N$ let $f_d \colon \R^d \to \R^d$ and $g_d \colon \R^d \to \R$ satisfy for all $x = (x_1, \ldots, x_d) \in \R^d$ that
\begin{align}
	\label{eq:example:intro:1}
	f_d(x) =
	\begin{cases}
		|x_1| &\colon \quad d=1\\
	(|x_d|, |x_1|, |x_2|, \ldots, |x_{d-1}|) &\colon \quad  d >1
	\end{cases} 
\end{align}
and
\begin{align}
	\label{eq:example:intro:2}
 g_d(x) = \max\{|x_1|, |x_2|, \ldots, |x_d|\}.
\end{align}
Note that \eqref{eq:example:intro:1} and \eqref{eq:example:intro:2}  ensure that for all $d \in \N$, $x, y \in \R^d$ it holds that
\begin{align}
\|f_d(x) - f_d(y)\| \leq \|x-y\|, \qquad |g_d(x)| \leq \|x\|, \qquad  \text{and} \qquad |g_d(x) - g_d(y)| \leq \|x-y\|
\end{align}
(cf., e.g., Beneventano et al.~\cite[Lemma~5.3]{BeneventanoEtAl2021}).
\Cref{thm:feynman} in Section~\ref{sec:feynman:kac} below therefore implies that for every $d \in \N$ 	there exists a unique viscosity solution
$u_d \in  \{v \in C([0, T] \times \R^d, \R) \colon \allowbreak \inf_{ q \in (0,\infty) }
\allowbreak \sup_{ (t,x) \in [0,T] \times \R^d } \allowbreak
\frac{ | v(t,x) | }{
	1 + \| x \|^q
}
< \infty \}$ of
\begin{equation}
	\begin{split}
		( \tfrac{ \partial }{\partial t} u_d )( t, x ) 
	& = 
	(\Delta_x u_d )( t, x ) +
	( \tfrac{ \partial }{\partial x} u_d )( t, x )
	\,
	f_d( x )
	\end{split}
\end{equation}
with  $ u_d( 0, x ) = g_d( x )$
for $ ( t, x ) \in (0,T) \times \R^d $.
Moreover, we observe that the families $(f_d)_{d \in \N}$ and $(g_d)_{d \in \N}$ can  be exactly represented as realizations of  ANNs with  the number of parameters growing at most quadratically in the input dimension. More specifically, using the notation of  \Cref{thm:DNNerrorEstimateLaplace} we note that there exists $\kappa \in (0, \infty)$ such that for all $d \in \N$ there exist $ \mathfrak{f},  \mathfrak{g} \in \ANNs$ such that
\begin{align}
\mathcal{P}(\mathfrak{f}) + \mathcal{P}(\mathfrak{g}) \leq \kappa d^2, \qquad \mathcal{R}(\mathfrak{f}) = f_d, \qquad \text{and} \qquad \mathcal{R}(\mathfrak{g}) = g_d
\end{align}
(cf., e.g., Beneventano et al.~\cite[Proposition 5.4]{BeneventanoEtAl2021}).
\Cref{thm:DNNerrorEstimateLaplace}  hence proves that  there exist $c\in \R$ such that
for all 
$d\in\N$, $\varepsilon\in (0,1]$  there exists $\mathfrak{u} \in \ANNs$ such
that
$\mathcal{R}(\mathfrak{u})\in C(\R^{d+1},\R)$,
$\paramANN(\mathfrak{u})\le c \varepsilon^{-c}  d^{c}$, and 
\begin{equation}
	\begin{split}
		& \left[
		\int_{[0,T]\times[0,1]^d}
		\left\vert
		u_d(y) 
		- 
		(\mathcal{R}(\mathfrak{u}))(y)
		\right\vert^p
		d y
		\right]^{\!\nicefrac{1}{p}} 
		\le
		\varepsilon.
	\end{split}
\end{equation}
\end{example}

\Cref{thm:DNNerrorEstimateLaplace}  is a purely theoretical result which asserts the existence of ANNs that can approximate the solutions of the PDEs without the curse of dimensionality. However, the proof of \Cref{thm:DNNerrorEstimateLaplace}  and the earlier results, e.g., in   \cite{GrohsHornungEtAl2023,JentzenSalimovaWelti2018, GrohsWurstemberger2023}  on which this work is partially based on,  respectively, in some way suggest a concrete class of algorithms, specifically,  a concrete class of ANN architectures with which the PDEs could be solved numerically and in this regard we refer to Becker et al.~\cite{BeckerEtAl2024} for details and concrete numerical simulations. For standard feedforward fully connected  ANN architectures used to approximatively solve Kolmogorov PDEs we refer, e.g., to Beck et al.~\cite{BeckBeckerEtAl2021} and Berner et al.\ \cite{BernerDablanderGrohs2020}.

The remainder of this article is structured in the following way. In Section~\ref{sec:SDEs}
we establish in \Cref{prop:perturbation_PDE_2} and \Cref{lem:Euler} suitable weak and strong error estimates for  Euler--Maruyama approximations for a certain class of SDEs.  In Section~\ref{sec:PDEs}  we use these weak and strong error estimates for  Euler--Maruyama approximations to establish in \Cref{prop:PDE_approx_Lp} below suitable error estimates for Monte Carlo Euler approximations for a class of SDEs with perturbed drift coefficient functions. In  Section~\ref{sec:DNN_PDEs} we use these error estimates for Monte Carlo Euler approximations to establish in \Cref{theorem:DNNerrorEstimate} below that for every $T \in (0, \infty)$, $a \in \R$, $ b \in (a, \infty)$ it holds that solutions of  certain Kolmogorov PDEs can be approximated by deep ANNs on the space-time region $[0, T] \times [a, b]^d$ without the curse of dimensionality.

\section{Numerical approximations for stochastic differential equations (SDEs)}	
\label{sec:SDEs}

In this section
we establish in \Cref{prop:perturbation_PDE_2} and \Cref{lem:Euler} below suitable weak and strong error estimates for  Euler--Maruyama approximations for a certain class of SDEs (see, e.g., Kloeden \& Platen~\cite{kp92} for an extensive introduction  to  numerical approximations for SDEs). Our proofs of  \Cref{prop:perturbation_PDE_2} and \Cref{lem:Euler}  are based on the elementary a priori moment estimates in Lemmas~\ref{lem:momentsGauss}--\ref{lem:sde-lp-bound} below.
\Cref{lem:momentsGauss}  is, e.g.,  proved as Gonon et al.~\cite[Lemma~3.1]{GononGrohsEtAl2022} (see also, e.g.,   Jentzen et al.~\cite[Lemma 4.2]{JentzenSalimovaWelti2018}) and a slightly modified version of
 \Cref{lem:sde-lp-bound} 
 is, e.g., proved as  Gonon et al.~\cite[Lemma~3.4]{GononGrohsEtAl2022} (see also, e.g., Jentzen et al.~\cite[Lemma 4.1]{JentzenSalimovaWelti2018}).

\subsection{A priori moment bounds for Gaussian random variables}

\input{MomentBounds}
	\subsection{Weak error  estimates for Euler--Maruyama approximations}
		\input{SDEapproximation}
	\subsection{Strong error  estimates for linearly interpolated Euler--Maruyama approximations}

\input{InterpolationEulerMethod}

\section{Numerical approximations for partial differential equations (PDEs) }
\label{sec:PDEs}

  In this section  we use the weak and strong error estimates which we have presented 
   in \Cref{prop:perturbation_PDE_2} and \Cref{lem:Euler} in Section~\ref{sec:SDEs} above
  to establish in \Cref{prop:PDE_approx_Lp} below suitable error estimates for Monte Carlo Euler approximations for a class of SDEs with perturbed drift coefficient functions.

 Besides \Cref{prop:perturbation_PDE_2} and \Cref{lem:Euler},
our proof of  \Cref{prop:PDE_approx_Lp} also  employs a special case of the famous
Feynman--Kac formula, which provides 
a connection between
solutions of SDEs and solutions of deterministic Kolmogorov PDEs.
For completeness we briefly recall in \Cref{thm:feynman} below this special case of the
Feynman--Kac formula.
\Cref{thm:feynman} is well-known in the literature, cf., e.g., Hairer et al.~\cite[Subsection 4.4]{HairerHutzenthalerJentzenAOP}, Jentzen et al.~\cite[Theorem~3.1]{JentzenSalimovaWelti2018}, and Beck et al.~\cite[Theorem~1.1]{BeckHutzenthalerJentzen2021}.

\input{Feynman}

	\subsection{Approximation error estimates for  Monte Carlo Euler approximations}

\input{errorEstimates}

\section{Deep artificial neural network (ANN) approximations for PDEs}
\label{sec:DNN_PDEs}

In this section we  establish in \Cref{theorem:DNNerrorEstimate} in Subsection~\ref{sub:cost} below the main result of this article. 

 \Cref{theorem:DNNerrorEstimate}, in particular,
proves that for every $T \in (0, \infty)$, $a \in \R$, $ b \in (a, \infty)$  it holds that solutions of  certain Kolmogorov PDEs can be approximated by deep ANNs on the space-time region $[0, T] \times [a, b]^d$ without the curse of dimensionality.
 In our proof of \Cref{theorem:DNNerrorEstimate} we employ the auxiliary intermediate result in \Cref{prop:DNNerrorEstimate}
 in Subsection~\ref{subsec:Appr_DNNs} below. Our proof of \Cref{prop:DNNerrorEstimate},
 in turn, uses the error estimates for Monte Carlo Euler approximations which we have presented in 
 \Cref{prop:PDE_approx_Lp}  in Section~\ref{sec:PDEs} above as well as the ANN approximation result for Monte Carlo Euler approximations in 
\Cref{Cor:MCeuler} in
  Subsection~\ref{subsec:DNN:MCE} below.

  Our proof of \Cref{Cor:MCeuler}  employs the auxiliary results in \Cref{Cor:ApproxOfMCSum} and \Cref{Lemma:LpBoundsGronwall} in
  Subsection~\ref{subsec:DNN:MCE} below. Our proof of \Cref{Cor:ApproxOfMCSum}, in turn, uses the ANN approximation result for Monte Carlo Euler approximations in \Cref{Thm:ApproxOfEulerWithGronwall}  in
  Subsection~\ref{subsec:DNN:MCE} below.
  Our proof of \Cref{Thm:ApproxOfEulerWithGronwall}  is  based on an application of  \cite[ Proposition~3.10]{GrohsHornungEtAl2023}  and is  very similar to the proof of  \cite[Theorem~3.12]{GrohsHornungEtAl2023}.

Our proof of \Cref{theorem:DNNerrorEstimate} in Subsection~\ref{sub:cost} below  also employs several  well-known concepts and results from an appropriate calculus for ANNs from the scientific literature which we briefly recall in Subsections~\ref{subsec:DNNs}--\ref{subsec:DNN:composition} below.
In particular, \Cref{Def:ANN} is, e.g., \cite[Definition~2.1]{GrohsHornungEtAl2023},
\Cref{Def:multidim_version} is, e.g., \cite[Definition~2.2]{GrohsHornungEtAl2023},
\Cref{Definition:ANNrealization} is, e.g., \cite[Definition~2.3]{GrohsHornungEtAl2023},
\Cref{Definition:ANNcomposition} is, e.g., \cite[Definition~2.5]{GrohsHornungEtAl2023},
Lemma~\ref{Lemma:CompositionAssociative} is, e.g.,  \cite[Lemma 2.8]{GrohsHornungEtAl2023},
and
\Cref{Definition:ANNconcatenation} is, e.g., \cite[Definition~2.15]{GrohsHornungEtAl2023}.

\input{ANNdefinitions}

\input{DNNmonteCarloEuler}

\input{DNNapproximationERRORnew}

\input{DNNcostEstimates}

\subsection*{Acknowledgments}

Philipp Grohs is gratefully acknowledged for several useful comments.
This work has been partially funded by  the Swiss National Science Foundation (SNSF) through the research grant  \allowbreak 200020\_175699,
by the Deutsche Forschungsgemeinschaft (DFG, German Research Foundation)    through CRC 1173,  by the Karlsruhe House of Young Scientists (KHYS) through a research travel grant,  by  ETH Foundations of Data Science (ETH - FDS), and by the European Union (ERC, MONTECARLO, 101045811).
The views and the opinions expressed in this work are however those of the authors only
and do not necessarily reflect those of the European Union or the European Research Council.
Neither the European Union nor the granting authority can be held responsible for them. In addition,
the second author gratefully acknowledges  the Deutsche Forschungsgemeinschaft (DFG, German Research Foundation) under Germany's Excellence Strategy EXC 2044-390685587, Mathematics M\"unster: Dynamics-Geometry-Structure.

\bibliographystyle{acm}
\bibliography{bibfile.bib}

\end{document}

%% file: MomentBounds.tex
\begin{definition}[Standard norms]
	\label{Def:euclideanNorm}
	We denote by $\norm{\cdot} \colon (\cup_{d\in\N} \R^d) \to [0,\infty)$ the function which satisfies for all $d\in\N$, $ x = ( x_1, x_2,\dots, x_{d} ) \in \R^{d} $ that
	\begin{equation}\label{euclideanNorm:Equation}
\norm{x}=\big[\smallsum\nolimits_{j=1}^d \vert x_j\vert^2\big]^{\nicefrac{1}{2}}.
	\end{equation}
\end{definition}

\begin{lemma}
	\label{lem:momentsGauss}
	Let
	$ d \in \N$,
	$ p \in (0,\infty)$,
	let
	$
	( \Omega, \mathcal{F}, \P ) 
	$
	be a probability space,
	and let
	$
	X \colon \Omega \to \R^d
	$
	be a centered Gaussian random variable. 
	Then 
	\begin{equation}
	\label{eq:bm-lp}
	\begin{split}
	\big( \E\big[ \| X \|^p \big] \big)^{ \nicefrac{1}{p} }
	& \leq 
	\sqrt{ \max\{1,p-1\} \operatorname{Trace}(\cov(X))} 
	\end{split}
	\end{equation}
	(cf.\ Definition~\ref{Def:euclideanNorm}).
\end{lemma}

\subsection{A priori moment bounds for solutions of SDEs}

\begin{lemma}
	\label{lem:sde-lp-bound}
	Let $ d \in \N $, $ \xi \in \R^d $, $ p \in [1,\infty) $, 
	$ c, C, T \in [0,\infty) $, 
	let $ (\Omega, \mathcal{F},\P) $ be a probability space, 
	let $\mu \colon \R^d \to \R^d$ and $\chi \colon  [0,T] \to [0,T]$
	be measurable functions, assume for all $x \in \R^d$, $t \in
	[0,T]$ that
	$\|\mu(x)\| \leq C+c\|x\|$ and $\chi(t) \leq t$, 
	and let $X, \beta  \colon [0,T]\times \Omega \to \R^d$ be  stochastic
	processes with continuous sample paths which satisfy for all 
	$ t \in [0,T] $ that
	\begin{equation}
	\label{eq:apriori1-ass1}
	\P\!\left(X_t = \xi + \int_0^t \mu\!\left(X_{\chi(s)}\right) ds +
	\beta_t\right) = 1
	\end{equation}
	(cf.\ Definition~\ref{Def:euclideanNorm}).
	Then 
	\begin{equation}
	\label{eq:lemma:apriori}
	\begin{split}
	\sup_{t \in [0, T]} \big(\E\big [\|X_t\|^p\big]\big)^{\nicefrac{1}{p}} 
	& \leq
	\Big(
	\|\xi\| + C T 
	+ 
	\sup_{t\in[0,T]}\big(
	\E\big[ 
	\| \beta_t \|^p
	\big]
	\big)^{ \nicefrac{ 1 }{ p } }
	\Big)
	\,
	e^{ c T } .
	\end{split}
	\end{equation}
\end{lemma}

%% file: SDEapproximation.tex
\begin{prop}
	\label{prop:perturbation_PDE_2}
	Let $ d, m \in \N $, $ \xi \in \R^d $, 
	$ T \in (0,\infty) $,
	$ c, C, \varepsilon_0, \varepsilon_1, \varepsilon_2, \varsigma_0, \varsigma_1, \varsigma_2, L_0, L_1, $ $\ell \in [0,\infty) $, 
	$ h \in [0,T] $,
	$p \in [2, \infty)$,
	$q \in (1, 2]$
	satisfy
	$
	\nicefrac{ 1 }{ p } + \nicefrac{ 1 }{ q } = 1
	$,
		let 	$ B \in \R^{ d \times m } $, $ (\varpi_r)_{r \in (0,\infty)} \subseteq \R $
		satisfy for all $ r \in (0,\infty) $ that
		\begin{equation}\label{perturbation_PDE_2:varpi}
					\varpi_r = 
					\max\!\big\{1,\sqrt{\max\{1,r-1\}  \operatorname{Trace}(B^* B)}\big\},
		\end{equation}
	let $ ( \Omega, \mathcal{F}, \P ) $ be a probability space, 
	let $ W \colon [0,T] \times \Omega \to \R^m $ be a standard Brownian motion, 
	let 
	$ \dnnFunction_0 \colon \R^d \to \R $, 
	$ f_1 \colon \R^d \to \R^d $, 
	$ \dnnFunction_2 \colon \R^d \to \R^d $, 
	and
	$ \chi \colon [0,T] \to [0,T] $ be functions,
	let 
	$ f_0 \colon \R^d \to \R $
and
	$ \dnnFunction_1 \colon \R^d \to \R^d $
	be measurable functions, 
	assume for all 
	$ t \in [0,T] $, $ x, y \in \R^d $ that
	\begin{equation}
	|  f_0( x ) - \dnnFunction_0( x ) |
	\leq 
	\varepsilon_0 
	( 1 + \| x \|^{ \varsigma_0 } )
	,
	\qquad
	\|  f_1( x ) - \dnnFunction_1( x ) \|
	\leq 
	\varepsilon_1
	( 1 + \| x \|^{ \varsigma_1 } )
	,
	\end{equation}
	\begin{equation}
	| \dnnFunction_0( x ) - \dnnFunction_0( y ) | 
	\leq 
	L_0
	\left[
	1 
	+ 
	\int_0^1
	\big[
	r \| x \| + ( 1 - r ) \| y \|
	\big]^{ \ell }
	\,
	dr
	\right]
	\| x - y \| ,
	\end{equation}
	\begin{equation}
	\label{eq:Prop:weak:f1}
	\| f_1( x ) - f_1( y ) \| \leq L_1 \| x - y \| 
	,
	\quad
	\| \dnnFunction_1( x ) \|
	\leq 
	C + c \| x \|, \quad 	\| \xi - \dnnFunction_2( \xi ) \|
	\leq
	\varepsilon_2 ( 1 + \| \xi \|^{ \varsigma_2 } ),
	\end{equation}
	and 
	$\chi( t ) 
	= 
	\max\!\left(
	\{ 0, h, 2 h , \dots \}
	\cap [0,t] 
	\right)$,
	and 
	let $ X, Y \colon [0,T] \times \Omega \to \R^d $ be stochastic processes with 
	continuous sample paths which satisfy
	for all $ t \in [0,T] $ that 
	\begin{equation}
	X_t = \xi + \int_0^t f_1( X_s ) \, ds + B W_t \qandq Y_t = \dnnFunction_2( \xi ) + \int_0^t \dnnFunction_1\big( Y_{ \chi( s ) } \big) \, ds + B W_t
	\end{equation}
	(cf.\ Definition~\ref{Def:euclideanNorm}).
	Then it holds for all $t\in[0,T]$ that
\begin{align*}
&
\big| \E\big[ f_0( X_t ) \big] - \E\big[ \dnnFunction_0( Y_t ) \big] \big|
\le 
\left(\varepsilon_2 ( 1 + \| \xi \|^{ \varsigma_2 } )+\varepsilon_0+\varepsilon_1  +h+h^{ \nicefrac{ 1 }{ 2 } }\right)
\\&\cdot e^{\left[\max\{\varsigma_0,1\}L_1+1 - \nicefrac{1}{p}+\ell \max\{L_1,c\}+\max\{\varsigma_1,1\} c\right]
	T}
(\varpi_{\max\{\varsigma_0,\ell q,p \varsigma_1, p\}})^{\max\{\varsigma_0,\ell+\max\{1,\varsigma_1\}\}} \numberthis \\&\cdot(\max\{T,1\})^{\max\{\varsigma_0,\ell+\max\{\varsigma_1,1\}+ \nicefrac{ 1 }{ p }\}}
\max\{L_0,1\}\, \max\{L_1,1\}\,
2^{ \max\{ \ell - 1, 0 \} } 
\\&\cdot
\left[
\max\{C,1\}+5\max\{C,c,1\}
\big(
\| \xi\|+
\varepsilon_2 ( 1 + \| \xi \|^{ \varsigma_2 } )
+ 
2 \max\{\| f_1(0) \| ,C,1\}\big)^{\max\{\varsigma_0,\ell+\max\{\varsigma_1,1\}\}}
\right]\!.
\end{align*}
\end{prop}

\begin{proof}[Proof of Proposition~\ref{prop:perturbation_PDE_2}]

Observe that \eqref{perturbation_PDE_2:varpi}, 
Lemma~\ref{lem:momentsGauss},  the fact that for all $t\in[0,T]$ it holds that $BW_t$ is a centered Gaussian random variable, and the fact   that for all $t\in[0,T]$ it holds that  $\cov(BW_t)=B B^\ast\, t$
assure that for all $r \in (0, \infty)$, $t\in[0,T]$ it holds that
\begin{equation}\label{perturbation_SDE:BrownianMoments}
\begin{split}
\big(	\E\big[
	\| 
	B W_t 
	\|^r
	\big]
	\big)^{ \nicefrac{ 1 }{ r } }
	&\le 
	\sqrt{ \max\{1,r-1\} \operatorname{Trace}(\cov(BW_t))}
	\\& = 
	\sqrt{ \max\{1,r-1\} \operatorname{Trace}(B B^\ast) t}
	\\&\le \max\!\big\{t^{\nicefrac{1}{2}},\sqrt{\max\{1,r-1\}  \operatorname{Trace}(B^* B)t}\big\}
	= \varpi_r t^{\nicefrac{1}{2}}.
\end{split}
\end{equation}
In addition, note that \eqref{eq:Prop:weak:f1} shows that for all $ x \in \R^d $ it holds that
\begin{equation}
\label{eq:f1_linear_growth}
\left\| f_1( x ) \right\|
\leq
\left\| f_1( x ) - f_1( 0 ) \right\|
+
\left\| f_1( 0 ) \right\|
\leq
\left\| f_1( 0 ) \right\|
+
L_1 \| x \|
.
\end{equation}
H\"older's inequality, \eqref{eq:Prop:weak:f1}, 
Lemma~\ref{lem:sde-lp-bound},
 and \eqref{perturbation_SDE:BrownianMoments}
hence demonstrate that 
for all $ r \in (0,\infty) $, $t\in [0,T]$ it holds that
\begin{equation}\label{perturbation_SDE:EulerMoments}
\begin{split}
\sup_{ s\in [0,t] }
\big(
\E\big[ 
\| Y_s \|^{r}
\big]
\big)^{ \nicefrac{ 1 }{r } }
&\le 
\sup_{ s\in [0,t] }
\big(
\E\big[ 
\| Y_s \|^{\max\{r,1\}}
\big]
\big)^{ \nicefrac{ 1 }{ \max\{r,1\} } }
\\&\leq
\bigg[
\| \dnnFunction_2( \xi ) \| 
+ 
C t 
+ \sup_{ s\in [0,t] }
\big( 
\E\big[ \| B W_s \|^{\max\{r,1\}} \big]
\big)^{ \nicefrac{ 1 }{ \max\{r,1\} } }
\bigg]
e^{ c t } 
\\ &\le
\left[
\| \dnnFunction_2( \xi ) \| 
+ 
C t 
+ 
\varpi_{ \max\{r,1\} } t^{\nicefrac{1}{2}}\right]
e^{ c t } 
\end{split}
\end{equation}
and 
\begin{equation}\label{perturbation_SDE:SDEMoments}
\begin{split}
\sup_{ s\in [0,t] }
\big(
\E\big[ 
\| X_s \|^{r}
\big]
\big)^{ \nicefrac{ 1 }{r } }
&\le 
\sup_{ s\in [0,t] }
\big(
\E\big[ 
\| X_s \|^{\max\{r,1\}}
\big]
\big)^{ \nicefrac{ 1 }{ \max\{r,1\} } }
\\& \leq
\bigg[
\| \xi \| 
+ 
\| f_1(0) \| t 
+  \sup_{ s\in [0,t] }
\big( 
\E\big[ \| B W_s \|^{\max\{r,1\}} \big]
\big)^{ \nicefrac{ 1 }{ \max\{r,1\} } }
\bigg]
e^{ L_1 t }
\\ & \le
\left[
\| \xi \| 
+ 
\| f_1(0) \| t 
+ 
\varpi_{ \max\{r,1\} } t^{\nicefrac{1}{2}}
\right]
e^{ L_1 t }
.
\end{split}
\end{equation}
H\"older's inequality, \eqref{eq:Prop:weak:f1}, and the triangle inequality therefore imply that for all $ r \in (0,\infty) $, $t \in [0, T]$ it holds that
\begin{equation}
\label{eq:phi_1:Y}
\begin{split}
\sup_{ s\in [0,t] }
\big( 
\E\big[
\|
\dnnFunction_1( Y_s )
\|^{r}
\big]
\big)^{ \nicefrac{ 1 }{ r } }
	&\le\sup_{ s\in [0,t] }
\big( 
\E\big[
\|
\dnnFunction_1( Y_s )
\|^{\max\{r,1\}}
\big]
\big)^{ \nicefrac{ 1 }{ \max\{r,1\} } }
\\&\le C+c	\sup_{ s\in [0,t] }
\big( 
\E\big[
\| Y_s \|^{\max\{r,1\}}
\big]
\big)^{ \nicefrac{ 1 }{ \max\{r,1\} } }
\\&\le C+c
\left(
\| \dnnFunction_2( \xi ) \| 
+ 
C t 
+ 
\varpi_{ \max\{r,1\} } t^{\nicefrac{1}{2}}\right)
e^{ c t }. 
\end{split}
\end{equation}
Moreover, note that \cite[Lemma~4.3]{JentzenSalimovaWelti2018} assures for all $t\in [0,T]$ that 
\begin{equation}
\begin{split}
&
\big| \E\big[ f_0( X_t ) \big] - \E\big[ \dnnFunction_0( Y_t ) \big] \big|
\\ & 
\leq
\varepsilon_0 
\left(
1 
+
\E\!\left[ 
\| X_t \|^{ \varsigma_0 }
\right]
\right)
+
L_0
\,
2^{ \max\{ \ell - 1, 0 \} }
\,
e^{
	\left[ 
	L_1 
	+
	1 - \nicefrac{1}{p} 
	\right]
	t
}
\left[
1 
+
\left(
\E\big[
\| X_t \|^{ \ell q }
\big]
\right)^{ \nicefrac{ 1 }{ q } }
+
\left(
\E\big[
\| Y_t \|^{ \ell q }
\big]
\right)^{ \nicefrac{ 1 }{ q } }
\right]
\\ &  \cdot 
\bigg[
\left\| \xi - \dnnFunction_2( \xi ) \right\|
+
\varepsilon_1 
t^{ \nicefrac{ 1 }{ p } }
\!
\bigg[
1
+
\sup_{s \in [0,t] }
\big(
\E\big[
\| Y_s \|^{ p \varsigma_1 }
\big]
\big)^{ \nicefrac{ 1 }{ p } }
\bigg] +
h
t^{ \nicefrac{ 1 }{ p } }
L_1 
\bigg[
\sup_{s \in [0,t] }
\big( 
\E\big[
\|
\dnnFunction_1( Y_s )
\|^p
\big]
\big)^{ \nicefrac{ 1 }{ p } }
\bigg]
\\ & 
+
t^{ \nicefrac{ 1 }{ p } }
L_1 
\big(
\E\big[
\| 
B W_h 
\|^p
\big]
\big)^{ \nicefrac{ 1 }{ p } }
\bigg].
\end{split}
\end{equation}
This, \eqref{perturbation_SDE:BrownianMoments}, \eqref{perturbation_SDE:EulerMoments},  \eqref{perturbation_SDE:SDEMoments}, and \eqref{eq:phi_1:Y} prove that for all $t\in [0,T]$ it holds that
\begin{equation}
\label{eq:diff:F0}
\begin{split}
	&
	\big| \E\big[ f_0( X_t ) \big] - \E\big[ \dnnFunction_0( Y_t ) \big] \big|
	\\ & 
	\leq
	\varepsilon_0 
	\left[
	1 
	+
	\left(
	\| \xi \| 
	+ 
	\| f_1(0) \| t 
	+ 
	\varpi_{\max\{\varsigma_0,1\}} t^{\nicefrac{1}{2}}
	\right)^{\varsigma_0}
	e^{ \varsigma_0 L_1 t }
	\right]
	+
	L_0
	\,
	2^{ \max\{ \ell - 1, 0 \} }
	\,
	e^{
		\left[ 
		L_1 
		+
		1 - \nicefrac{1}{p} 
		\right]
		t
	}
	\\&\cdot\left[
	1 
	+
	\left(
	\| \xi \| 
	+ 
	\| f_1(0) \| t 
	+ 
	\varpi_{ \max\{\ell q,1\}} t^{ \nicefrac{1}{2}}
	\right)^{\ell}
	e^{  \ell  L_1 t }
	+
	\left(
	\| \dnnFunction_2( \xi ) \| 
	+ 
	C t 
	+ 
	\varpi_{\max\{\ell q,1\}}   t^{ \nicefrac{1}{2}}
	\right)^\ell
	e^{ \ell c t } 
	\right]
	\\ & \cdot 
	\Big[
	\left\| \xi - \dnnFunction_2( \xi ) \right\|
	+
	\varepsilon_1 
	t^{  \nicefrac{ 1 }{ p } }
	\!
	\left[
	1
	+
	\left(
	\| \dnnFunction_2( \xi ) \| 
	+ 
	C t 
	+ 
	\varpi_{\max\{p \varsigma_1,1\}}   t^{ \nicefrac{1}{2}}
	\right)^{\varsigma_1}
	e^{ \varsigma_1 c t } 
	\right]
	\\ &
	+
	h
	t^{  \nicefrac{ 1 }{ p } }
	L_1 
	\!
	\left[
	C+c
	\left(
	\| \dnnFunction_2( \xi ) \| 
	+ 
	C t 
	+ 
	\varpi_p t^{ \nicefrac{1}{2}}\right)
	e^{ c t }
	\right]
	+
	h^{ \nicefrac{ 1 }{ 2 } }t^{  \nicefrac{ 1 }{ p } }
	L_1 
	\varpi_p
	\Big].
\end{split}
\end{equation}
In addition, observe that the fact that $(0, \infty) \ni r \mapsto \varpi_r \in (0, \infty)$ is non-decreasing and the hypothesis that $p \in [2, \infty)$ imply that 
\begin{equation}
\varpi_{\max\{\varsigma_0,\ell q,p \varsigma_1, p\}}\ge \max\{\varpi_{\max\{\varsigma_0,1\}}, \varpi_{ \max\{\ell q,1\}}, \varpi_{\max\{p \varsigma_1,1\}}, \varpi_p\}.
\end{equation}
This and \eqref{eq:diff:F0}  ensure that for all $t\in [0,T]$ it holds that
\begin{equation}
\begin{split}
&
\big| \E\big[ f_0( X_t ) \big] - \E\big[ \dnnFunction_0( Y_t ) \big] \big|
\\&
\leq
\varepsilon_0 e^{ \varsigma_0 L_1 T }
\left[
1 
+
\left(
\| \xi \| 
+ 
\| f_1(0) \| T 
+ 
\varpi_{\max\{\varsigma_0,\ell q,p \varsigma_1, p\}} T^{\nicefrac{1}{2}}
\right)^{\varsigma_0}
\right]
\\ & 
+
L_0
\,
2^{ \max\{ \ell - 1, 0 \} }
\,
e^{
	\left[ 
	L_1 
	+
	1 - \nicefrac{1}{p} 
	+\ell \max\{L_1,c\}+\max\{\varsigma_1,1\} c
	\right]
	T
}
\\&\cdot\left[
1 
+
\left(
\| \xi \| 
+ 
\| f_1(0) \| T 
+ 
\varpi_{\max\{\varsigma_0,\ell q,p \varsigma_1, p\}} T^{\nicefrac{1}{2}}
\right)^{\ell}
+
\left(
\| \dnnFunction_2( \xi ) \| 
+ 
C T 
+ 
\varpi_{\max\{\varsigma_0,\ell q,p \varsigma_1, p\}}   T^{\nicefrac{1}{2}}
\right)^\ell
\right]
\\ & \cdot 
\Big[
\left\| \xi - \dnnFunction_2( \xi ) \right\|
+
\varepsilon_1 
T^{ \nicefrac{ 1 }{ p } }
\!
\left[
1
+
\left(
\| \dnnFunction_2( \xi ) \| 
+ 
C T 
+ 
\varpi_{\max\{\varsigma_0,\ell q,p \varsigma_1, p\}}  T^{\nicefrac{1}{2}}
\right)^{\varsigma_1}
\right]  
\\ &
\quad+
h
T^{ \nicefrac{ 1 }{ p } }
L_1 
\left[
C+c
\left(
\| \dnnFunction_2( \xi ) \| 
+ 
C T 
+ 
\varpi_{\max\{\varsigma_0,\ell q,p \varsigma_1, p\}} T^{\nicefrac{1}{2}}\right)
\right]
+
h^{ \nicefrac{ 1 }{ 2 } }T^{ \nicefrac{ 1 }{ p } }
L_1 
\varpi_{\max\{\varsigma_0,\ell q,p \varsigma_1, p\}}
\Big].
\end{split}
\end{equation}
Combining this with the fact that $\varpi_{\max\{\varsigma_0,\ell q,p \varsigma_1, p\}}\ge 1$ and the fact that $T^{\nicefrac{1}{2}}\le (\max\{T,1\})^{\nicefrac{1}{2}}\le \max\{T,1\} $ assures that for all $t\in [0,T]$ it holds that
\begin{equation}
\begin{split}
&
\big| \E\big[ f_0( X_t ) \big] - \E\big[ \dnnFunction_0( Y_t ) \big] \big|
\\ & 
\leq
\varepsilon_0 e^{ \varsigma_0 L_1 T } (\varpi_{\max\{\varsigma_0,\ell q,p \varsigma_1, p\}})^{\varsigma_0}
\left[
1 
+
\left(
\| \xi \| 
+ 
\| f_1(0) \| T 
+ 
T^{\nicefrac{1}{2}}
\right)^{\varsigma_0}
\right]
\\ & 
+
L_0
\,
2^{ \max\{ \ell - 1, 0 \} }
\,
e^{
	\left[ 
	L_1 
	+
	1 - \nicefrac{1}{p} 
	+\ell \max\{L_1,c\}+\max\{\varsigma_1,1\} c
	\right]
	T
}
(\varpi_{\max\{\varsigma_0,\ell q,p \varsigma_1, p\}})^{\ell+\max\{1,\varsigma_1\}}
\\&\cdot\left[
1 
+
\left(
\| \xi \| 
+ 
\| f_1(0) \| T 
+ 
T^{\nicefrac{1}{2}}
\right)^{\ell}
+
\left(
\| \dnnFunction_2( \xi ) \| 
+ 
C T 
+ 
T^{\nicefrac{1}{2}}
\right)^\ell
\right]
\\ & \cdot 
\Big[
\left\| \xi - \dnnFunction_2( \xi ) \right\|
+
\varepsilon_1 
T^{ \nicefrac{ 1 }{ p } }
\!
\left[
1
+
\left(
\| \dnnFunction_2( \xi ) \| 
+ 
C T 
+ 
T^{\nicefrac{1}{2}}
\right)^{\varsigma_1}
\right]  
\\ &
+
h
T^{ \nicefrac{ 1 }{ p } }
L_1 
\left[
C+c
\left(
\| \dnnFunction_2( \xi ) \| 
+ 
C T 
+ 
T^{\nicefrac{1}{2}}\right)
\right]
+
h^{ \nicefrac{ 1 }{ 2 } }T^{ \nicefrac{ 1 }{ p } }
L_1 
\Big].
\end{split}
\end{equation}
Therefore, we obtain that for all $t\in [0,T]$ it holds that
\begin{equation}
\begin{split}
&
\big| \E\big[ f_0( X_t ) \big] - \E\big[ \dnnFunction_0( Y_t ) \big] \big|
\\ & 
\leq
\varepsilon_0 e^{\varsigma_0 L_1 T} (\varpi_{\max\{\varsigma_0,\ell q,p \varsigma_1, p\}})^{\varsigma_0}\left(
1 
+ (\max\{T,1\})^{\varsigma_0}
\left(
\| \xi \| 
+ 
\| f_1(0) \| 
+ 
1
\right)^{\varsigma_0}
\right)
\\ & 
+
L_0
\,
2^{ \max\{ \ell - 1, 0 \} }
\,
e^{
	\left[ 
	L_1 
	+
	1 - \nicefrac{1}{p} 
	+\ell \max\{L_1,c\}+\max\{\varsigma_1,1\} c
	\right]
	T
}
(\varpi_{\max\{\varsigma_0,\ell q,p \varsigma_1, p\}})^{\ell+\max\{1,\varsigma_1\}}
\\&\cdot\left[
1 
+
 (\max\{T,1\})^{\ell}\left[\left(
\| \xi \| 
+ 
\| f_1(0) \| 
+ 
1
\right)^{\ell}
+
\left(
\| \dnnFunction_2( \xi ) \| 
+ 
C  
+ 
1
\right)^\ell
\right]
\right]
\\ & \cdot 
\Big[
\left\| \xi - \dnnFunction_2( \xi ) \right\|
+
\varepsilon_1 
T^{ \nicefrac{ 1 }{ p } }
\!
\left[
1
+  (\max\{T,1\})^{\varsigma_1}
\left(
\| \dnnFunction_2( \xi ) \| 
+ 
C  
+ 1
\right)^{\varsigma_1}
\right]  
\\ &
+
h
T^{ \nicefrac{ 1 }{ p } } \max\{T,1\}
L_1 
\left[
C+c
\left(
\| \dnnFunction_2( \xi ) \| 
+ 
C  
+ 
1\right)
\right]
+
h^{ \nicefrac{ 1 }{ 2 } }T^{ \nicefrac{ 1 }{ p } }
L_1 
\Big].
\end{split}
\end{equation}
Hence, we obtain that for all $t\in [0,T]$ it holds that
\begin{equation}
\begin{split}
&
\big| \E\big[ f_0( X_t ) \big] - \E\big[ \dnnFunction_0( Y_t ) \big] \big|
\\ & 
\leq
\varepsilon_0 e^{\varsigma_0 L_1 T}(\varpi_{\max\{\varsigma_0,\ell q,p \varsigma_1, p\}})^{\varsigma_0}
(\max\{T,1\})^{\varsigma_0}
\left(
1 
+ 
\left(
\| \xi \| 
+ 
\| f_1(0) \| 
+ 
1
\right)^{\varsigma_0}
\right)
\\ & 
+
L_0
\,
2^{ \max\{ \ell - 1, 0 \} }
\,
e^{
	\left[ 
	L_1 
	+
	1 - \nicefrac{1}{p} 
	+\ell \max\{L_1,c\}+\max\{\varsigma_1,1\} c
	\right]
	T
}
(\varpi_{\max\{\varsigma_0,\ell q,p \varsigma_1, p\}})^{\ell+\max\{1,\varsigma_1\}}
\\&\cdot (\max\{T,1\})^{\ell+\max\{\varsigma_1,1\}+ \nicefrac{ 1 }{ p }} \left[
1 
+
\left(
\| \xi \| 
+ 
\| f_1(0) \| 
+ 
1
\right)^{\ell}
+
\left(
\| \dnnFunction_2( \xi ) \| 
+ 
C  
+ 
1
\right)^\ell
\right]
\\ & \cdot 
\Big[
\left\| \xi - \dnnFunction_2( \xi ) \right\|
+
\varepsilon_1  
\!
\left[
1
+  
\left(
\| \dnnFunction_2( \xi ) \| 
+ 
C  
+ 1
\right)^{\varsigma_1}
\right]  
+
h
L_1 
\left[
C+c
\left(
\| \dnnFunction_2( \xi ) \| 
+ 
C  
+ 
1\right)
\right]
+
h^{ \nicefrac{ 1 }{ 2 } }
L_1 
\Big].
\end{split}
\end{equation}
This implies that for all $t\in [0,T]$ it holds that
\begin{equation}
\begin{split}
&
\big| \E\big[ f_0( X_t ) \big] - \E\big[ \dnnFunction_0( Y_t ) \big] \big|
\\ & 
\leq
\varepsilon_0 e^{\varsigma_0 L_1 T} (\varpi_{\max\{\varsigma_0,\ell q,p \varsigma_1, p\}})^{\varsigma_0} (\max\{T,1\})^{\varsigma_0}
\\&\cdot\left(
1 
+ 
\big(
\max\{\| \dnnFunction_2( \xi ) \|,\|\xi\|\}
+ 
2 \,\max\{\| f_1(0) \| ,C,1\}
\big)^{\varsigma_0}
\right)
\\ & 
+
L_0
\,
2^{ \max\{ \ell - 1, 0 \} }
\,
e^{
	\left[ 
	L_1 
	+
	1 - \nicefrac{1}{p} 
	+\ell \max\{L_1,c\}+\max\{\varsigma_1,1\} c
	\right]
	T
}
(\varpi_{\max\{\varsigma_0,\ell q,p \varsigma_1, p\}})^{\ell+\max\{1,\varsigma_1\}}
\\&\cdot (\max\{T,1\})^{\ell+\max\{\varsigma_1,1\}+ \nicefrac{ 1 }{ p }} \left[
1 
+
2 \big(
\max\{\| \dnnFunction_2( \xi ) \|,\|\xi\|\}
+ 
2 \,\max\{\| f_1(0) \| ,C,1\}
\big)^{\ell}
\right]
\\ & \cdot 
\Big[
\left\| \xi - \dnnFunction_2( \xi ) \right\|
+
\varepsilon_1  
\!
\left[
1
+  
\left(
\max\{\| \dnnFunction_2( \xi ) \|,\|\xi\|\} 
+ 
2 \max\{\| f_1(0) \| ,C,1\}
\right)^{\varsigma_1}
\right]  
\\&+
h
L_1 
\left[
C+c
\left(
\max\{\| \dnnFunction_2( \xi ) \|,\|\xi\|\} 
+ 
2 \max\{\| f_1(0) \| ,C,1\}\right)
\right]
+
h^{ \nicefrac{ 1 }{ 2 } }
L_1 
\Big].
\end{split}
\end{equation}
Therefore, we obtain that for all $t\in [0,T]$ it holds that
\begin{equation}
\begin{split}
&
\big| \E\big[ f_0( X_t ) \big] - \E\big[ \dnnFunction_0( Y_t ) \big] \big|
\\ & 
\leq
\varepsilon_0 e^{\varsigma_0 L_1 T} (\varpi_{\max\{\varsigma_0,\ell q,p \varsigma_1, p\}})^{\varsigma_0} (\max\{T,1\})^{\varsigma_0}
\\&\cdot\left(
1 
+ 
\big(
\max\{\| \dnnFunction_2( \xi ) \|,\|\xi\|\}
+ 
2 \,\max\{\| f_1(0) \| ,C,1\}
\big)^{\varsigma_0}
\right)
\\ & 
+
L_0
\,
2^{ \max\{ \ell - 1, 0 \} }
\,
e^{
	\left[ 
	L_1 
	+
	1 - \nicefrac{1}{p} 
	+\ell \max\{L_1,c\}+\max\{\varsigma_1,1\} c
	\right]
	T
}
(\varpi_{\max\{\varsigma_0,\ell q,p \varsigma_1, p\}})^{\ell+\max\{1,\varsigma_1\}}
\\&\cdot (\max\{T,1\})^{\ell+\max\{\varsigma_1,1\}+ \nicefrac{ 1 }{ p }} \left[
1 
+
2 \big(
\max\{\| \dnnFunction_2( \xi ) \|,\|\xi\|\}
+ 
2 \,\max\{\| f_1(0) \| ,C,1\}
\big)^{\ell}
\right]
\\ & \cdot 
\left(\left\| \xi - \dnnFunction_2( \xi ) \right\|+\varepsilon_1  +h+h^{ \nicefrac{ 1 }{ 2 } }\right)
\max\{L_1,1\} 
\\&\cdot
\left[
\max\{C,1\}+\max\{c,1\}
\big(
\max\{\| \dnnFunction_2( \xi ) \|,\|\xi\|\} 
+ 
2 \max\{\| f_1(0) \| ,C,1\}\big)^{\max\{\varsigma_1,1\}}
\right]
.
\end{split}
\end{equation}
Combining this with the fact that
  for all $a,b \in [0, \infty)$, $z\in [1,\infty)$ it holds that
\begin{equation}\label{perturbation_SDE:FirstFact}
\begin{split}
\big(1+2z^\ell\big) \big(a+b z^{\max\{\varsigma_1,1\}}\big)
&=a+b z^{\max\{\varsigma_1,1\}}+2az^\ell+2b z^{\ell+\max\{\varsigma_1,1\}}
\\&\le a+(3b+2a)z^{\ell+\max\{\varsigma_1,1\}}
\\&\le a+5 \max\{a,b\}z^{\ell+\max\{\varsigma_1,1\}}
\end{split}
\end{equation}
demonstrates that for all $t\in [0,T]$ it holds that
\begin{equation}
\begin{split}
&
\big| \E\big[ f_0( X_t ) \big] - \E\big[ \dnnFunction_0( Y_t ) \big] \big|
\\ & 
\leq
\varepsilon_0 e^{\varsigma_0 L_1 T} (\varpi_{\max\{\varsigma_0,\ell q,p \varsigma_1, p\}})^{\varsigma_0} (\max\{T,1\})^{\varsigma_0}
\\&\cdot \left(
1 
+ 
\big(
\max\{\| \dnnFunction_2( \xi ) \|,\|\xi\|\}
+ 
2 \,\max\{\| f_1(0) \| ,C,1\}
\big)^{\varsigma_0}
\right)
\\ & 
+
L_0
\,
2^{ \max\{ \ell - 1, 0 \} }
\,
e^{
	\left[ 
	L_1 
	+
	1 - \nicefrac{1}{p} 
	+\ell \max\{L_1,c\}+\max\{\varsigma_1,1\} c
	\right]
	T
}
(\varpi_{\max\{\varsigma_0,\ell q,p \varsigma_1, p\}})^{\ell+\max\{1,\varsigma_1\}}
\\&\cdot (\max\{T,1\})^{\ell+\max\{\varsigma_1,1\}+ \nicefrac{ 1 }{ p }} 
\left(\left\| \xi - \dnnFunction_2( \xi ) \right\|+\varepsilon_1  +h+h^{ \nicefrac{ 1 }{ 2 } }\right)
\max\{L_1,1\} 
\\&\cdot
\left[
\max\{C,1\}+5\max\{C,c,1\}
\big(
\max\{\| \dnnFunction_2( \xi ) \|,\|\xi\|\} 
+ 
2 \max\{\| f_1(0) \| ,C,1\}\big)^{\ell+\max\{\varsigma_1,1\}}
\right]
.
\end{split}
\end{equation}
Therefore, we obtain that for all $t\in [0,T]$ it holds that
\begin{align*}
&
\big| \E\big[ f_0( X_t ) \big] - \E\big[ \dnnFunction_0( Y_t ) \big] \big|
\le 
\left(\left\| \xi - \dnnFunction_2( \xi ) \right\|+\varepsilon_0+\varepsilon_1  +h+h^{ \nicefrac{ 1 }{ 2 } }\right)
\\&\cdot e^{\left[\max\{\varsigma_0,1\}L_1+1 - \nicefrac{1}{p}+\ell \max\{L_1,c\}+\max\{\varsigma_1,1\} c\right]
	T}
(\varpi_{\max\{\varsigma_0,\ell q,p \varsigma_1, p\}})^{\max\{\varsigma_0,\ell+\max\{1,\varsigma_1\}\}} \numberthis \\&\cdot(\max\{T,1\})^{\max\{\varsigma_0,\ell+\max\{\varsigma_1,1\}+ \nicefrac{ 1 }{ p }\}}
\max\{L_0,1\}\,\max\{L_1,1\}\,
2^{ \max\{ \ell - 1, 0 \} } 
\\&\cdot
\left[
\max\{C,1\}+5\max\{C,c,1\}
\big(
\max\{\| \dnnFunction_2( \xi ) \|,\|\xi\|\} 
+ 
2 \max\{\| f_1(0) \| ,C,1\}\big)^{\max\{\varsigma_0,\ell+\max\{\varsigma_1,1\}\}}
\right].
\end{align*}
The hypothesis that $	\| \dnnFunction_2( \xi )-\xi \|\le \varepsilon_2 ( 1 + \| \xi \|^{ \varsigma_2 } )$ and the fact that
\begin{equation}
	\max\{\| \xi\|,\|\dnnFunction_2( \xi ) \|\}\le \| \xi\|+	\| \dnnFunction_2( \xi )-\xi \|
	\leq  \| \xi\|+
	\varepsilon_2 ( 1 + \| \xi \|^{ \varsigma_2 } )
\end{equation}
hence imply that for all $t\in[0,T]$ it holds that
\begin{align*}
&
\big| \E\big[ f_0( X_t ) \big] - \E\big[ \dnnFunction_0( Y_t ) \big] \big|
\le 
\left(\varepsilon_2 ( 1 + \| \xi \|^{ \varsigma_2 } )+\varepsilon_0+\varepsilon_1  +h+h^{ \nicefrac{ 1 }{ 2 } }\right)
\\&\cdot e^{\left[\max\{\varsigma_0,1\}L_1+1 - \nicefrac{1}{p}+\ell \max\{L_1,c\}+\max\{\varsigma_1,1\} c\right]
	T}
(\varpi_{\max\{\varsigma_0,\ell q,p \varsigma_1, p\}})^{\max\{\varsigma_0,\ell+\max\{1,\varsigma_1\}\}} \numberthis \\&\cdot(\max\{T,1\})^{\max\{\varsigma_0,\ell+\max\{\varsigma_1,1\}+ \nicefrac{ 1 }{ p }\}}
\max\{L_0,1\}\, \max\{L_1,1\}\,
2^{ \max\{ \ell - 1, 0 \} } 
\\&\cdot
\left[
\max\{C,1\}+5\max\{C,c,1\}
\big(
\| \xi\|+
\varepsilon_2 ( 1 + \| \xi \|^{ \varsigma_2 } )
+ 
2 \max\{\| f_1(0) \| ,C,1\}\big)^{\max\{\varsigma_0,\ell+\max\{\varsigma_1,1\}\}}
\right]\!.
\end{align*}
This completes the proof of Proposition~\ref{prop:perturbation_PDE_2}.
\end{proof}

%% file: InterpolationEulerMethod.tex
\begin{lemma}\label{lem:Euler}
	Let $d,m,N\in\N$, $T, p\in (0,\infty)$, $C, c \in [0, \infty)$,  $q\in [1,\infty)$,  $x\in\R^d$, $B\in\R^{d\times m}$, $\timeGrid_0,\timeGrid_1,\dots, \timeGrid_N\in [0,T]$ satisfy that $0=\timeGrid_0<\timeGrid_1<\ldots <\timeGrid_{N-1}<\timeGrid_N=T$,
 let $\mu\colon \R^d\to \R^d$ be a 
measurable function, assume for all $x \in \R^d$ that
	$\|\mu(x)\| \leq C+c\|x\|$,
			let $ \downFixed{\cdot} \colon [0,T] \to [0,T] $ 
		satisfy for all 
			$ t \in [0,T] $ that $			\downFixed{t}
			=
			\max\!\left(
			\left\{ 
			\timeGrid_0, \timeGrid_1, \dots,\timeGrid_N
			\right\}
			\cap 
			[0,t]
			\right)$,
	let $ ( \Omega, \mathcal{F}, \P ) $ be a probability space, 
	let $ W \colon [0,T] \times \Omega \to \R^m $ be a standard Brownian motion,
	let 
	$ \zigZagProcess \colon [0,T] \times \Omega \to \R^d $
satisfy for all 
	$ t \in [0,T] $
	that
	\begin{equation}
	\label{eq:zigZag}
	\zigZagProcess_t 
	=
	x
	+
	\int_0^t
	\mu( 
	\zigZagProcess_{ \downFixed{s} } 
	)
	\,
	ds
	+
	B
	W_t
	,
	\end{equation}
and let 
$ \affineProcess \colon [0,T] \times \Omega \to \R^d $
satisfy for all 
$n\in\{0,1,\dots,N-1\}$,
$ t \in [\timeGrid_n,\timeGrid_{n+1}]$  
that $\affineProcess_0=x$ and
\begin{equation}
\label{eq:affine}
\affineProcess_t 
=
\affineProcess_{\timeGrid_n}+ \tfrac{t-\timeGrid_n}{\timeGrid_{n+1}-\timeGrid_n}\left[\mu( 
\affineProcess_{\timeGrid_n} 
) (\timeGrid_{n+1}-\timeGrid_n)
+B
(W_{\timeGrid_{n+1}}-W_{\timeGrid_n})\right]
\end{equation}
(cf.\ Definition~\ref{Def:euclideanNorm}).
Then 
\begin{enumerate}[(i)]
	\item\label{item:stochastic} it holds that $\zigZagProcess$ and $\affineProcess$ are stochastic processes,
	\item\label{item:EulerZero}
	it holds for all 
	$n\in\{0,1,\dots,N\}$ that $\affineProcess_{\timeGrid_n}=\zigZagProcess_{\timeGrid_n}$,
	\item 
		\label{item:EulerII}
	 it holds for all 
	 $n\in\{0,1,\dots,N-1\}$,
	 $ t \in (\timeGrid_n,\timeGrid_{n+1})$  
	 that
	\begin{equation}
	\begin{split}
	\big( \E\big[ \| \affineProcess_t-\zigZagProcess_t \|^p \big] \big)^{ \nicefrac{1}{p} }
\le \tfrac{1}{2} \sqrt{ \max\{1,p-1\} (\timeGrid_{n+1}-\timeGrid_n) \operatorname{Trace}(B B^{ \ast } )}, 
	\end{split}
	\end{equation}
	and
		\item 
	\label{item:EulerIII}
	it holds for all $t\in[0,T]$ that 
	\begin{equation}
	\begin{split}
	&\max\!\left\{ \big( \E\big[ \| \zigZagProcess_t \|^q \big] \big)^{ \nicefrac{1}{q} },\big( \E\big[ \| \affineProcess_t \|^q \big] \big)^{ \nicefrac{1}{q} }\right\}
	 \\&\leq
\Big[
\|x\| + C T 
+ 
 \sqrt{ \max\{1,q-1\} T \operatorname{Trace}(B B^{ \ast } )}
\Big]
\,
e^{ c T }.
	\end{split}
	\end{equation}
\end{enumerate}
\end{lemma}

\begin{proof}[Proof of Lemma~\ref{lem:Euler}]
	Throughout this proof 
	for every $\mathfrak{d}\in \N$ let $\operatorname{I}_{\mathfrak{d}}\in \R^{\mathfrak{d}\times \mathfrak{d}}$ be the identity matrix in $\R^{\mathfrak{d}\times \mathfrak{d}}$, 
	let 
	$ \upFixed{\cdot} \colon [0,T] \to [0,T] $ 
satisfy for all 
	$ t \in [0,T] $ that $			\upFixed{t}
	=
	\min\!\left(
	\left\{ 
	\timeGrid_0, \timeGrid_1, \dots,\timeGrid_N
	\right\}
	\cap 
	[t,T]
	\right)$, 	and
let $\rho\colon [0,T]\to [0,1]$ satisfy for all $t \in [0, T]$ that
\begin{equation}\label{EulerAuxiliaryFunction}
 \rho(t)=  \begin{cases}
\frac{t-\downFixed{t}}{\upFixed{t}-\downFixed{t}} &\colon \quad t \notin \{\timeGrid_0,\timeGrid_1,\dots,\timeGrid_N\}\\
0 & \colon  \quad t \in \{\timeGrid_0,\timeGrid_1,\dots,\timeGrid_N\}
\end{cases}.
\end{equation}
Observe that \eqref{eq:zigZag}, the fact that for all  $t \in [0, T]$ it holds that $ \Omega \ni \omega \mapsto  W_t (\omega) \in \R^m$ is measurable, and induction imply that for all $t \in [0, T]$ it holds that $\Omega \ni \omega \mapsto \zigZagProcess_t (\omega) \in \R^d$ is measurable.  Moreover, note that \eqref{eq:affine}, the fact that for all  $t \in [0, T]$ it holds that $ \Omega \ni \omega \mapsto  W_t (\omega) \in \R^m$ is measurable, and induction prove that  for all $t \in [0, T]$ it holds that $\Omega \ni \omega \mapsto \affineProcess_t (\omega) \in \R^d$ is measurable. Combining this with the fact that for all $t \in [0, T]$ it holds that $\Omega \ni \omega \mapsto \zigZagProcess_t (\omega) \in \R^d$ is measurable establishes item~\eqref{item:stochastic}.
Next we claim that for all $n\in\{0,1,\dots, N\}$, $t\in [\timeGrid_{\max\{n-1, 0\}}, \timeGrid_n]$ it holds that $\affineProcess_{\timeGrid_n}=\zigZagProcess_{\timeGrid_n}$ and 
\begin{equation}\label{affineProcessIntegralRepresentation}
	\affineProcess_{t}=x+\int_0^{t}\mu(\affineProcess_{\downFixed{s}})\,ds+B W_{\downFixed{t}} +\rho(t) B(W_{\upFixed{t}}-W_{\downFixed{t}}).
\end{equation}
We prove \eqref{affineProcessIntegralRepresentation} by induction on $n\in\{0,1,\dots,N\}$. 
Note that the fact that $\affineProcess_{\timeGrid_0}=x$, the fact that $\rho(\timeGrid_0)=0$, and the fact that $W_{\timeGrid_0}=0$ demonstrate that
\begin{equation}
	\affineProcess_{\timeGrid_0}=x=x+\int_0^{\timeGrid_0}\mu(\affineProcess_{\downFixed{s}})\,ds+B W_{\downFixed{\timeGrid_0}} +\rho(\timeGrid_0) B(W_{\upFixed{\timeGrid_0}}-W_{\downFixed{\timeGrid_0}}).
\end{equation}
This and the fact that $\zigZagProcess_{\timeGrid_0}=x$ prove \eqref{affineProcessIntegralRepresentation} in the base case $n=0$.
For the induction step $\{0,1,\dots,N-1\}\ni n\to n+1\in \{1,2,\dots,N\}$ assume that there exists $n\in\{0,1,\dots,N-1\}$ which satisfies that for all 
$m\in\{0,1,\dots, n\}$, $t\in [\timeGrid_{\max\{n-1, 0\}}, \timeGrid_n]$ it holds that $\affineProcess_{\timeGrid_m}=\zigZagProcess_{\timeGrid_m}$
and
\begin{equation}\label{affineProcessIntegralRepresentationInduction}
\affineProcess_{t}=x+\int_0^{t}\mu(\affineProcess_{\downFixed{s}})\,ds+B W_{\downFixed{t}} +
\rho(t) B(W_{\upFixed{t}}-W_{\downFixed{t}}).
\end{equation}
Note that  \eqref{eq:zigZag} and \eqref{affineProcessIntegralRepresentationInduction} imply that 
\begin{equation}\label{affineProcessIntegralRepresentationInductionTwo}
\begin{split}
\affineProcess_{\timeGrid_n}
&=\zigZagProcess_{\timeGrid_n}=x+\int_0^{\timeGrid_n}\mu(\zigZagProcess_{\downFixed{s}})\,ds+B W_{\timeGrid_n}
=x+\int_0^{\timeGrid_n}\mu(\affineProcess_{\downFixed{s}})\,ds+B W_{\timeGrid_n}.
\end{split}
\end{equation}
Combining this with  \eqref{eq:affine} ensures that for all $t\in [\timeGrid_{n}, \timeGrid_{n+1}] = [\timeGrid_{\max\{n, 0\}}, \timeGrid_{n+1}]$ it holds that
\begin{equation}
\begin{split}
\affineProcess_t 
&=
\affineProcess_{\timeGrid_n}+ (t-\timeGrid_n)\mu( 
\affineProcess_{\timeGrid_n} 
)
+ \tfrac{t-\timeGrid_n}{\timeGrid_{n+1}-\timeGrid_n} B
(W_{\timeGrid_{n+1}}-W_{\timeGrid_n})
\\&=x+\int_0^{\timeGrid_n}\mu(\affineProcess_{\downFixed{s}})\,ds
 + BW_{\timeGrid_n}+ (t-\timeGrid_n)\mu( 
\affineProcess_{\timeGrid_n} )
+ \tfrac{t-\timeGrid_n}{\timeGrid_{n+1}-\timeGrid_n} B
(W_{\timeGrid_{n+1}}-W_{\timeGrid_n})
\\&=x+\int_0^{t}\mu(\affineProcess_{\downFixed{s}})\,ds+B W_{\timeGrid_n} +
 \tfrac{t-\timeGrid_n}{\timeGrid_{n+1}-\timeGrid_n} B
(W_{\timeGrid_{n+1}}-W_{\timeGrid_n}).
\end{split}
\end{equation}
Therefore, we obtain that for all $t\in [\timeGrid_{\max\{n, 0\}}, \timeGrid_{n+1}]$ it holds that 
\begin{equation}
\label{eq:induction:last}
\affineProcess_{t}=x+\int_0^{t}\mu(\affineProcess_{\downFixed{s}})\,ds+B W_{\downFixed{t}} +
\rho(t) B(W_{\upFixed{t}}-W_{\downFixed{t}}).
\end{equation}
This, \eqref{affineProcessIntegralRepresentationInduction}, and \eqref{eq:zigZag}  assure that 
\begin{equation}
\begin{split}
\affineProcess_{\timeGrid_{n+1}} =x+\int_0^{\timeGrid_{n+1}}\mu(\affineProcess_{\downFixed{s}})\,ds+B W_{\timeGrid_{n+1}}=x+\int_0^{\timeGrid_{n+1}}\mu(\zigZagProcess_{\downFixed{s}})\,ds+B W_{\timeGrid_{n+1}}
= \zigZagProcess_{\timeGrid_{n+1}}.
\end{split}
\end{equation}
Combining this with  \eqref{eq:induction:last} implies that for all $t\in [\timeGrid_{\max\{n, 0\}}, \timeGrid_{n+1}]$ it holds that $\affineProcess_{\timeGrid_{n+1}}=\zigZagProcess_{\timeGrid_{n+1}}$ and 
\begin{equation}
\affineProcess_{t}=x+\int_0^{t}\mu(\affineProcess_{\downFixed{s}})\,ds+B W_{\downFixed{t}} +
\rho(t) B(W_{\upFixed{t}}-W_{\downFixed{t}}).
\end{equation}
Induction thus proves \eqref{affineProcessIntegralRepresentation}. Next observe that \eqref{affineProcessIntegralRepresentation} establishes item~\eqref{item:EulerZero}. 
Moreover, note that  \eqref{eq:zigZag}, \eqref{eq:affine}, \eqref{EulerAuxiliaryFunction}, and \eqref{affineProcessIntegralRepresentation} demonstrate that for all $n\in\{0,1,\dots,N-1\}$,
$ t \in (\timeGrid_n,\timeGrid_{n+1})$   it holds that
	\begin{equation}
	\begin{split}
		&\affineProcess_t-\zigZagProcess_t\\
		&=\affineProcess_{\timeGrid_n}+ \rho(t)\!\left[(\timeGrid_{n+1}-\timeGrid_n)\mu( 
		\affineProcess_{\timeGrid_n} 
		)
		+B
		(W_{\timeGrid_{n+1}}-W_{\timeGrid_n})\right]
		-
		\left[x
		+
		\int_0^t
		\mu( 
		\zigZagProcess_{ \downFixed{s} } 
		)
		\,
		ds
		+
		B
		W_t\right]
				\\&=\affineProcess_{\timeGrid_n}+ (t-\timeGrid_n)\mu( 
		\affineProcess_{\timeGrid_n} 
		)
		+ \rho(t)B
		(W_{\timeGrid_{n+1}}-W_{\timeGrid_n})
	-
		\left[\zigZagProcess_{\timeGrid_n}
		+
		\int_{\timeGrid_n}^t
		\mu( 
		\zigZagProcess_{ \downFixed{s} } 
		)
		\,
		ds
		+
		B
		(W_t-W_{\timeGrid_n})\right]
						\\&= (t-\timeGrid_n)\mu( 
		\affineProcess_{\timeGrid_n} 
		)-(t-\timeGrid_n)
		\mu( 
		\zigZagProcess_{\timeGrid_n} 
		)
		+ \rho(t)B
		(W_{\timeGrid_{n+1}}-W_{\timeGrid_n})	
		-
		B
		(W_t-W_{\timeGrid_n}).
			\end{split}
		\end{equation}
This and \eqref{affineProcessIntegralRepresentation} prove that for all $n\in\{0,1,\dots,N-1\}$,
		$ t \in (\timeGrid_n,\timeGrid_{n+1})$   it holds that
		\begin{equation}\label{EulerCalculationOne}
		\begin{split}
			\affineProcess_t-\zigZagProcess_t
		&= \rho(t) B\big(W_{\timeGrid_{n+1}}-W_{\timeGrid_{n}}\big)
		+B
		W_{\timeGrid_{n}}-BW_t
		\\&= -[\rho(t)-1] BW_{\timeGrid_{n}}+\big[(\rho(t)-1)-\rho(t)\big] BW_t
		+\rho(t) B W_{\timeGrid_{n+1}}
		\\&= [\rho(t)-1] B\big(W_t-W_{\timeGrid_{n}}\big)+\rho(t)B\big( W_{\timeGrid_{n+1}}-W_t\big).
	\end{split}
	\end{equation}
In addition, note that
the hypothesis that $ W \colon [0,T] \times \Omega \to \R^m $ is a standard Brownian motion ensures that
\begin{enumerate}[(A)]
	\item it holds  for all  $a,\mathfrak{a}\in\R$, $r,s,t\in[0,T]$ with $r\le s\le t$  that $a B(W_t-W_s)+\mathfrak{a}B( W_{s}-W_r)$	
	is a centered Gaussian random variable  and
	\item it holds  for all  $a,\mathfrak{a}\in\R$, $r,s,t\in[0,T]$ with $r\le s\le t$ that
\begin{equation}\label{EulerCovDifferenceGeneral}
\covariance\big(a B(W_t-W_s)+\mathfrak{a}B( W_{s}-W_r)\big)=\left[a^2(t-s)+\mathfrak{a}^2(s-r)\right] B B^\ast.
\end{equation}		
\end{enumerate}
	Combining this with \eqref{EulerCalculationOne} ensures that for all $n\in \{0,1,\dots,N-1\}$, $t\in (\timeGrid_n,\timeGrid_{n+1})$ it holds that $\affineProcess_t-\zigZagProcess_t$ is a centered Gaussian random  variable. Moreover, note that \eqref{EulerCalculationOne} and \eqref{EulerCovDifferenceGeneral} demonstrate that  for all $n\in \{0,1,\dots,N-1\}$, $t\in (\timeGrid_n,\timeGrid_{n+1})$ it holds that
		\begin{equation}\label{EulerCovDifference}
	\covariance\big(\affineProcess_t-\zigZagProcess_t\big)=\left([\rho(t)-1]^2(t-\timeGrid_{n})+[\rho(t)]^2(\timeGrid_{n+1}-t)\right)B B^\ast.
	\end{equation}
	In addition, observe that \eqref{EulerAuxiliaryFunction} implies that for all $n\in \{0,1,\dots,N-1\}$, $t\in (\timeGrid_n,\timeGrid_{n+1})$  it holds that 
	\begin{equation}
	\begin{split}
	&[\rho(t)-1]^2(t-\timeGrid_{n})+[\rho(t)]^2(\timeGrid_{n+1}-t)
	\\&= [\rho(t)]^2 (\timeGrid_{n+1}-\timeGrid_{n})+\big[1-2\rho(t)\big] (t-\timeGrid_{n})
	\\&=
		\frac{(t-\timeGrid_{n})^2}{(\timeGrid_{n+1}-\timeGrid_{n})}+(t-\timeGrid_{n})-\frac{2(t-\timeGrid_{n})^2}{(\timeGrid_{n+1}-\timeGrid_{n})}
		= (t-\timeGrid_{n})
		\left[1-\frac{(t-\timeGrid_{n})}{(\timeGrid_{n+1}-\timeGrid_{n})}\right]
		\\&= 
		\frac{(t-\timeGrid_{n})(\timeGrid_{n+1}-t)}{(\timeGrid_{n+1}-\timeGrid_{n})}
		.
	\end{split}
	\end{equation}
	This and the fact that for all $a\in \R$, $b\in (a,\infty)$, $r\in [a,b]$ it holds that 
	\begin{equation}
		(r-a)(b-r)\le \left(\tfrac{1}{2}(b+a)-a\right)\left(b-\tfrac{1}{2}(b+a)\right)
		=\tfrac{1}{4} (b-a)^2
	\end{equation}
	show that for all $n\in \{0,1,\dots,N-1\}$, $t\in (\timeGrid_n,\timeGrid_{n+1})$ it holds that 
		\begin{equation}
	[\rho(t)-1]^2(t-\timeGrid_{n})+[\rho(t)]^2(\timeGrid_{n+1}-t)
	\le \tfrac{1}{4} \big(\timeGrid_{n+1}-\timeGrid_{n}\big)
	.
	\end{equation}		
	The fact that $B B^\ast$ is a symmetric positive semidefinite matrix and \eqref{EulerCovDifference} therefore imply that for all $t\in[0,T]$ it holds that 
	\begin{equation}
		\operatorname{Trace}\!\big(\!\covariance\big(\affineProcess_t-\zigZagProcess_t\big)\big)\le 
		\tfrac{1}{4} \big(\timeGrid_{n+1}-\timeGrid_{n}\big) \operatorname{Trace}(B B^\ast) .
	\end{equation}
	Lemma~\ref{lem:momentsGauss} hence demonstrates that for all $n\in \{0,1,\dots,N-1\}$, $t\in (\timeGrid_n,\timeGrid_{n+1})$ it holds that 
	\begin{equation}
	\begin{split}
	\big( \E\big[ \| \affineProcess_t-\zigZagProcess_t \|^p \big] \big)^{ \nicefrac{1}{p} }
	&\le \sqrt{ \max\{1,p-1\}  \operatorname{Trace}\!\big(\!\covariance\big(\affineProcess_t-\zigZagProcess_t\big)\big)} 
\\&\le \tfrac{1}{2} \sqrt{ \max\{1,p-1\} \big(\timeGrid_{n+1}-\timeGrid_{n}\big) \operatorname{Trace}(B B^\ast)}. 
	\end{split}
	\end{equation}
	This establishes item \eqref{item:EulerII}.
	Next note that  Lemma~\ref{lem:momentsGauss}, Lemma~\ref{lem:sde-lp-bound},  the fact that for all $t\in[0,T]$ it holds that $BW_t$ is a centered Gaussian random variable, and  the fact that for all $t\in[0,T]$ it holds that $\cov(BW_t)=B B^\ast\, t$ ensure that for all $t\in [0,T]$ it holds that 
	\begin{equation}\label{EulerMomentsZigZag}
\begin{split}
 \big(\E\! \left[\|\zigZagProcess_t\|^q\right]\big)^{\nicefrac{1}{q}} 
 &\leq
\Big[
\|x\| + C T 
+ 
\sup_{t\in[0,T]}\big(
\E\big[ 
\| B W_t \|^q
\big]
\big)^{ \nicefrac{ 1 }{ q } }
\Big]
\,
e^{ c T } 
\\& \leq
\Big[
\|x\| + C T 
+ 
\sup_{t\in[0,T]}\sqrt{ \max\{1,q-1\} \operatorname{Trace}(B B^\ast)t}
\Big]
\,
e^{ c T }
\\&=
\Big[
\|x\| + C T 
+ 
\sqrt{ \max\{1,q-1\} \operatorname{Trace}(B B^\ast)T}
\Big]
\,
e^{ c T }.
\end{split}
\end{equation}
	Next note that \eqref{EulerCovDifferenceGeneral}, the fact that $W_0=0$,  the fact that $B B^\ast$ is a symmetric positive semidefinite matrix, and the fact that $\forall \, t \in [0, T] \colon 0 \leq \rho(t) \leq 1$ imply that
\begin{enumerate}[a)]
	\item it holds for all $t\in [0,T]$ that $B
	W_{\downFixed{t}}
	+ \rho(t) B\big(W_{\upFixed{t}}-W_{\downFixed{t}}\big)$ is a centered Gaussian random variable and
	\item  it holds for all $t\in [0,T]$ that 
	\begin{equation}
	\begin{split}
	&\operatorname{Trace}\!\left(\covariance\left(B W_{\downFixed{t}}
	+ \rho(t) B\big(W_{\upFixed{t}}-W_{\downFixed{t}}\big)\right)\right)
	\\&=\operatorname{Trace}(B B^\ast) \left[\downFixed{t}+[\rho(t)]^2\big(\upFixed{t}-\downFixed{t}\big)\right]
	\le \operatorname{Trace}(B B^\ast) \upFixed{t}.
	\end{split}
	\end{equation}
\end{enumerate}	
Combining this with \eqref{affineProcessIntegralRepresentation}, Lemma~\ref{lem:momentsGauss}, and Lemma~\ref{lem:sde-lp-bound} demonstrates that for all $t\in [0,T]$ it holds that
\begin{equation}
\begin{split}
	 \big(\E\! \left[\|\affineProcess_t\|^q\right]\big)^{\nicefrac{1}{q}} 
	&\leq
	\Big[
	\|x\| + C T 
	+ 
	\sup_{t\in[0,T]}\big(
	\E\big[ 
	\| B W_{\downFixed{t}}
	+ \rho(t) B\big(W_{\upFixed{t}}-W_{\downFixed{t}}\big) \|^q
	\big]
	\big)^{ \nicefrac{ 1 }{ q } }
	\Big]
	\,
	e^{ c T } 
	\\& \leq
	\Big[
	\|x\| + C T 
	+ 
	\sup_{t\in[0,T]}\sqrt{ \max\{1,q-1\} \operatorname{Trace}(B B^\ast)\upFixed{t}}
	\Big]
	\,
	e^{ c T }
	\\&=
	\Big[
	\|x\| + C T 
	+ 
	\sqrt{ \max\{1,q-1\} \operatorname{Trace}(B B^\ast)T}
	\Big]
	\,
	e^{ c T }.
\end{split}
\end{equation}
This and \eqref{EulerMomentsZigZag} establish item~\eqref{item:EulerIII}.	This completes the proof of Lemma~\ref{lem:Euler}.
\end{proof}

%% file: Feynman.tex
\subsection{On the Feynman--Kac formula for additive noise driven SDEs}
\label{sec:feynman:kac}

\begin{prop}
\label{thm:feynman}
Let $ ( \Omega, \mathcal{F}, \P ) $ be a probability space, 
let $ T \in (0,\infty) $, $ d, m \in \N $, $ B \in \R^{ d \times m } $, 
$\varphi \in C(\R^d, \R)$,
let $ W \colon [0,T] \times \Omega \to \R^m $ be a standard Brownian motion, 
let $ \left< \cdot , \cdot \right> \colon \R^d \times \R^d \to \R $
be the standard scalar product on $\R^d$, 
let $ \mu \colon \R^d \to \R^d $ be a locally Lipschitz continuous function, 
and assume that
\begin{equation}
\inf_{ p \in (0,\infty) }
\sup_{ x \in \R^d }
\left[
\frac{ | \varphi(x) | }{
(
1 + \| x \|^p
)
}
+
\frac{
\left\| \mu(x) \right\|
}{
(
1 + \| x \|
)
}
\right]
< \infty
\end{equation}
(cf.\ Definition~\ref{Def:euclideanNorm}).
Then 
\begin{enumerate}[(i)]
\item 
\label{item1:feynmann}
there exist unique stochastic processes
$ X^x \colon [0,T] \times \Omega \to \R^d $, $ x \in \R^d $, 
with continuous sample paths 
which satisfy for all $ x \in \R^d $, $ t \in [0,T] $ that
\begin{equation}
X^x_t = x + \int_0^t \mu( X^x_s ) \, ds + B W_t 
,
\end{equation}
\item 
there exists a unique viscosity solution $u \in \{v \in C([0, T] \times \R^d, \R) \colon \allowbreak \inf_{ p \in (0,\infty) }
\allowbreak \sup_{ (t,x) \in [0,T] \times \R^d } \allowbreak
\frac{ | v(t,x) | }{
	1 + \| x \|^p
}
< \infty \}
$
of
\begin{equation}
\label{eq:CD.F4}
( \tfrac{ \partial }{ \partial t } 
u )(t,x)
=
\big\langle 
(\nabla_x u)(t,x),
\mu(x) 
\big\rangle 
+
\tfrac{1}{2}
\operatorname{Trace}\!\big(
B 
B^{*}       
(\textup{Hess}_x u)(t,x) 
\big)
\end{equation}
with $
u( 0, x ) = \varphi(x)
$
for $ (t,x) \in (0,T) \times \R^d $, and
\item 
it holds for all $ t \in [0,T] $, $ x \in \R^d $ that 
$
\E[ 
| \varphi( X^x_t ) | 
] < \infty
$
and $u(t,x) = \E[ \varphi( X^x_t ) ]$.
\end{enumerate}
\end{prop}

%% file: errorEstimates.tex
\begin{prop}
\label{prop:PDE_approx_Lp}
Let 
$ T, \kappa \in (0,\infty) $, $\eta \in [1, \infty)$,  $p \in [2, \infty)$,
let 
$
A_d = ( a_{ d, i, j } )_{ (i, j) \in \{ 1, 2,\dots, d \}^2 } $ $ \in \R^{ d \times d }
$,
$ d \in \N $,
be symmetric positive semidefinite matrices, 
 let $\nu_d  \colon \mathcal{B}([0,T]\times\R^d) \to [0,\infty)$, $d\in\N$, be finite measures which satisfy for all $d\in\N$ that 
\begin{equation}\label{PDE_approx_Lp:MeasureAssumption}
	\left[\int_{[0,T]\times \R^d} 
	\|x \|^{2p \max\{2\kappa, 3\}}
	\, \nu_d (d t, d x) \right]^{\!\nicefrac{1}{p}}\leq \eta d^{\eta},
\end{equation}
let
$f^m_d \in C( \R^d, \R^{md-m+1})$, $m\in\{0,1\}$, $d \in \N $,
and
  $\dnnFunction^m_{d, \varepsilon}\in C( \R^d,\R^{md-m+1})$, $m\in\{0,1\}$, $d\in\N$, $\varepsilon\in (0,1]$,
	satisfy for all
	$ d \in \N $, 
	$ \varepsilon \in (0,1] $, 
	$m\in\{0,1\}$, 
	$ 
	x,y \in \R^d
	$
	that 
	\vspace{-1ex}
	\begin{gather}
	\label{corDNNerrorEstimate:growthPhiOne}
		|
	f^0_d( x )
	| 
	+
	\operatorname{Trace}(A_d) 	\leq 
	\kappa d^{ \kappa }
	(
	1 + \| x \|^{ \kappa }
	), \qquad
	\| 
	f^1_d( x ) 
	- 
	f^1_d( y )
	\|
	\leq 
	\kappa 
	\| x - y \|, 
	\\
	\label{eq:appr:coef:diff}
	\| f^m_d(x) 
	- 
	\dnnFunction^m_{d, \varepsilon} (x)
	\|
	\leq 
	\varepsilon \kappa d^{ \kappa }
	(
	1 + \| x \|^{ \kappa }
	), \qquad \|
	\dnnFunction^1_{d, \varepsilon}(x)    
	\|	
	\leq 
	\kappa ( d^{ \kappa } + \| x \| ),
	\\
\label{approximationLocallyLipschitz} \andq  | \dnnFunction^0_{d, \varepsilon} (x) - \dnnFunction^0_{d, \varepsilon} (y)| \leq \kappa d^{\kappa} (1 + \|x\|^{\kappa} + \|y \|^{\kappa})\|x-y\|,
	\end{gather}
let $ ( \Omega, \mathcal{F}, \P ) $ be a probability space, 
let $ W^{ d, m } \colon [0,T] \times \Omega \to \R^d $, $ d, m \in \N $, 
be independent standard Brownian motions, 
and let 
$ \affineProcess^{ N, d, m, x } \colon [0,T] \times \Omega \to \R^d $, $ x \in \R^d $,
$ N,d, m \in \N $,
be stochastic processes 
which satisfy for all 
$N, d, m \in \N $,
$ x \in \R^d $,
$n\in\{0,1,\dots,N-1\}$,
$ t \in \big[\tfrac{nT}{N},\tfrac{(n+1)T}{N}\big]$  
that $\affineProcess^{ N, d, m, x }_0=x$ and
\begin{equation}
\begin{split}
\affineProcess^{ N, d, m, x }_t 
=
\affineProcess^{ N, d, m, x }_{\frac{nT}{N}}+ \left(\tfrac{tN}{T}-n\right)\!
\Big[\tfrac{T}{N}\dnnFunction^1_{d, \deltaIndex} \big( 
\affineProcess^{ N, d, m, x }_{\frac{nT}{N}} 
\big)
+\sqrt{ 2 A_d }
\big(W^{ d, m }_{\frac{(n+1)T}{N}}-W^{ d, m }_{\frac{nT}{N}}\big)\Big]
\end{split}
\end{equation}
(cf.\ Definition~\ref{Def:euclideanNorm}).
Then 
\begin{enumerate}[(i)]
\item 
\label{item:existence_vis}
for every $d \in \N$		there exists a unique viscosity solution
$u_d \in  \{v \in C([0, T] \times \R^d, \R) \colon \allowbreak \inf_{ q \in (0,\infty) }
\allowbreak \sup_{ (t,x) \in [0,T] \times \R^d } \allowbreak
\frac{ | v(t,x) | }{
	1 + \| x \|^q
}
< \infty \}$ of
\begin{equation}
\begin{split}
( \tfrac{ \partial }{\partial t} u_d )( t, x ) 
& = 
( \tfrac{ \partial }{\partial x} u_d )( t, x )
\,
f^1_d( x )
+
\sum_{ i, j = 1 }^d
a_{ d, i, j }
\,
( \tfrac{ \partial^2 }{ \partial x_i \partial x_j } u_d )( t, x )
\end{split}
\end{equation}
with $ u_d( 0, x ) = f^0_d( x )$
for $ ( t, x ) \in (0,T) \times \R^d $
and 
\item
\label{item:PDE_approxMainStatement}
there exists $\mathcal{C}\in \R$ such that for all 
$d,N,M\in\N$  
it holds that 
\begin{equation}
\begin{split}
& \left(
\E\bigg[\int_{[0,T]\times\R^d}
\big|
u_d(t,x) 
- 
\tfrac{ 1 }{ M } 
\big[ 
\textstyle
\sum_{ m = 1 }^{M}
\dnnFunction^0_{d, \deltaIndex}\big(
\affineProcess^{N, d, m, x }_t
\big)
\big]
\big|^p
\,
\nu_d (d t, d x)
\bigg]
\right)^{\!\nicefrac{1}{p}} 
\\&\le 
\mathcal{C} \!
\left[ \frac{d^{\kappa(\kappa+4)+\max\{\eta,\kappa(2\kappa+1)\}}}{N^{\nicefrac{1}{2}}}
+ \frac{d^{\kappa+\max\{\eta,\kappa^2\}}}{M^{\nicefrac{1}{2}}} \right] \left[\max \!\big\{1,\nu_d([0,T]\times\R^d)\big\}\right]^{\!\nicefrac{1}{p}}\!.
\end{split}
\end{equation}
\end{enumerate}
\end{prop}

\begin{proof}[Proof of Proposition~\ref{prop:PDE_approx_Lp}]
Throughout this proof 
let $ \iota \in \R $ satisfy that
$ \iota = \max\{ \kappa, 1 \} $, 
let $C, \mathcal{C}_1, \mathcal{C}_2, \mathcal{C} \in (0,\infty)$ satisfy that
\begin{equation}
C=e^{
	\kappa^2 T
} 2^{\max\{0,\kappa-1\}}
\left(\eta
+ \left[\kappa  T 
+ 
\max\!\left\{1,\sqrt{2 (p\iota-1)  \kappa }\right\} T^{\nicefrac{1}{2}}\right]^{\!\kappa}\right),
\end{equation}
\begin{equation}\label{eq:defMathcalC}
\begin{split}
\mathcal{C}_1&=\iota^2
2^{\iota}(\kappa +1)\left[\max\!\left\{1,\sqrt{2 \max\{1,2\kappa-1\}  \kappa }\right\}\right]^{\!2\iota} e^{[3 \iota^2+\nicefrac{1}{2}]
	T}
(\max\{T,1\})^{\kappa + \iota + \nicefrac{3}{2}}
\\&\cdot \big[\!\max\{2 \kappa (\kappa +1 ) ,1\} \big]
\big[
1+5
\eta \, 2^{\kappa+\iota-1}
+5 (4\iota)^{\kappa+\iota} \, 2^{\kappa+\iota-1}
\big],
\end{split}
\end{equation}
\begin{equation}\label{eq:defMathcalCTwo}
\begin{split}
\mathcal{C}_2&=\tfrac{1}{\sqrt{2}} \kappa^{\nicefrac{3}{2}} e^{\kappa^2 T} 2^{\iota}  (\max\{T,1\})^{\kappa+\nicefrac{1}{2}} 
\left(\eta+1+	\left[\kappa 
+ 
\max\!\left\{1,\sqrt{2 \max\{1,2\kappa-1\}  \kappa }\right\} 
\right]^{\!\kappa}\right),
\end{split}
\end{equation}
and
\begin{equation}\label{eq:defFinalConstant}
	\mathcal{C}=\max\{\mathcal{C}_1+\mathcal{C}_2,8  \kappa  (1+C) \sqrt{p-1}\},
\end{equation} 
let $N,M\in\N$, $ \delta\in (0,\infty) $ satisfy that $\delta=\sqrt{T/N}$,
let $ \mathcal{A}_d \in \R^{ d \times d } $, $ d \in \N $, satisfy 
for all $ d \in \N $ that
$
\mathcal{A}_d = \sqrt{ 2 A_d }
$,
let $ \varpi_{ d, q } \in \R $, $ d \in \N $, $ q \in (0,\infty) $, satisfy 
for all $ q \in (0,\infty) $, $ d \in \N $  that
\begin{equation}\label{eq:PDE_approx_Lp_Varpi}
\begin{split}
\varpi_{ d, q } &= 
\max\!\Big\{1,\sqrt{\max\{1,q-1\}  \operatorname{Trace}( (\mathcal{A}_d)^*  \mathcal{A}_d)}\Big\},
\end{split}
\end{equation}
let $ X^{ d, x } \colon [0,T] \times \Omega \to \R^d $, $ x \in \R^d $,  $ d \in \N $, 
be stochastic processes with continuous sample paths 
which satisfy for all  $ d \in \N $, $ x \in \R^d $,  $ t \in [0,T] $ that
\begin{equation}
\label{eq:X_processes}
X^{ d, x }_t 
= x + \int_0^t f^1_d( X^{ d, x }_s ) \, ds 
+ 
\mathcal{A}_d
W^{ d, 1 }_t 
\end{equation} 
(cf.~item~\eqref{item1:feynmann} in \Cref{thm:feynman}), 
let $ \downFixed{\cdot} \colon [0,T] \to [0,T] $  satisfy for all
 $ t \in [0,T] $ that
 \begin{equation}
 	\downFixed{t}
 	=
 	\max\!\left(
 	\big\{0,\delta^2,2 \delta^2, \dots\big\}
 	\cap 
 	[0,t]
 	\right)\!,
 \end{equation}
let $\upFixed{\cdot} \colon [0,T] \to [0,T] $  satisfy for all
$ t \in [0,T] $ that
\begin{equation}
\upFixed{t}
=
\min\!\left(
\big\{0,\delta^2,2 \delta^2, \dots \big\}
\cap 
[t,T]
\right)\!,
\end{equation}
and let $\zigZagProcess^{d, x } \colon [0,T] \times \Omega \to \R^d $,
$ x \in \R^d $,
$ d \in \N $,
be stochastic processes with continuous sample paths 
which satisfy for all 
$ d \in \N $,
$ x \in \R^d $,
$ t \in [0,T]$  
that
\begin{equation}
\zigZagProcess^{d, x }_t 
=
x
+
\int_0^t
\dnnFunction^1_{d, \deltaIndexProof}\big( 
\zigZagProcess^{d, x }_{ \downFixed{s} } 
\big)
\,
ds
+
\mathcal{A}_d
W^{ d, 1 }_t.
\end{equation}
Note that  H\"older's inequality and \eqref{PDE_approx_Lp:MeasureAssumption} imply that for all $d \in \N$, $r\in (0,2\max\{2\kappa,3\})$ it holds that
\begin{equation}\label{PDE_approx_Lp:Hoelder}
\begin{split}
&\left[\int_{[0,T]\times\R^d} 
\|x\|^{ pr}
\, \nu_d (d t, d x)\right]^{\!\nicefrac{1}{p}}
\\&\leq 
\left[ \int_{[0,T]\times\R^d} 
\|x\|^{ 2p\max\{2\kappa,3\}}
\, \nu_d (d t, d x)  \right]^{\!\nicefrac{r}{(2p\max\{2\kappa,3\})}}
\big[\nu_d([0,T]\times\R^d)\big]^{\nicefrac{(1-\nicefrac{r}{(2\max\{2\kappa,3\})})}{p}}
\\&\leq 
(\eta d^\eta)^{\nicefrac{r}{(2\max\{2\kappa,3\})}}
\max\!\left\{1,[\nu_d([0,T]\times\R^d)]^{\nicefrac{1}{p}}\right\}
\\&\leq 
\eta d^\eta
\max\!\left\{1,[\nu_d([0,T]\times\R^d)]^{\nicefrac{1}{p}}\right\}
.
\end{split}
\end{equation}
Furthermore, observe that
\eqref{corDNNerrorEstimate:growthPhiOne} and
\Cref{thm:feynman} establish item~\eqref{item:existence_vis}. 
It thus remains to prove item~\eqref{item:PDE_approxMainStatement}.
For this note that 
the triangle inequality
and 
\Cref{thm:feynman}
ensure that 
for all $ d \in \N $
it holds that
\begin{equation}
\label{eq:apply_feynman}
\begin{split}
&
\bigg[\int_{[0,T]\times\R^d}
\E\Big[
\big|
u_d(t,x) 
- 
\tfrac{ 1 }{ M } 
\big[ 
\textstyle
\sum\nolimits_{ m = 1 }^M
\dnnFunction^0_{d, \deltaIndexProof}\big(\affineProcess_t^{N, d, m, x } \big)
\big]
\big|^p
\Big]
\,
\nu_d (d t, d x)\bigg]^{\!\nicefrac{1}{p}}
\\ &
\leq 
 \bigg[\int_{[0,T]\times\R^d}
\E\Big[
\big|
u_d(t,x) 
- 
\E\big[
\dnnFunction^0_{d, \deltaIndexProof}\big(
\affineProcess_t^{N, d, 1, x }
\big)
\big]
\big|^p
\Big]
\,
\nu_d (d t, d x)\bigg]^{\!\nicefrac{1}{p}}
\\ &
+
 \bigg[\int_{[0,T]\times\R^d}
\E\Big[
\big|
\E\big[
\dnnFunction^0_{d, \deltaIndexProof}(
\affineProcess_t^{N, d, 1, x }
)
\big]
- 
\tfrac{ 1 }{ M } 
\big[ 
\textstyle
\sum_{ m = 1 }^M
\dnnFunction^0_{d, \deltaIndexProof}(
\affineProcess_t^{N, d, m, x }
)
\big]
\big|^p
\Big]
\,
\nu_d (d t, d x)\bigg]^{\!\nicefrac{1}{p}}
\\ & =
 \bigg[\int_{[0,T]\times\R^d}
\big|
\E\big[ 
f^0_d( X^{ d, x }_t )
\big]
- 
\E\big[
\dnnFunction^0_{d, \deltaIndexProof}(
\affineProcess_t^{N, d, 1, x }
)
\big]
\big|^p
\,
\nu_d (d t, d x)\bigg]^{\!\nicefrac{1}{p}}
\\ &
+
 \bigg[\int_{[0,T]\times\R^d}
\E\Big[
\big|
\E\big[
\dnnFunction^0_{d, \deltaIndexProof}(
\affineProcess_t^{N, d, 1, x }
)
\big]
- 
\tfrac{ 1 }{ M } 
\big[ 
\textstyle
\sum_{ m = 1 }^M
\dnnFunction^0_{d, \deltaIndexProof}(
\affineProcess_t^{N, d, m, x }
)
\big]
\big|^p
\Big]
\,
\nu_d (d t, d x)\bigg]^{\!\nicefrac{1}{p}}
.
\end{split}
\end{equation}
This and, e.g., \cite[Corollary~2.5]{GrohsWurstemberger2023}
 prove that
for all $d \in \N $
it holds that
\begin{equation}\label{eq:apply_MonteCarlo}
\begin{split}
&
\bigg[\int_{[0,T]\times\R^d}
\E\Big[
\big|
u_d(t,x) 
- 
\tfrac{ 1 }{ M } 
\big[ 
\textstyle
\sum_{ m = 1 }^M
\dnnFunction^0_{d, \deltaIndexProof}(
\affineProcess_t^{N, d, m, x }
)
\big]
\big|^p
\Big]
\,
\nu_d (d t, d x)\bigg]^{\!\nicefrac{1}{p}}
\\ & \leq
 \bigg[\int_{[0,T]\times\R^d}
\big|
\E\big[ 
f^0_d( X^{ d, x }_t )
\big]
- 
\E\big[
\dnnFunction^0_{d, \deltaIndexProof}(
\affineProcess_t^{N, d, 1, x }
)
\big]
\big|^p
\,
\nu_d (d t, d x)\bigg]^{\!\nicefrac{1}{p}}
\\ &
+
\frac{  2 \sqrt{p-1} }{ M^{\nicefrac{1}{2}} } 
\bigg[\int_{[0,T]\times\R^d}
\E\Big[
\big|
\dnnFunction^0_{d, \deltaIndexProof}(
\affineProcess_t^{N, d, 1, x }
)
-
\E\big[
\dnnFunction^0_{d, \deltaIndexProof}(
\affineProcess_t^{N, d, 1, x }
)
\big]
\big|^p
\Big]
\,
\nu_d (d t, d x)\bigg]^{\!\nicefrac{1}{p}}
.
\end{split}
\end{equation}
Next note that the triangle inequality and H\"older's inequality imply  for all $ d \in \N $ that
\begin{equation}
\begin{split}
&\bigg[\int_{[0,T]\times\R^d}
\E\Big[
\big|
\dnnFunction^0_{d, \deltaIndexProof}(
\affineProcess_t^{N, d, 1, x }
)
-
\E\big[
\dnnFunction^0_{d, \deltaIndexProof}(
\affineProcess_t^{N, d, 1, x }
)
\big]
\big|^p
\Big]
\,
\nu_d (d t, d x)\bigg]^{\!\nicefrac{1}{p}}
\\&\le 
\bigg[\int_{[0,T]\times\R^d}
\E\Big[
\big|
\dnnFunction^0_{d, \deltaIndexProof}(
\affineProcess_t^{N, d, 1, x }
)
\big|^p
\Big]
\,
\nu_d (d t, d x)\bigg]^{\!\nicefrac{1}{p}}
\\&+
\bigg[\int_{[0,T]\times\R^d}
\big|
\E\big[
\dnnFunction^0_{d, \deltaIndexProof}(
\affineProcess_t^{N, d, 1, x }
)
\big]
\big|^p
\,
\nu_d (d t, d x)\bigg]^{\!\nicefrac{1}{p}}
\\&\le 
2\bigg[\int_{[0,T]\times\R^d}
\E\Big[
\big|
\dnnFunction^0_{d, \deltaIndexProof}(
\affineProcess_t^{N, d, 1, x }
)
\big|^p
\Big]
\,
\nu_d (d t, d x)\bigg]^{\!\nicefrac{1}{p}}.
\end{split}
\end{equation}
Combining this and \eqref{eq:apply_MonteCarlo} demonstrates for all $ d\in \N$ that
\begin{equation}
\label{eq:weak:RHS}
\begin{split}
&
\bigg[\int_{[0,T]\times\R^d}
\E\Big[
\big|
u_d(t,x) 
- 
\tfrac{ 1 }{ M } 
\big[ 
\textstyle
\sum_{ m = 1 }^M
\dnnFunction^0_{d, \deltaIndexProof}(
\affineProcess_t^{N, d, m, x }
)
\big]
\big|^p
\Big]
\,
\nu_d (d t, d x)\bigg]^{\!\nicefrac{1}{p}}
\\ & \leq
\bigg[\int_{[0,T]\times\R^d}
\big|
\E\big[ 
f^0_d( X^{ d, x }_t )
\big]
- 
\E\big[
\dnnFunction^0_{d, \deltaIndexProof}(
\affineProcess_t^{N, d, 1, x }
)
\big]
\big|^p
\,
\nu_d (d t, d x)\bigg]^{\!\nicefrac{1}{p}}
\\ &
 +
\frac{ 4 \sqrt{p-1} }{ M^{\nicefrac{1}{2}} } 
\left[\int_{[0,T]\times\R^d}
\E\big[
|
\dnnFunction^0_{d, \deltaIndexProof}(
\affineProcess_t^{N, d, 1, x }
)
|^p
\big]
\,
\nu_d (d t, d x)\right]^{\!\nicefrac{1}{p}}.
\end{split}
\end{equation}
Next observe that the fact that  $\forall \, a,b, q\in [0,\infty) \colon a^q+b^q\le (a+b)^q + (a+b)^q = 2 (a+b)^q $ proves that
for all $d \in \N$, $x, y \in \R^d$ it holds that
\begin{equation}
\begin{split}
&2 \int_0^1 \big[r \|x\| + (1-r) \|y\|\big]^{\kappa} \, dr \geq  \int_0^1 \big[ r^{\kappa} \|x\|^{\kappa} + (1-r)^{\kappa} \|y\|^{\kappa} \big] \, dr \\
&= \big[ \|x\|^{\kappa} + \|y\|^{\kappa} \big] \int_0^1 r^{\kappa} \, dr = \frac{\big[ \|x\|^{\kappa} + \|y\|^{\kappa} \big]}{\kappa +1}.
\end{split}
\end{equation}
This and  \eqref{approximationLocallyLipschitz}
 show that for all $ d \in \N$, $x, y \in \R^d$ it holds that
\begin{equation}\label{eq:weak:error:PhiZeroLocalLipschitzInt}
\begin{split}
&|\dnnFunction^0_{d, \deltaIndexProof} (x) - \dnnFunction^0_{d, \deltaIndexProof}(y)|   \leq \kappa d^{\kappa} (1 +  \|x\|^{\kappa} + \|y \|^{\kappa} )\|x-y\| \\
& \leq \kappa d^{\kappa} \left[ 1 + 2 (\kappa +1) \int_0^1 \big[r \|x\| + (1-r) \|y\|\big]^{\kappa} \, dr \right] \|x-y\|\\
& \leq 2 \kappa (\kappa +1 ) d^{\kappa} \left[ 1 +  \int_0^1 \big[r \|x\| + (1-r) \|y\|\big]^{\kappa} \, dr \right] \|x-y\|.
\end{split}
\end{equation}
Proposition~\ref{prop:perturbation_PDE_2} (applied with 
$d \is d$, $m \is d $, $ \xi \is x $, 
$ T \is T $, $c \is \kappa$, $ C \is \kappa d^{\kappa}$, $\varepsilon_0 \is \deltaIndexProof \kappa d^{\kappa}$, $\varepsilon_1 \is \deltaIndexProof \kappa d^{\kappa}$, $\varepsilon_2 \is 0$, $\varsigma_0 \is \kappa$, $\varsigma_1 \is \kappa$, $\varsigma_2 \is 0$, $L_0  \is  2 \kappa (\kappa +1 ) d^{\kappa}$, $L_1 \is \kappa$, $\ell \is \kappa$,
$ h \is \delta^2 $, $ p \is 2$, $ q \is 2$,
$ B \is \mathcal{A}_d$, 
$ (\varpi_r)_{r \in (0,\infty)} \is (\varpi_{d,r})_{r \in (0,\infty)} $, 
$ \left\| \cdot \right\|  \is \left\| \cdot \right\|$, 
$ ( \Omega, \mathcal{F}, \P ) \is ( \Omega, \mathcal{F}, \P ) $, 
$ W \is W^{d,1} $, 
$ \dnnFunction_0 \is \dnnFunction^0_{d, \deltaIndexProof}$, 
$ f_1 \is f^1_d $, 
$ \dnnFunction_2 \is \id_{\R^d}$, 
$ \chi(s) \is \downFixed{s}$, 
$ f_0 \is f^0_d $,
$ \dnnFunction_1 \is \dnnFunction^1_{d, \deltaIndexProof} $,
$X \is X^{d,x}$, $Y \is \zigZagProcess^{d, x }$
for $d \in \N$, $x \in \R^d$, $s\in [0,T]$
in the notation of Proposition~\ref{prop:perturbation_PDE_2}) hence ensures that for all $d \in \N$,  $x \in \R^d$, $t\in [0,T]$ it holds that
	\begin{equation}
\begin{split}\label{eq:weak:error:pointwiseStart}
&
\big| \E\big[ f^0_d( X_t^{d,x}) \big] - \E\big[ \dnnFunction^0_{d, \deltaIndexProof}( \zigZagProcess^{d, x }_t ) \big] \big|
\le 
\left(\deltaIndexProof \kappa d^{\kappa}+\deltaIndexProof \kappa d^{\kappa} +\delta^2+\delta\right)
\\&\cdot e^{\left[\max\{\kappa ,1\}\kappa+ 1 - \nicefrac{1}{2}+\kappa \max\{\kappa,\kappa\}+\max\{\kappa ,1\} \kappa\right]
	T}
(\varpi_{d,\max\{\kappa ,2\kappa ,2 \kappa , 2\}})^{\max\{\kappa, \kappa+\max\{1,\kappa \}\} } \\&\cdot(\max\{T,1\})^{\max\{\kappa ,\kappa+\max\{\kappa, 1 \} + \nicefrac{ 1 }{ 2 }\}}
\max\{2 \kappa (\kappa +1 ) d^{\kappa},1\}\, \max\{\kappa,1\}\,
2^{ \max\{ \kappa - 1, 0 \} } 
\\&\cdot
\left[
\max\{\kappa d^\kappa,1\}+5\max\{\kappa d^\kappa,\kappa,1\}
\big(
\| x\|
+ 
2 \max\{\| f^1_d(0) \| ,\kappa d^\kappa,1\}\big)^{\max\{\kappa,\kappa+\max\{\kappa ,1\}\}}
\right]\!.
\end{split}
\end{equation}
Next note that \eqref{corDNNerrorEstimate:growthPhiOne} demonstrates that for all $d\in\N$, $q\in (0,\infty)$ it holds that
\begin{equation}
\begin{split}\label{estimateVarPiD}
\varpi_{ d, q } &= 
\max\!\left\{1,\sqrt{\max\{1,q-1\}  \operatorname{Trace}( (\mathcal{A}_d)^{\ast}  \mathcal{A}_d)}\right\}
\\&=\max\!\left\{1,\sqrt{2 \max\{1,q-1\}  \operatorname{Trace}( A_d)}\right\}
\\&\le \max\!\left\{1,\sqrt{2 \max\{1,q-1\}  \kappa d^\kappa}\right\}
\\&\le d^{\nicefrac{\kappa}{2}} \max\!\left\{1,\sqrt{2 \max\{1,q-1\}  \kappa }\right\}.
\end{split}
\end{equation}
Moreover, observe that
 \eqref{eq:appr:coef:diff}  implies that for all $d\in\N$ it holds that
\begin{equation}
\begin{split}
\| 
f^1_d(0) 
\|
&\le \| 
f^1_d(0) 
- 
\dnnFunction^1_{d, \deltaIndexProof} (0)
\|+\| 
\dnnFunction^1_{d, \deltaIndexProof} (0)
\|
\\&\le \min\{\delta, 1\} \kappa d^{ \kappa }+
\kappa  d^{ \kappa }
\le 2 \kappa d^\kappa.
\end{split}
\end{equation}
This, \eqref{eq:weak:error:pointwiseStart}, \eqref{estimateVarPiD} and the fact that $\iota=\max\{\kappa,1\}$ ensure that for all $d \in \N$,  $x \in \R^d$, $t\in [0,T]$ it holds that
	\begin{equation}\label{eq:weak:error:pointwise}
\begin{split}
&
\big| \E\big[ f^0_d( X_t^{d,x}) \big] - \E\big[ \dnnFunction^0_{d, \deltaIndexProof}( \zigZagProcess^{d, x }_t ) \big] \big|
\\&\le 
\left(2\delta \kappa d^{\kappa}+ \delta^2 + \delta\right) e^{[3 \iota^2+\nicefrac{1}{2}]
	T}
\left[d^{\nicefrac{\kappa}{2}}\max\!\left\{1,\sqrt{2 \max\{1,2\kappa-1\}  \kappa }\right\}\right]^{\!2\iota} 
\\&\cdot(\max\{T,1\})^{\kappa + \iota + \nicefrac{ 1 }{ 2 }}
\max\{2 \kappa (\kappa +1 ) d^{\kappa},1\}\,\iota\,
2^{ \iota-1 } 
\\&\cdot
\left[
\max\{\kappa d^\kappa,1\}+5\max\{\kappa d^\kappa,1\}
\big(
\| x\|
+ 
2 \max\{2 \kappa d^\kappa,1\}\big)^{\kappa+\iota}
\right]\!.
\end{split}
\end{equation}
The triangle inequality hence ensures that for all  $d \in \N$ it holds that
\begin{equation}\label{eq:weak:error:integrated}
\begin{split}
&\left[\int_{[0,T]\times \R^d }
\big|
\E\big[ 
f^0_d( X^{ d, x }_t )
\big]
- 
\E\big[
\dnnFunction^0_{d, \deltaIndexProof}(
\zigZagProcess^{d, x }_t
)
\big]
\big|^p
\,
\nu_d (d t, d x)\right]^{\!\nicefrac{1}{p}} \\
& \leq 
 \left(2\delta \kappa d^{\kappa}+ \delta^2 +\delta\right)
e^{[3 \iota^2+\nicefrac{1}{2}]
	T}
\left[d^{\nicefrac{\kappa}{2}}\max\!\left\{1,\sqrt{2 \max\{1,2\kappa-1\}  \kappa }\right\}\right]^{\!2\iota} 
\\&\cdot(\max\{T,1\})^{\kappa + \iota+ \nicefrac{ 1 }{ 2 }}
\max\{2 \kappa (\kappa +1 ) d^{\kappa},1\}\,\iota
2^{ \iota-1 } 
\Bigg[
\max\{\kappa d^\kappa,1\}[\nu_d ([0,T]\times \R^d)]^{\nicefrac{1}{p}} \\& + 5\max\{\kappa d^\kappa,1\} 
\left[\int_{[0,T]\times \R^d }
\big(
\| x\|
+ 
2 \max\{2 \kappa d^\kappa,1\}\big)^{p(\kappa+\iota)}
\,
\nu_d (d t, d x)\right]^{\!\nicefrac{1}{p}}\Bigg].
\end{split}
\end{equation}
Moreover, observe that the fact that $ \forall \, y, z \in \R, \alpha \in [1, \infty) \colon |y + z|^{\alpha}  \leq 2^{\alpha -1}(|y|^{\alpha} + |z|^{\alpha})$ and the triangle inequality demonstrate that for all  $d \in \N$ it holds that
\begin{align*}
	&\left[\int_{[0,T]\times \R^d }
	\big(
	\| x\|
	+ 
	2 \max\{2 \kappa d^\kappa,1\}\big)^{p(\kappa+\iota)}
	\,
	\nu_d (d t, d x)\right]^{\!\nicefrac{1}{p}}
		\\&\le 2^{\kappa+\iota-1} \left[\int_{[0,T]\times \R^d }
	\Big[
	\| x\|^{\kappa+\iota}
	+ 
	(2 \max\{2 \kappa d^\kappa,1\})^{\kappa+\iota}\Big]^{p}
	\,
	\nu_d (d t, d x)\right]^{\!\nicefrac{1}{p}} \numberthis
	\\&\le  2^{\kappa+\iota-1}
	\left[\int_{[0,T]\times \R^d }
	\| x\|^{p(\kappa+\iota)}
	\,
	\nu_d (d t, d x)\right]^{\!\nicefrac{1}{p}}
+
	 2^{\kappa+\iota-1} [\nu_d ([0,T]\times \R^d)]^{\nicefrac{1}{p}}
	(2 \max\{2 \kappa d^\kappa,1\})^{\kappa+\iota}.
\end{align*}
This, the fact that $\kappa+\iota=\max\{2\kappa,\kappa+1\}<2\max\{2\kappa,3\}$, and 
\eqref{PDE_approx_Lp:Hoelder}
ensure that for all  $d \in \N$ it holds that
\begin{equation}
\label{eq:integral:x}
\begin{split}
&\left[\int_{[0,T]\times \R^d }
\big(
\| x\|
+ 
2 \max\{2 \kappa d^\kappa,1\}\big)^{p(\kappa+\iota)}
\,
\nu_d (d t, d x)\right]^{\!\nicefrac{1}{p}}
\\&\le  2^{\kappa+\iota-1}
\eta d^\eta
\max\!\left\{1,[\nu_d([0,T]\times\R^d)]^{\nicefrac{1}{p}}\right\}
\\&+
2^{\kappa+\iota-1} [\nu_d ([0,T]\times \R^d)]^{\nicefrac{1}{p}}
(2 \max\{2 \kappa d^\kappa,1\})^{\kappa+\iota}
\\&\le  2^{\kappa+\iota-1}
\max\!\left\{1,[\nu_d([0,T]\times\R^d)]^{\nicefrac{1}{p}}\right\}
\left[\eta d^\eta+(2 \max\{2 \kappa d^\kappa,1\})^{\kappa+\iota}\right].
\end{split}
\end{equation} 
In addition, note that the fact that $\delta=\sqrt{T/N}$ proves that for all $d \in \N$ it holds that
\begin{equation}
\begin{split}
&2 \delta \kappa d^{\kappa} + \delta^2 + \delta \le 2 \delta \kappa d^{\kappa} + \delta \sqrt{T} + \delta\\
& \leq 2 \delta \kappa d^{\kappa} + 2 \max\{\sqrt{T}, 1\} \delta \le 2 \max\{\sqrt{T}, 1\} (\delta \kappa d^{\kappa} + \delta).
\end{split}
\end{equation}
Combining this, \eqref{eq:weak:error:integrated}, and \eqref{eq:integral:x} shows that for all  $d \in \N$ it holds that
\begin{equation}
\begin{split}
&\left[\int_{[0,T]\times \R^d }
\big|
\E\big[ 
f^0_d( X^{ d, x }_t )
\big]
- 
\E\big[
\dnnFunction^0_{d, \deltaIndexProof}(
\zigZagProcess^{d, x }_t
)
\big]
\big|^p
\,
\nu_d (d t, d x)\right]^{\!\nicefrac{1}{p}} 
\\& \leq 
\iota
2^{\iota} \max\{\sqrt{T}, 1\} \left(\delta \kappa d^{\kappa}+\delta\right)\left[d^{\nicefrac{\kappa}{2}} \max\!\left\{1,\sqrt{2 \max\{1,2\kappa-1\}  \kappa }\right\}\right]^{\!2\iota} 
\\&\cdot e^{[3 \iota^2+\nicefrac{1}{2}]
	T}
(\max\{T,1\})^{\kappa + \iota + \nicefrac{ 1 }{ 2 }}
\max\{2 \kappa (\kappa +1 ) d^{\kappa},1\}
 \max\{\kappa d^\kappa,1\} 
 \\&\cdot\max\!\left\{1,[\nu_d ([0,T]\times \R^d)]^{\nicefrac{1}{p}}\right\} 
\Big(
1+5 \!\left[\eta d^\eta+
(2 \max\{2 \kappa d^\kappa,1\})^{\kappa+\iota}\right] 2^{\kappa+\iota-1}
\Big)
.
\end{split}
\end{equation}
Hence, we obtain that for all  $d \in \N$ it holds that 
\begin{equation}
\begin{split}
&\left[\int_{[0,T]\times \R^d }
\big|
\E\big[ 
f^0_d( X^{ d, x }_t )
\big]
- 
\E\big[
\dnnFunction^0_{d, \deltaIndexProof}(
\zigZagProcess^{d, x }_t
)
\big]
\big|^p
\,
\nu_d (d t, d x)\right]^{\!\nicefrac{1}{p}} 
\\& \leq 
\iota
2^{\iota} d^{3 \kappa+\kappa\iota} \delta (\kappa +1)\left[\max\!\left\{1,\sqrt{2 \max\{1,2\kappa-1\}  \kappa }\right\}\right]^{\!2\iota} 
\\&\cdot e^{[3 \iota^2+\nicefrac{1}{2}]
	T}
(\max\{T,1\})^{\kappa + \iota + 1}
\max\{2 \kappa (\kappa +1 ),1\}
\max\{\kappa ,1\} 
\\&\cdot \max\!\left\{1,[\nu_d ([0,T]\times \R^d)]^{\nicefrac{1}{p}}\right\}
\Big(
1+5 \!
\left[\eta d^{\eta}
+
(2 \max\{2 \kappa d^\kappa,1\})^{\kappa+\iota}\right]
 2^{\kappa+\iota-1}
\Big)
\\& \leq 
\iota^2
2^{\iota} d^{3 \kappa+\kappa\iota+\max\{\eta,\kappa(\kappa+\iota)\}} \delta (\kappa +1)\left[\max\!\left\{1,\sqrt{2 \max\{1,2\kappa-1\}  \kappa }\right\}\right]^{\!2\iota} 
\\&\cdot e^{[3 \iota^2+\nicefrac{1}{2}]
	T}
(\max\{T,1\})^{\kappa + \iota + 1}
\max\{2 \kappa (\kappa +1 ),1\} \max\!\left\{1,[\nu_d ([0,T]\times \R^d)]^{\nicefrac{1}{p}}\right\}
\\&\cdot 
\big[
1+5
\eta \, 2^{\kappa+\iota-1}
+5 
(4\iota)^{\kappa+\iota} \,
2^{\kappa+\iota-1} 
\big]
.
\end{split}
\end{equation}
The fact that $\delta=\sqrt{T/N}$ and \eqref{eq:defMathcalC} hence prove that for all  $d \in \N$ it holds that 
\begin{equation}\label{SDEerror}
\begin{split}
&\left[\int_{[0,T]\times \R^d }
\big|
\E\big[ 
f^0_d( X^{ d, x }_t )
\big]
- 
\E\big[
\dnnFunction^0_{d, \deltaIndexProof}(
\zigZagProcess^{d, x }_t
)
\big]
\big|^p
\,
\nu_d (d t, d x)\right]^{\!\nicefrac{1}{p}} 
\\& \leq 
N^{-\nicefrac{1}{2}} \iota^2
2^{\iota} d^{3\kappa+\kappa\iota+\max\{\eta,\kappa(\kappa+\iota)\}} (\kappa +1)\left[\max\!\left\{1,\sqrt{2 \max\{1,2\kappa-1\}  \kappa }\right\}\right]^{\!2\iota} 
\\&\cdot e^{[3 \iota^2+\nicefrac{1}{2}]
	T}
(\max\{T,1\})^{\kappa + \iota +\nicefrac{3}{2}}
\max\{2 \kappa (\kappa +1 ) ,1\}
\\&\cdot
\big[
1+5
\eta \, 2^{\kappa+\iota-1}
+5
(4\iota)^{\kappa+\iota} \, 2^{\kappa+\iota-1}
\big]\max\!\left\{1, [\nu_d ([0,T]\times \R^d)]^{\nicefrac{1}{p}}\right\}
\\& = 
N^{-\nicefrac{1}{2}} \,\mathcal{C}_1\, d^{3\kappa+\kappa\iota+\max\{\eta,\kappa(\kappa+\iota)\}}  \max\!\left\{1, [\nu_d ([0,T]\times \R^d)]^{\nicefrac{1}{p}}\right\}.
\end{split}
\end{equation}
Furthermore, observe that \eqref{approximationLocallyLipschitz}, the Cauchy-Schwarz inequality, and the triangle inequality ensure that for all $d \in \N$,  $x \in \R^d$, $t\in [0,T]$ it holds that
	\begin{equation}
\begin{split}\label{perturbation_PDE:YZdifferenceStart}
&
\big| \E\big[ \dnnFunction^0_{d, \deltaIndexProof}( \zigZagProcess^{d, x }_t ) \big]-\E\big[ \dnnFunction^0_{d, \deltaIndexProof}( \affineProcess_t^{N, d, 1,x} ) \big] \big|
\\&
\le 
\E\big[ \big| \dnnFunction^0_{d, \deltaIndexProof}( \zigZagProcess^{d, x }_t ) -\dnnFunction^0_{d, \deltaIndexProof}( \affineProcess_t^{N, d, 1,x} ) \big| \big]
\\&
\leq \kappa d^{\kappa} \E \big[\big(1 + \Norm{\zigZagProcess^{d, x }_t}^{\kappa} + \Norm{\affineProcess_t^{N, d, 1,x}}^{\kappa}\big)\Norm{\zigZagProcess^{d, x }_t-\affineProcess_t^{N, d, 1,x}} \big]
\\&
\leq \kappa d^{\kappa} \big(\E \big[\big(1 + \Norm{\zigZagProcess^{d, x }_t}^{\kappa} + \Norm{\affineProcess_t^{N, d, 1,x}}^{\kappa}\big)^2\big]\big)^{\!\nicefrac{1}{2}} 
\big(\E \big[\Norm{\zigZagProcess^{d, x }_t-\affineProcess_t^{N, d, 1,x}}^2 \big]\big)^{\!\nicefrac{1}{2}}
\\&
\leq \kappa d^{\kappa} \left[1+\big(\E \big[\Norm{\zigZagProcess^{d, x }_t}^{2\kappa} \big]\big)^{\!\nicefrac{1}{2}}
+\big(\E \big[\Norm{\affineProcess_t^{N, d, 1,x}}^{2\kappa} \big]\big)^{\!\nicefrac{1}{2}}\right]  \big(\E\big[\Norm{\zigZagProcess^{d, x }_t-\affineProcess_t^{N, d, 1,x}}^2 \big]\big)^{\!\nicefrac{1}{2}}.
\end{split}
\end{equation}
Next note that \eqref{eq:appr:coef:diff} and items~\eqref{item:EulerZero}--\eqref{item:EulerII} in Lemma~\ref{lem:Euler} 
prove that for all $d \in \N$,  $x \in \R^d$, $t\in [0,T]$ it holds that
	\begin{equation}\label{perturbation_PDE:EulerMomentsDifference}
\begin{split}
\big(\E\big[\Norm{\zigZagProcess^{d, x }_t-\affineProcess_t^{N, d, 1,x}}^2 \big]\big)^{\!\nicefrac{1}{2}}
&\le \tfrac{1}{2} \sqrt{  \big(\upFixed{t}-\downFixed{t}\big) \operatorname{Trace}( \mathcal{A}_d \mathcal{A}_d^{ \ast } )}
\\&= \tfrac{1}{2} \sqrt{ 2 \big(\upFixed{t}-\downFixed{t}\big) \operatorname{Trace}(A_d)}
\le \tfrac{1}{\sqrt{2}} \delta \sqrt{ \kappa d^\kappa}. 
\end{split}
\end{equation}
Moreover, observe that  H\"older's inequality, \eqref{eq:appr:coef:diff},
and item~\eqref{item:EulerIII} in Lemma~\ref{lem:Euler}
show that for all 
$ q \in (0,\infty) $, 
$ d, m \in \N $,
$ x \in \R^d $,
$t\in[0,T]$
it holds that
\begin{equation}
\begin{split}
&
\max\!\left\{\big(\E\big[ 
\| \zigZagProcess^{d, x }_t \|^{q}
\big]
\big)^{\! \nicefrac{ 1 }{q } },\big(\E\big[ 
\| \affineProcess_t^{N, d, m, x } \|^{q}
\big]
\big)^{ \!\nicefrac{ 1 }{q } }\right\}
\\&\le 
\max\!\left\{
\big(
\E\big[ 
\| \zigZagProcess^{d, x }_t \|^{\max\{q,1\}}
\big]
\big)^{ \!\nicefrac{ 1 }{ \max\{q,1\} } },
\big(
\E\big[ 
\| \affineProcess_t^{N, d, m, x } \|^{\max\{q,1\}}
\big]
\big)^{\! \nicefrac{ 1 }{ \max\{q,1\} } }
\right\}
\\& \leq
\Big[
\|x\| + \kappa d^\kappa T 
+ 
\sqrt{ \max\{1,\max\{1,q\}-1\} T \operatorname{Trace}( \mathcal{A}_d \mathcal{A}_d^{ \ast })}
\Big]
\,
e^{ \kappa T }.
\end{split}
\end{equation}
Combining this with \eqref{eq:PDE_approx_Lp_Varpi} ensures that for all 
$ q \in (0,\infty) $, 
$ d, m \in \N $,
$ x \in \R^d $,
$t\in[0,T]$
it holds that
\begin{equation}\label{perturbation_PDE:EulerMoments}
\begin{split}
&
\max\!\left\{\big(\E\big[ 
\| \zigZagProcess^{d, x }_t \|^{q}
\big]
\big)^{\! \nicefrac{ 1 }{q } },\big(\E\big[ 
\| \affineProcess_t^{N, d, m, x } \|^{q}
\big]
\big)^{\! \nicefrac{ 1 }{q } }\right\}
 \le
\Big[
\| x \|
+
\kappa d^{ \kappa } T
+ 
\varpi_{d,\max\{q,1\}} T^{\nicefrac{1}{2}}
\Big]
\,
e^{
	\kappa T
}.
\end{split}
\end{equation}
The fact that $ \forall \, y, z \in \R$, $\alpha \in (0, \infty) \colon |y + z|^{\alpha}  \leq 2^{\max\{0,\alpha -1\}}\allowbreak(|y|^{\alpha} + |z|^{\alpha})$, \eqref{perturbation_PDE:YZdifferenceStart}, and \eqref{perturbation_PDE:EulerMomentsDifference} hence demonstrate that for all $d \in \N$,  $x \in \R^d$, $t\in [0,T]$ it holds that 
	\begin{equation}
\begin{split}\label{perturbation_PDE:YZdifferenceContinued}
&
\big| \E\big[ \dnnFunction^0_{d, \deltaIndexProof}( \zigZagProcess^{d, x }_t ) \big]-\E\big[ \dnnFunction^0_{d, \deltaIndexProof}( \affineProcess_t^{N, d, 1,x} ) \big] \big|
\\&\le 
\kappa d^{\kappa} \left[1+2 \,\Big(
\| x \|
+
\kappa d^{ \kappa } T
+ 
\varpi_{d,\max\{2\kappa,1\}} T^{\nicefrac{1}{2}}
\Big)^\kappa
\,
e^{
	\kappa^2 T
}
\right]
\tfrac{1}{\sqrt{2}} \delta \sqrt{ \kappa d^\kappa}
\\&\le \tfrac{1}{\sqrt{2}} \delta
\big(\kappa d^{\kappa}\big)^{\!\nicefrac{3}{2}} e^{
	\kappa^2 T
} \left[1+ 2^{1+\max\{0,\kappa -1\}}\Big(
\| x \|^\kappa
+
\big[\kappa d^{ \kappa } T
+ 
\varpi_{d,\max\{2\kappa,1\}} T^{\nicefrac{1}{2}}
\big]^\kappa\Big)
\,
\right]
\\&\le \tfrac{1}{\sqrt{2}} \delta
\big(\kappa d^{\kappa}\big)^{\!\nicefrac{3}{2}} e^{
	\kappa^2 T
} 2^{\max\{1,\kappa \}} \left[\| x \|^\kappa+ \Big(
1
+
\big[\kappa d^{ \kappa } T
+ 
\varpi_{d,\max\{2\kappa,1\}} T^{\nicefrac{1}{2}}
\big]^\kappa\Big)
\,
\right]
\!.
\end{split}
\end{equation}
Combining this with the triangle inequality
assures that for all $d \in \N$ it holds that 
\begin{align*}\label{perturbation_PDE:YZdifferenceIntegrated}
	&\left[\int_{[0,T]\times \R^d }
	\big|
	\E\big[ \dnnFunction^0_{d, \deltaIndexProof}( \zigZagProcess^{d, x }_t ) \big]-\E\big[ \dnnFunction^0_{d, \deltaIndexProof}( \affineProcess_t^{N, d, 1,x} ) \big]
	\big|^p
	\,
	\nu_d (d t, d x)\right]^{\!\nicefrac{1}{p}} \numberthis
	\\&\le  \tfrac{1}{\sqrt{2}} \delta
\big(\kappa d^{\kappa}\big)^{\!\nicefrac{3}{2}} e^{
	\kappa^2 T
} 2^{\iota} 
\left[\int_{[0,T]\times \R^d }
\left[\| x \|^\kappa+ \Big(
1
+
\big[\kappa d^{ \kappa } T
+ 
\varpi_{d,\max\{2\kappa,1\}} T^{\nicefrac{1}{2}}
\big]^\kappa\Big)
\,
\right]^p
\nu_d (d t, d x)\right]^{\!\nicefrac{1}{p}}
	\\&\le  \tfrac{1}{\sqrt{2}} \delta
\big(\kappa d^{\kappa}\big)^{\!\nicefrac{3}{2}} e^{
	\kappa^2 T
} 2^{\iota} \left[\int_{[0,T]\times \R^d }
\| x \|^{p\kappa}
\,
\nu_d (d t, d x)\right]^{\!\nicefrac{1}{p}}
\\&+ \tfrac{1}{\sqrt{2}} \delta
\big(\kappa d^{\kappa}\big)^{\!\nicefrac{3}{2}} e^{
	\kappa^2 T
} 2^{\iota}
\Big(
1
+
\big[\kappa d^{ \kappa } T
+ 
\varpi_{d,\max\{2\kappa,1\}} T^{\nicefrac{1}{2}}
\big]^\kappa\Big) [\nu_d([0,T]\times\R^d)]^{\nicefrac{1}{p}}.
\end{align*}
Next note that \eqref{estimateVarPiD} ensures for all $d\in\N$ that 
\begin{equation}
\begin{split}
&1+\big[\kappa d^{ \kappa } T
+ 
\varpi_{d,\max\{2\kappa,1\}} T^{\nicefrac{1}{2}}
\big]^\kappa
\\&\le 	1+\left[\kappa d^{ \kappa } T
+ 
d^{\nicefrac{\kappa}{2}} \max\!\left\{1,\sqrt{2 \max\{1,2\kappa-1\}  \kappa }\right\} T^{\nicefrac{1}{2}}
\right]^{\!\kappa}
\\&\le 1+d^{(\kappa^2)}(\max\{T,1\})^\kappa	\left[\kappa 
+ 
\max\!\left\{1,\sqrt{2 \max\{1,2\kappa-1\}  \kappa }\right\} 
\right]^{\!\kappa}.
\end{split}
\end{equation}
Hence, we obtain for all $d\in\N$ that 
\begin{equation}\label{perturbation_PDE:YZdifferenceAuxiliary}
\begin{split}
&1+\big[\kappa d^{ \kappa } T
+ 
\varpi_{d,\max\{2\kappa,1\}} T^{\nicefrac{1}{2}}
\big]^\kappa
\\&\le d^{(\kappa^2)}(\max\{T,1\})^\kappa \left(1+	\left[\kappa 
+ 
\max\!\left\{1,\sqrt{2 \max\{1,2\kappa-1\}  \kappa }\right\} 
\right]^{\!\kappa}\right).
\end{split}
\end{equation}
This, \eqref{perturbation_PDE:YZdifferenceIntegrated}, \eqref{PDE_approx_Lp:Hoelder},
and the fact that $\kappa<2\max\{2\kappa,3\}$
 establish that for all $d \in \N$ it holds that 
\begin{equation}
\begin{split}
&\left[\int_{[0,T]\times \R^d }
\big|
\E\big[ \dnnFunction^0_{d, \deltaIndexProof}( \zigZagProcess^{d, x }_t ) \big]-\E\big[ \dnnFunction^0_{d, \deltaIndexProof}( \affineProcess_t^{N, d, 1,x} ) \big]
\big|^p
\,
\nu_d (d t, d x)\right]^{\!\nicefrac{1}{p}}
\\&\le  \tfrac{1}{\sqrt{2}} \delta
\big(\kappa d^{\kappa}\big)^{\!\nicefrac{3}{2}} e^{
	\kappa^2 T
} 2^{\iota} \eta d^\eta
\max\!\left\{1,[\nu_d([0,T]\times\R^d)]^{\nicefrac{1}{p}}\right\}
\\&+ \tfrac{1}{\sqrt{2}} \delta
\big(\kappa d^{\kappa}\big)^{\!\nicefrac{3}{2}} e^{
	\kappa^2 T
} 2^{\iota}
d^{(\kappa^2)}(\max\{T,1\})^\kappa  [\nu_d([0,T]\times\R^d)]^{\nicefrac{1}{p}}
\\&\cdot \left(1+	\left[\kappa 
+ 
\max\!\left\{1,\sqrt{2 \max\{1,2\kappa-1\}  \kappa }\right\} 
\right]^{\!\kappa}\right).
\end{split}
\end{equation}
Combining this with \eqref{eq:defMathcalCTwo} and the fact that $\delta=\sqrt{T/N}$ implies that
 for all $d \in \N$ it holds that 
\begin{equation}\label{perturbation_PDE:YZdifferenceIntegratedContinued}
\begin{split}
&\left[\int_{[0,T]\times \R^d }
\big|
\E\big[ \dnnFunction^0_{d, \deltaIndexProof}( \zigZagProcess^{d, x }_t ) \big]-\E\big[ \dnnFunction^0_{d, \deltaIndexProof}( \affineProcess_t^{N, d, 1,x} ) \big]
\big|^p
\,
\nu_d (d t, d x)\right]^{\!\nicefrac{1}{p}}
\\&\le \delta d^{\nicefrac{3\kappa}{2}+\eta} \tfrac{1}{\sqrt{2}} \kappa^{\nicefrac{3}{2}} e^{\kappa^2 T} 2^{\iota}  \eta
\max\!\left\{1,[\nu_d([0,T]\times\R^d)]^{\nicefrac{1}{p}}\right\}
\\&+ 
\delta d^{\nicefrac{3\kappa}{2}+\kappa^2} \tfrac{1}{\sqrt{2}} \kappa^{\nicefrac{3}{2}} e^{\kappa^2 T} 2^{\iota}
[\nu_d([0,T]\times\R^d)]^{\nicefrac{1}{p}}
\\&\cdot (\max\{T,1\})^\kappa \left(1+	\left[\kappa 
+ 
\max\!\left\{1,\sqrt{2 \max\{1,2\kappa-1\}  \kappa }\right\} 
\right]^{\!\kappa}\right)
\\&\le \delta d^{\nicefrac{3\kappa}{2}+\max\{\eta,\kappa^2\}} \tfrac{1}{\sqrt{2}} \kappa^{\nicefrac{3}{2}} e^{\kappa^2 T} 2^{\iota}  
\max\!\left\{1,[\nu_d([0,T]\times\R^d)]^{\nicefrac{1}{p}}\right\}
\\&\cdot (\max\{T,1\})^\kappa \left(\eta+1+	\left[\kappa 
+ 
\max\!\left\{1,\sqrt{2 \max\{1,2\kappa-1\}  \kappa }\right\} 
\right]^{\!\kappa}\right)
\\&\le N^{-\nicefrac{1}{2}} d^{\nicefrac{3\kappa}{2}+\max\{\eta,\kappa^2\}} \tfrac{1}{\sqrt{2}} \kappa^{\nicefrac{3}{2}} e^{\kappa^2 T} 2^{\iota}  
\max\!\left\{1,[\nu_d([0,T]\times\R^d)]^{\nicefrac{1}{p}}\right\}
\\&\cdot (\max\{T,1\})^{\kappa+\nicefrac{1}{2}} \left(\eta+1+	\left[\kappa 
+ 
\max\!\left\{1,\sqrt{2 \max\{1,2\kappa-1\}  \kappa }\right\} 
\right]^{\!\kappa}\right)
\\&= N^{-\nicefrac{1}{2}} d^{\nicefrac{3\kappa}{2}+\max\{\eta,\kappa^2\}} \mathcal{C}_2
\max\!\left\{1,[\nu_d([0,T]\times\R^d)]^{\nicefrac{1}{p}}\right\}.
\end{split}
\end{equation}
Next note that the triangle inequality proves that for all $d \in \N$ it holds that 
\begin{equation}
\begin{split}
&\left[\int_{[0,T]\times \R^d }
\big|
\E\big[ 
f^0_d( X^{ d, x }_t )
\big]
- 
\E\big[
\dnnFunction^0_{d, \deltaIndexProof}(
\affineProcess_t^{N, d, 1, x }
)
\big]
\big|^p
\,
\nu_d (d t, d x)\right]^{\!\nicefrac{1}{p}}
\\&\le\left[\int_{[0,T]\times \R^d }
\big|
\E\big[ 
f^0_d( X^{ d, x }_t )
\big]
- 
\E\big[
\dnnFunction^0_{d, \deltaIndexProof}(
\zigZagProcess^{d, x }_t
)
\big]
\big|^p
\,
\nu_d (d t, d x)\right]^{\!\nicefrac{1}{p}}
\\&+\left[\int_{[0,T]\times \R^d }
\big|
\E\big[ \dnnFunction^0_{d, \deltaIndexProof}( \zigZagProcess^{d, x }_t ) \big]-\E\big[ \dnnFunction^0_{d, \deltaIndexProof}( \affineProcess_t^{N, d, 1,x} ) \big]
\big|^p
\,
\nu_d (d t, d x)\right]^{\!\nicefrac{1}{p}}.
\end{split}
\end{equation}
This, \eqref{SDEerror}, and \eqref{perturbation_PDE:YZdifferenceIntegratedContinued} ensure that  for all $d \in \N$ it holds that 
\begin{equation}
\begin{split}
&\left[\int_{[0,T]\times \R^d }
\big|
\E\big[ 
f^0_d( X^{ d, x }_t )
\big]
- 
\E\big[
\dnnFunction^0_{d, \deltaIndexProof}(
\affineProcess_t^{N, d, 1, x }
)
\big]
\big|^p
\,
\nu_d (d t, d x)\right]^{\!\nicefrac{1}{p}}
\\&\le
N^{-\nicefrac{1}{2}} \,\mathcal{C}_1\, d^{3\kappa+\kappa\iota+\max\{\eta,\kappa(\kappa+\iota)\}}  \max\!\left\{1, [\nu_d ([0,T]\times \R^d)]^{\nicefrac{1}{p}}\right\}
\\&+
N^{-\nicefrac{1}{2}} d^{\nicefrac{3\kappa}{2}+\max\{\eta,\kappa^2\}} \mathcal{C}_2
\max\!\left\{1,[\nu_d([0,T]\times\R^d)]^{\nicefrac{1}{p}}\right\}.
\end{split}
\end{equation}
The fact that $\iota\le \kappa+1$, the fact that ${\nicefrac{3\kappa}{2}+\max\{\eta,\kappa^2\}}\le \kappa(\kappa+4)+\max\{\eta,\kappa(2\kappa+1)\}$, and \eqref{eq:defFinalConstant} hence demonstrate that for all $d \in \N$ it holds that 
\begin{equation}\label{perturbation_PDE:SDEtoZerror}
\begin{split}
&\left[\int_{[0,T]\times \R^d }
\big|
\E\big[ 
f^0_d( X^{ d, x }_t )
\big]
- 
\E\big[
\dnnFunction^0_{d, \deltaIndexProof}(
\affineProcess_t^{N, d, 1, x }
)
\big]
\big|^p
\,
\nu_d (d t, d x)\right]^{\!\nicefrac{1}{p}}
\\&\le N^{-\nicefrac{1}{2}} d^{\kappa(\kappa+4)+\max\{\eta,\kappa(2\kappa+1)\}} \left[\mathcal{C}_1+ \mathcal{C}_2\right] \max\!\left\{1, [\nu_d ([0,T]\times \R^d)]^{\nicefrac{1}{p}}\right\}
\\&\le N^{-\nicefrac{1}{2}} d^{\kappa(\kappa+4)+\max\{\eta,\kappa(2\kappa+1)\}} \mathcal{C} \max\!\left\{1, [\nu_d ([0,T]\times \R^d)]^{\nicefrac{1}{p}}\right\}\!.
\end{split}
\end{equation}
Next observe that  \eqref{corDNNerrorEstimate:growthPhiOne} and \eqref{eq:appr:coef:diff} prove  for all 
$ d \in \N $,  $ x \in \R^d $
that
\begin{equation}
\label{eq:Phi_0_d_estimate}
\begin{split}
| 
\dnnFunction^0_{d, \deltaIndexProof}( x ) 
|
& \leq
| 
\dnnFunction^0_{d, \deltaIndexProof}( x ) 
-
f^0_d( x )
|
+
| 
f^0_d( x ) 
|
\\ &
\leq
\deltaIndexProof \kappa d^{ \kappa } 
( 1 + \| x \|^{ \kappa } )
+
\kappa d^{ \kappa } 
( 1 + \| x \|^{ \kappa } )
\leq
2 \kappa d^{ \kappa } 
( 1 + \| x \|^{ \kappa } )
.
\end{split}
\end{equation}
Moreover, note that  \eqref{perturbation_PDE:EulerMoments} and the fact that $ \forall \, y, z \in \R, \alpha \in (0, \infty) \colon |y + z|^{\alpha}  \leq 2^{\max\{0,\alpha -1\}}(|y|^{\alpha} + |z|^{\alpha})$ imply that for all  $d\in\N$ it holds that
\begin{equation}
	\begin{split}
	&\left[\int_{[0,T]\times \R^d }\E  \big[   \|
	\affineProcess_t^{N, d, 1, x }
	\|^{ p\kappa }
	\big]
	\, \nu_d (d t, d x)\right]^{\!\nicefrac{1}{p}}
	\\&\le 	
	e^{
		\kappa^2 T
	} \left[  \int_{[0,T]\times \R^d }  \big[
	\|x \|
	+
	\kappa d^{ \kappa } T 
	+ 
	\varpi_{d,\max\{p\kappa,1\}} T^{\nicefrac{1}{2}}
	\big]^{p \kappa}
	\, \nu_d (d t, d x) \right]^{\!\nicefrac{1}{p}}
		\\&\le 	
	e^{
		\kappa^2 T
	} 2^{\max\{0,\kappa-1\}}\left[\int_{[0,T]\times \R^d }  \big[
	\|x \|^\kappa
	+
	(\kappa d^{ \kappa } T 
	+ 
	\varpi_{d,\max\{p\kappa,1\}} T^{\nicefrac{1}{2}})^\kappa
	\big]^{p}
	\, \nu_d (d t, d x) \right]^{\!\nicefrac{1}{p}}.
	\end{split}
\end{equation}
Combining this with the triangle inequality ensures that for all  $d\in\N$ it holds that 
\begin{equation}
\begin{split}
&\left[\int_{[0,T]\times \R^d }\E  \big[   \|
\affineProcess_t^{N, d, 1, x }
\|^{ p\kappa }
\big]
\, \nu_d (d t, d x)\right]^{\!\nicefrac{1}{p}}
\\&\le 	
e^{
	\kappa^2 T
} 2^{\max\{0,\kappa-1\}}\left[  \int_{[0,T]\times \R^d } 
\|x \|^{p\kappa} \, \nu_d (d t, d x) \right]^{\!\nicefrac{1}{p}}
\\& +e^{
	\kappa^2 T
} 2^{\max\{0,\kappa-1\}} (\kappa d^{ \kappa } T 
+ 
\varpi_{d,\max\{p\kappa,1\}} T^{\nicefrac{1}{2}})^{\kappa}
[\nu_d ([0,T]\times \R^d)]^{\nicefrac{1}{p}}.
\end{split}
\end{equation}
This, \eqref{PDE_approx_Lp:Hoelder},   and  the fact that $\kappa<2\max\{2\kappa,3\}$  demonstrate that for all  $d\in\N$ it holds that 
\begin{equation}
\label{eq:prop:norm:bound:1}
\begin{split}
&\left[\int_{[0,T]\times \R^d }\E  \big[   \|
\affineProcess_t^{N, d, 1, x }
\|^{ p\kappa } 
\big]
\, \nu_d (d t, d x)\right]^{\!\nicefrac{1}{p}}
\\&\le 	
e^{
	\kappa^2 T
} 2^{\max\{0,\kappa-1\}}\left(\eta d^\eta
+ (\kappa d^{ \kappa } T 
+ 
\varpi_{d,\max\{p\kappa,1\}} T^{\nicefrac{1}{2}})^{\kappa}\right)
\max\!\left\{1,[\nu_d([0,T]\times\R^d)]^{\nicefrac{1}{p}}\right\}.
\end{split}
\end{equation}
In addition, observe that \eqref{estimateVarPiD} and the fact that $\max\{p\kappa,1\}\le p\iota$ prove that for all $d \in \N$ it holds that
\begin{equation}
\begin{split}
\varpi_{d,\max\{p\kappa,1\}}  &\leq d^{\nicefrac{\kappa}{2}} \max\!\left\{1,\sqrt{2 \max\{1,\max\{p\kappa,1\}-1\}  \kappa }\right\}\\& \leq d^{\nicefrac{\kappa}{2}} \max\!\left\{1,\sqrt{2 \max\{1, p\iota-1\}  \kappa }\right\}\\
& = d^{\nicefrac{\kappa}{2}} \max\!\left\{1,\sqrt{2 ( p\iota-1)  \kappa }\right\}.
\end{split}
\end{equation}
This and \eqref{eq:prop:norm:bound:1} ensure that for all  $d\in\N$ it holds that 
\begin{equation}
\begin{split}
&\left[\int_{[0,T]\times \R^d }\E  \big[   \|
\affineProcess_t^{N, d, 1, x }
\|^{ p\kappa } 
\big]
\, \nu_d (d t, d x)\right]^{\!\nicefrac{1}{p}}
\\&\le 	e^{
	\kappa^2 T
} 2^{\max\{0,\kappa-1\}}
\left(\eta d^\eta
+ \left[\kappa d^{ \kappa } T 
+ 
 d^{\nicefrac{\kappa}{2}}\max\!\left\{1,\sqrt{2 (p\iota-1) \kappa }\right\} T^{\nicefrac{1}{2}}\right]^{\!\kappa}\right)
\\&\cdot\max\!\left\{1,[\nu_d([0,T]\times\R^d)]^{\nicefrac{1}{p}}\right\}
\\&\le 	e^{
	\kappa^2 T
} 2^{\max\{0,\kappa-1\}}
d^{\max\{\eta,\kappa^2\}} \left(\eta
+ \left[\kappa  T 
+ 
\max\!\left\{1,\sqrt{2 (p\iota-1)  \kappa }\right\} T^{\nicefrac{1}{2}}\right]^{\!\kappa}\right)
\\&\cdot\max\!\left\{1,[\nu_d([0,T]\times\R^d)]^{\nicefrac{1}{p}}\right\}
\\&= 	
C d^{\max\{\eta,\kappa^2\}}  \max\!\left\{1,[\nu_d([0,T]\times\R^d)]^{\nicefrac{1}{p}}\right\}.
\end{split}
\end{equation}

This, the triangle inequality, and \eqref{eq:Phi_0_d_estimate} assure that for all  $d\in\N$ it holds that
\begin{equation}
	\begin{split}
	&\left[\int_{[0,T]\times\R^d}
	\E\big[
	|
	\dnnFunction^0_{d, \deltaIndexProof}(
	\affineProcess_t^{N, d, 1, x }
	)
	|^p
	\big]
	\,
	\nu_d (d t, d x)\right]^{\!\nicefrac{1}{p}}
	\\&\le 2 \kappa d^\kappa \left([\nu_d([0,T]\times\R^d)]^{\nicefrac{1}{p}}+C d^{\max\{\eta,\kappa^2\}}  \max\!\left\{1,[\nu_d([0,T]\times\R^d)]^{\nicefrac{1}{p}}\right\}\right)
	\\&\le  2 \kappa d^{\kappa+\max\{\eta,\kappa^2\}} \big(1+C   \big) \max\!\left\{1,[\nu_d([0,T]\times\R^d)]^{\nicefrac{1}{p}}\right\}.
	\end{split}
\end{equation}
Combining this, Fubini's theorem, \eqref{perturbation_PDE:SDEtoZerror},  \eqref{eq:defFinalConstant}, and \eqref{eq:weak:RHS} proves that for all  $d\in\N$ it holds that 
\begin{equation}
\label{eq:controlTotalError}
\begin{split}
& \left(
\E\bigg[\int_{[0,T]\times\R^d}
\big|
u_d(t,x) 
- 
\tfrac{ 1 }{ M } 
\big[ 
\textstyle
\sum_{ m = 1 }^{M}
\dnnFunction^0_{d, \deltaIndexProof}\big(
\affineProcess_t^{N, d, m, x }
\big)
\big]
\big|^p
\,
\nu_d (d t, d x)
\bigg]
\right)^{\!\nicefrac{1}{p}} 
\\&= N^{-\nicefrac{1}{2}} \mathcal{C} d^{\kappa(\kappa+4)+\max\{\eta,\kappa(2\kappa+1)\}}  \max\!\left\{1, [\nu_d ([0,T]\times \R^d)]^{\nicefrac{1}{p}}\right\}
\\&+ \frac{4 \sqrt{p-1}}{M^{\nicefrac{1}{2}}} 2 \kappa d^{\kappa+\max\{\eta,\kappa^2\}} (1+C   ) \max\!\left\{1,[\nu_d([0,T]\times\R^d)]^{\nicefrac{1}{p}}\right\}
\\&\le \mathcal{C} \!
\left[ \frac{d^{\kappa(\kappa+4)+\max\{\eta,\kappa(2\kappa+1)\}}}{N^{\nicefrac{1}{2}}}
+ \frac{d^{\kappa+\max\{\eta,\kappa^2\}}}{M^{\nicefrac{1}{2}}} \right] \left[\max\!\big\{1,\nu_d([0,T]\times\R^d)\big\}\right]^{\!\nicefrac{1}{p}}\!.
\end{split}
\end{equation}
This establishes  item~\eqref{item:PDE_approxMainStatement}.
This completes the proof of Proposition~\ref{prop:PDE_approx_Lp}.
\end{proof}

%% file: ANNdefinitions.tex
\subsection{ANNs}
\label{subsec:DNNs}

\begin{definition}[ANNs]
	\label{Def:ANN}
	We denote by $\ANNs$ the set given by 
		\begin{equation}
	\begin{split}
	\ANNs
	&=
	\cup_{L \in \N}
	\cup_{ l_0,l_1,\ldots, l_L \in \N}
	\left(
	\times_{k = 1}^L (\R^{l_k \times l_{k-1}} \times \R^{l_k})
	\right)
	\end{split}
	\end{equation}
	and we denote by 	$
	\paramANN, 
	\lengthANN, \inDimANN, \outDimANN \colon \ANNs \to \N
	$ and
	$\dims\colon\ANNs\to  \cup_{L=2}^\infty\, \N^{L}$
	the functions which satisfy
	for all $ L\in\N$, $l_0,l_1,\ldots, l_L \in \N$, 
	$
	\Phi 
	\in  \allowbreak
	( \times_{k = 1}^L\allowbreak(\R^{l_k \times l_{k-1}} \times \R^{l_k}))$
	that
	\begin{align}
	\paramANN(\Phi)
	=
	\textstyle\sum_{k = 1}^L l_k(l_{k-1} + 1), \qquad \lengthANN(\Phi)=L, \qquad \inDimANN(\Phi)=l_0, \qquad \outDimANN(\Phi)=l_L,
	\end{align}
and $\dims(\Phi)= (l_0,l_1,\ldots, l_L)$. 
\end{definition}

\subsection{Realizations of ANNs}

\begin{definition}[Multidimensional versions]
	\label{Def:multidim_version}
	Let $d \in \N$ and let $\psi \colon \R \to \R$ be a function.
	Then we denote by $\mathfrak{M}_{\psi, d} \colon \R^d \to \R^d$ the function which satisfies for all $ x = ( x_1, x_2, \dots, x_{d} ) \in \R^{d} $ that
	\begin{equation}\label{multidim_version:Equation}
	\mathfrak{M}_{\psi, d}( x ) 
	=
	\left(
	\psi(x_1)
	, \psi(x_2),
	\ldots
	,
	\psi(x_d)
	\right).
	\end{equation}
\end{definition}

\begin{definition}[Realizations associated to ANNs]
	\label{Definition:ANNrealization}
	Let $a\in C(\R,\R)$.
	Then we denote by 
	$
	\functionANNgeneral \colon \ANNs \to (\cup_{k,l\in\N}\,C(\R^k,\R^l))
	$
	the function which satisfies
	for all  $ L\in\N$, $l_0,l_1,\ldots, l_L \in \N$, 
	$
	\Phi 
	=
	((W_1, B_1), (W_2, B_2), \allowbreak \ldots, (W_L,\allowbreak B_L))
	\in  \allowbreak
	( \times_{k = 1}^L\allowbreak(\R^{l_k \times l_{k-1}} \times \R^{l_k}))
	$,
	$x_0 \in \R^{l_0}, x_1 \in \R^{l_1}, \ldots, x_{L} \in \R^{l_{L}}$ 
	with $\forall \, k \in \N \cap (0,L) \colon x_k =\activationDim{l_k}(W_k x_{k-1} + B_k)$  
	that
	\begin{equation}
	\label{setting_NN:ass2}
	\functionANNgeneral(\Phi) \in C(\R^{l_0},\R^{l_L})\qandq
	( \functionANNgeneral(\Phi) ) (x_0) = W_L x_{L-1} + B_L
	\end{equation}
	(cf.\ \Cref{Def:ANN,Def:multidim_version} and Figure~\ref{figure_defin}).
\end{definition}

\def\layersep{2.5cm}
\begin{figure}[h]
	\centering
	\begin{adjustbox}{width=\textwidth}
		\begin{tikzpicture}[shorten >=1pt,->,draw=black!50, node distance=\layersep]
			\tikzstyle{every pin edge}=[<-,shorten <=1pt]
			\tikzstyle{output neuron}=[very thick, circle,draw=black, fill=red!30, minimum size=40pt,inner sep=0pt]
			\tikzstyle{input neuron}=[very thick, circle, draw=black,fill=-red!80,minimum size=40pt,inner sep=0pt]
			\tikzstyle{hidden neuron}=[very thick, circle,draw=black,fill={rgb:black,1;white,5},minimum size=30pt,inner sep=0pt]
			\tikzstyle{annot} = [text width=9em, text centered]
			\tikzstyle{annot2} = [text width=4em, text centered]
			
			\node[input neuron] (I-1) at (-3,-0.cm) {\large $1$};
			\node[input neuron] (I-2) at (-3,-2.7cm) {\large $2$};
			\node(I-dots) at (-3,-4.2cm) {\vdots};
			\node[input neuron] (I-3) at (-3,-5.8cm) {\large $l_0$};

			\path[yshift = 2cm]
			node[hidden neuron](H0-1) at (0*\layersep, -1 cm) { $1$};
			\path[yshift = 2cm]
			node[hidden neuron](H0-2) at (0*\layersep, -3.5 cm) { $2$};
			\path[yshift = 2cm]
			node[hidden neuron](H0-3) at (0*\layersep, -6 cm) { $3$};
			\path[yshift = 2cm]
			node(H0-dots) at (0*\layersep, -7.4 cm) {\vdots};
			\path[yshift = 2cm]
			node[hidden neuron](H0-4) at (0*\layersep, -9 cm) { $l_1$};
			%
			\path[yshift = 2cm]
			node[hidden neuron](H1-1) at (1.5*\layersep, -1 cm) {$1$};
			\path[yshift = 2cm]
			node[hidden neuron](H1-2) at (1.5*\layersep, -3.5 cm) {$2$};
			\path[yshift = 2cm]
			node[hidden neuron](H1-3) at (1.5*\layersep, -6 cm) {$3$};
			\path[yshift = 2cm]
			node(H1-dots) at (1.5*\layersep, -7.4 cm) {\vdots};
			\path[yshift = 2cm]
			node[hidden neuron](H1-4) at (1.5*\layersep, -9 cm) {$l_2$};
			
			\path[yshift = 1cm]
			node(Hdot-1) at (2.7*\layersep, -1 cm) {$\cdots$};
			\path[yshift = 0.5cm]
			node(Hdot-2) at (2.7*\layersep, -3.4 cm) {$\cdots$};
			\path[yshift = 0.5cm]
			node(Hdot-dots) at (2.7*\layersep, -4.7 cm) {$\vdots$};
			\path[yshift = 0.5cm]
			node(Hdot-3) at (2.7*\layersep, -6.5 cm) {$\cdots$};
			
			\path[yshift = 2cm]
			node[hidden neuron](H2-1) at (3.9*\layersep, -1 cm) {$1$};
			\path[yshift = 2cm]
			node[hidden neuron](H2-2) at (3.9*\layersep, -3.5 cm) {$2$};
			\path[yshift = 2cm]
			node[hidden neuron](H2-3) at (3.9*\layersep, -6 cm) {$3$};
			\path[yshift = 2cm]
			node(H2-dots) at (3.9*\layersep, -7.4 cm) {\vdots};
			\path[yshift = 2cm]
			node[hidden neuron](H2-4) at (3.9*\layersep, -9 cm) {$l_{L-1}$};

			\path[yshift = 1.5cm]
			node[output neuron](O-1) at (5.4*\layersep,-1.5 cm) {\large $1$}; 
			\path[yshift = 1.5cm]
			node[output neuron](O-2) at (5.4*\layersep,-4.2 cm) {\large $2$}; 
			\path[yshift = 1.5cm]
			node(O-dots) at (5.4*\layersep, -5.7 cm) {\vdots};
			\path[yshift = 1.5cm]
			node[output neuron](O-3) at (5.4*\layersep,-7.2 cm) {\large $l_L$};
			\foreach \source in {1,2,3}
			\foreach \dest in {1,2,3,4}
			\path[-{Latex[length=2mm, width=4mm]}, line width = 0.8, draw=black] (I-\source) -- (H0-\dest);

			\foreach \source in {1,2,3,4}
			\foreach \dest in {1,2,3,4}
			\path[-{Latex[length=2mm, width=4mm]}, line width = 0.8, draw=black] (H0-\source) -- (H1-\dest);

			\foreach \source in {1,2,3,4}
			\foreach \dest in {1,2,3}
			\draw[-{Latex[length=2mm, width=4mm]}, line width = 0.8, draw=black] (H1-\source) -- (Hdot-\dest);
			
			\foreach \source in {1,2,3}
			\foreach \dest in {1,2,3,4}
			\draw[-{Latex[length=2mm, width=4mm]}, line width = 0.8, draw=black] (Hdot-\source) -- (H2-\dest);
			
			\foreach \source in {1,2,3,4}
			\foreach \dest in {1,2,3}
			\path[-{Latex[length=2mm, width=4mm]}, line width = 0.8, draw=black] (H2-\source) -- (O-\dest);

			\node[annot,above of=H0-1, node distance=1.4cm, align=center] (hl) {$1^{\text{st}}$ hidden layer };
			\node[annot,above of=H1-1, node distance=1.4cm, align=center] (hl) {$2^{\text{nd}}$ hidden layer};
			\node[annot,above of=H2-1, node distance=1.4cm, align=center] (hl2) {${(L - 1)}^{\text{th}}$ hidden layer};
			\node[annot,above of=I-1, node distance=1.2cm, align=center] {Input layer};
			\node[annot,above of=O-1, node distance=1.4cm, align=center] {Output layer};
			
	\node[annot,below of=H0-4, node distance=1.4cm, align=center] (sl) {$x_1 = \activationDim{l_1}(W_1 x_{0}$\\$ + B_1) \in \R^{l_1}$};
\node[annot,below of=H1-4, node distance=1.4cm, align=center] (sl) { $x_2 = \activationDim{l_2}(W_2 x_{1} $\\$+ B_2) \in \R^{l_2}$};
\node[annot,below of=H2-4, node distance=1.4cm, align=center] (sl2) {$x_{L-1} $\\
	$= \activationDim{l_{L-1}}(W_{L-1} x_{L-2} $\\$+ B_{L-1}) \in \R^{l_{L-1}}$};
\node[annot2,below of=I-3, node distance=1.2cm, align=center] {$x_0 \in \R^{l_0}$};
\node[annot,below of=O-3, node distance=1.8cm, align=center] {$(\mathcal{R}_a(\Phi)) (x_0) $\\$= W_L x_{L-1} + B_L$\\
	$\in \R^{l_L}$};
		\end{tikzpicture}
	\end{adjustbox}
	\caption{Graphical illustration for the realization function and the architecture of an ANN $ \Phi=
		((W_1, B_1),\allowbreak \ldots, (W_L,\allowbreak B_L))
		\in  \allowbreak
		( \times_{k = 1}^L\allowbreak(\R^{l_k \times l_{k-1}} \times \R^{l_k})) \subseteq \ANNs
		$ (see \Cref{Def:ANN}) where $ L\in\N$ describes the number of affine linear transformations, where $l_0,l_1,\ldots, l_L \in \N$ describe the dimensions of the layers of the ANN, and where the function $a  \colon \R \to \R$ represents the activation function (see \Cref{Definition:ANNrealization}).}
	\label{figure_defin}
\end{figure}

\begin{definition}[Rectifier function]
	\label{Definition:Relu1}
	We denote by $ \mathfrak{r} \colon \R \to \R $ the function which satisfies 
	for all $ x \in \R $ that $	\mathfrak{r} (x) = \max\{ x, 0 \}.$
\end{definition}

\subsection{Compositions of ANNs}
\label{subsec:DNN:composition}

\begin{definition}[Standard compositions of ANNs]
	\label{Definition:ANNcomposition}
	We denote by $\compANN{(\cdot)}{(\cdot)}\colon\allowbreak \{(\Phi_1,\Phi_2)\allowbreak\in\ANNs\times \ANNs\colon \inDimANN(\Phi_1)=\outDimANN(\Phi_2)\}\allowbreak\to\ANNs$ the function which satisfies for all 
	$ L,\mathfrak{L}\in\N$, $l_0,l_1,\ldots, l_L, \mathfrak{l}_0,\mathfrak{l}_1,\ldots, \mathfrak{l}_\mathfrak{L} \in \N$, 
	$
	\Phi_1
	=
	((W_1, B_1), (W_2, B_2), \allowbreak \ldots, (W_L,\allowbreak B_L))
	\in  \allowbreak
	( \times_{k = 1}^L\allowbreak(\R^{l_k \times l_{k-1}} \times \R^{l_k}))
	$,
	$
	\Phi_2
	=
	((\mathcal{W}_1, \mathfrak{B}_1),\allowbreak, (\mathcal{W}_2, \mathfrak{B}_2),\allowbreak \ldots, (\mathcal{W}_\mathfrak{L},\allowbreak \mathfrak{B}_\mathfrak{L}))
	\in  \allowbreak
	( \times_{k = 1}^\mathfrak{L}\allowbreak(\R^{\mathfrak{l}_k \times \mathfrak{l}_{k-1}} \times \R^{\mathfrak{l}_k}))
	$ 
	with $l_0=\inDimANN(\Phi_1)=\outDimANN(\Phi_2)=\mathfrak{l}_{\mathfrak{L}}$
	that
	\begin{equation}\label{ANNoperations:Composition}
	\begin{split}
	&\compANN{\Phi_1}{\Phi_2}=\\&
	\begin{cases} 
	\begin{array}{r}
	\big((\mathcal{W}_1, \mathfrak{B}_1),(\mathcal{W}_2, \mathfrak{B}_2),\ldots, (\mathcal{W}_{\mathfrak{L}-1},\allowbreak \mathfrak{B}_{\mathfrak{L}-1}),
	(W_1 \mathcal{W}_{\mathfrak{L}}, W_1 \mathfrak{B}_{\mathfrak{L}}+B_{1}),\\ (W_2, B_2), (W_3, B_3),\ldots,(W_{L},\allowbreak B_{L})\big)
	\end{array}
	&: L>1<\mathfrak{L} \\[3ex]
	\big( (W_1 \mathcal{W}_{1}, W_1 \mathfrak{B}_1+B_{1}), (W_2, B_2), (W_3, B_3),\ldots,(W_{L},\allowbreak B_{L}) \big)
	&: L>1=\mathfrak{L}\\[1ex]
	\big((\mathcal{W}_1, \mathfrak{B}_1),(\mathcal{W}_2, \mathfrak{B}_2),\allowbreak \ldots, (\mathcal{W}_{\mathfrak{L}-1},\allowbreak \mathfrak{B}_{\mathfrak{L}-1}),(W_1 \mathcal{W}_{\mathfrak{L}}, W_1 \mathfrak{B}_{\mathfrak{L}}+B_{1}) \big)
	&: L=1<\mathfrak{L}  \\[1ex]
	(W_1 \mathcal{W}_{1}, W_1 \mathfrak{B}_1+B_{1}) 
	&: L=1=\mathfrak{L} 
	\end{cases}
	\end{split}
	\end{equation}
	(cf.\ Definition~\ref{Def:ANN}).
\end{definition}

\begin{lemma}\label{Lemma:CompositionAssociative}
	Let 
	$\Phi_1,\Phi_2,\Phi_3\in\ANNs$
	satisfy that
	$\inDimANN(\Phi_1)=\outDimANN(\Phi_2)$ and
	$\inDimANN(\Phi_2)=\outDimANN(\Phi_3)$ 
	(cf.\ Definition~\ref{Def:ANN}).
	Then
	\begin{equation}
	\compANN{(\compANN{\Phi_1}{\Phi_2})}{\Phi_3}=\compANN{\Phi_1}{(\compANN{\Phi_2}{\Phi_3})}
	\end{equation}
	(cf.\  Definition~\ref{Definition:ANNcomposition}).
\end{lemma}

\begin{definition}[Compositions of ANNs involving artificial identities]
	\label{Definition:ANNconcatenation}
	Let $\Psi\in \ANNs$.
	Then	
	we denote by 
	\begin{equation}
	\concPsiANN{(\cdot)}{(\cdot)}\colon \{(\Phi_1,\Phi_2)\in\ANNs\times \ANNs\colon \inDimANN(\Phi_1)=\outDimANN(\Psi) \text{ and } \outDimANN(\Phi_2)=\inDimANN(\Psi)\}\allowbreak\to\ANNs
	\end{equation}
	the function which satisfies for all $\Phi_1,\Phi_2\in\ANNs$ with $\inDimANN(\Phi_1)=\outDimANN(\Psi)$ and $\outDimANN(\Phi_2)=\inDimANN(\Psi)$ that 
	\begin{equation}
	\begin{split}
	\concPsiANN{\Phi_1}{\Phi_2}
	= \compANN{\Phi_1}
	{
		(\compANN
		{\Psi}
		{\Phi_2})}
	=\compANN{(\compANN{\Phi_1}{\Psi})}{\Phi_2}
	\end{split}
	\end{equation}
	(cf.\ \Cref{Def:ANN,Definition:ANNcomposition} and \Cref{Lemma:CompositionAssociative}).
\end{definition}

%% file: DNNmonteCarloEuler.tex
\subsection{Deep ANN approximations for Monte Carlo Euler approximations}
\label{subsec:DNN:MCE}
%

\begin{prop}\label{Thm:ApproxOfEulerWithGronwall}
	Let $N, d \in \N$, $c, C \in [0,\infty)$, 
	 $T, \mathfrak{D} \in (0,\infty)$, $q\in(2,\infty)$, $\varepsilon\in (0,1]$,  $(\tau_n)_{n\in\{0,1,\dots,N\}}\allowbreak\subseteq\R$
	satisfy 
	for all $n\in\{0,1,\dots,N\}$ that $\tau_n=\tfrac{nT}{N}$
and
	\begin{equation}\label{ApproxOfEulerWithGronwall:paramBound}
	\mathfrak{D}=\big[\tfrac{720q}{(q-2)}\big][\LogBin(\eps^{-1})+q+1]-504,
	\end{equation}
	let 
	$\Phi\in \ANNs $
	satisfy 
	for all $x\in \R^d$ that $\inDimANN(\Phi)=\outDimANN(\Phi)=d$ and $\Vert (\functionANN(\Phi))(x)\Vert\le C +c\Vert x\Vert$,
	let 
	$Y= (Y^{x,y }_t)_{(t,x,y)\in [0,T]\times \R^d\times(\R^d)^N} \colon\allowbreak [0,T]\times \R^d\times (\R^d)^N \to \R^d $ 
satisfy for all 
	$n\in\{0,1,\dots,N-1\}$,
	$ t \in [\tau_{n},\tau_{n+1}]$,
	$ x \in \R^d $, $y=(y_1,y_2,\dots, y_N)\in (\R^d)^N$
	that $\affineProcess^{x,y }_0=x$ and
	\begin{equation}
	\label{ApproxOfEulerWithGronwall:Y_processes}
	\begin{split}
	&\affineProcess^{x,y }_t 
	=
	\affineProcess^{x,y }_{\tau_n}+ \left(\tfrac{tN}{T}-n\right)\!\left[\tfrac{T}{N}(\functionANN(\Phi)) ( 
	\affineProcess^{x,y }_{\tau_n} 
	)
	+
	y_{n+1}\right]\!,
	\end{split}
	\end{equation}
	and let $g_n\colon \R^d\times (\R^d)^N\to [0,\infty)$, $n\in\{0,1,\dots,N\}$,  satisfy for all $n\in\{0,1,\dots,N\}$,
	$ x \in \R^d $, $y=(y_1,y_2,\dots, y_N)\in (\R^d)^N$ that 
	\begin{equation}
	g_{n}(x,y)=\bigg[\Vert x\Vert + C \tau_n+
	\max_{m\in\{0,1,\dots,n\}}\big\Vert \smallsum_{k=1}^{m} y_{k}\big\Vert\bigg]
	\exp(c \tau_n)
	\end{equation}
	(cf.\ \Cref{Def:euclideanNorm,Def:ANN,Definition:ANNrealization,Definition:Relu1}).
	Then there exist $\Psi_{y}\in \ANNs$,  $y\in (\R^d)^N$, such that
	\begin{enumerate}[(i)]
		\item \label{ApproxOfEulerWithGronwall:Function} it holds for all  $y\in (\R^d)^N$ that $\functionANN (\Psi_{y})\in C(\R^{d+1},\R^d)$,
		\item it holds for all
		$n\in\{0,1,\dots,N-1\}$,
		$ t \in [\tau_{n},\tau_{n+1}]$,
		$x\in\R^d$, 
		$y\in (\R^d)^N$ that 
		\begin{equation}
		\Vert \affineProcess^{x,y }_t -(\functionANN (\Psi_{y}))(t,x)\Vert
		\le\varepsilon \big[2\sqrt{d}+(g_{n}(x,y))^q+(g_{n+1}(x,y))^q\big],
		\end{equation}
		\item it holds for all 
		$n\in\{0,1,\dots,N-1\}$,
		$ t \in [\tau_{n},\tau_{n+1}]$,
		$x\in\R^d$, 
		$y\in (\R^d)^N$
		that 
		\begin{equation}
		\Vert (\functionANN (\Psi_{y}))(t,x)\Vert
		\le 6\sqrt{d}+2\big[(g_{n}(x,y))^2+(g_{n+1}(x,y))^2\big],
		\end{equation}
		\item \label{ApproxOfEulerWithGronwall:ItemParams}
		it holds for all  $y\in (\R^d)^N$ that 
		\begin{equation}
		\begin{split}
		\paramANN(\Psi_{y})
		&\le \tfrac{9}{2}\, N^6 d^{16} \big[  2 (\lengthANN(\Phi) - 1)
		+
		\mathfrak{D}+(
		24
		+6\lengthANN(\Phi)
		+   [4+\paramANN(\Phi)]^{2})^{2}
		\big]^{2},
		\end{split}
		\end{equation}
		\item \label{ApproxOfEulerWithGronwall:Continuity}
		it holds  for all $t\in [0,T]$, $x\in\R^d$ that $
		[(\R^d)^N\ni y\mapsto (\functionANN (\Psi_{y}))(t,x)\in\R^d] \in C((\R^d)^N,\R^d)$, 
		and
		\item \label{ApproxOfEulerWithGronwall:Adaptedness}
		it holds  for all    $n\in\{0,1,\dots, N\}$, $t\in [0,\tau_n]$, $x\in\R^d$, $y=(y_1,y_2,\dots, y_N)$, $z=(z_1,z_2,\allowbreak\dots, z_N)\in (\R^d)^N$  with $\forall\, k\in \N\cap [0,n]\colon y_k=z_k$ that 
		\begin{equation}
		(\functionANN (\Psi_{y}))(t,x)=(\functionANN (\Psi_{z}))(t,x).
		\end{equation}
	\end{enumerate}
\end{prop}

\begin{proof}[Proof of \Cref{Thm:ApproxOfEulerWithGronwall}]	
	Throughout this proof let 
	$\Psi_{y}\in \ANNs$,  $y\in (\R^d)^N$, satisfy  that
	\begin{enumerate}[(I)]
		\item \label{ProofApproxOfEulerWithGronwall:Function} it holds for all  $y\in (\R^d)^N$  that $\functionANN (\Psi_{y})\in C(\R^{d+1},\R^d)$,
		\item\label{ProofApproxOfEulerWithGronwall:Error} it holds for all  
		$n\in\{0,1,\dots,N-1\}$,
		$ t \in [\tau_{n},\tau_{n+1}]$,
		$x\in\R^d$,
		$y\in (\R^d)^N$  that 
		\begin{equation}
		\Vert \affineProcess^{x,y }_t -(\functionANN (\Psi_{y}))(t,x)\Vert
		\le \varepsilon \big(2\sqrt{d}+\|  Y_{\tau_{n}}^{x,y }\|^q+\|  Y_{\tau_{n+1}}^{x,y }\|^q\big),
		\end{equation}
		\item \label{ProofApproxOfEulerWithGronwall:Growth} it holds for all 
		$n\in\{0,1,\dots,N-1\}$,
		$ t \in [\tau_{n},\tau_{n+1}]$,
		$x\in\R^d$,
		$y\in (\R^d)^N$  that 
		\begin{equation}
		\Vert (\functionANN (\Psi_{y}))(t,x)\Vert
		\le 6\sqrt{d}+2\big(\Vert Y_{\tau_{n}}^{x,y }\Vert^2+\Vert Y_{\tau_{n+1}}^{x,y }\Vert^2\big),
		\end{equation}
		\item \label{ProofApproxOfEulerWithGronwall:ItemParams}
		it holds for all  $y\in (\R^d)^N$ that 
		\begin{align*}
		&\paramANN(\Psi_{y}) \numberthis
		\\
		&\le\tfrac{1}{2} \Big[  6d^2N^2(\lengthANN(\Phi) - 1)
		+3N
		\big[d^2 \mathfrak{D}+(
		23
		+6N (\lengthANN(\Phi) - 1)
		+  7d^2+N [4d^2+\paramANN(\Phi)]^{2})^{2}\big]
		\Big]^{2},
		\end{align*}
		\item \label{ProofApproxOfEulerWithGronwall:Continuity}
		it holds  for all  $t\in [0,T]$, $x\in\R^d$ that 
		$
		[(\R^d)^N\ni y\mapsto (\functionANN (\Psi_{y}))(t,x)\in\R^d]\in C((\R^d)^N,\R^d)
		$, 
		and
		\item \label{ProofApproxOfEulerWithGronwall:Adaptedness}
		it holds  for all    $n\in\{0,1,\dots, N\}$, $t\in [0,\tau_n]$, $x\in\R^d$, $y=(y_1,y_2,\dots, y_N),\allowbreak z=(z_1,z_2,\allowbreak\dots,  z_N)\in (\R^d)^N$  with $\forall\, k\in \N\cap [0,n]\colon y_k=z_k$  that 
		\begin{equation}
		(\functionANN (\Psi_{y}))(t,x)=(\functionANN (\Psi_{z}))(t,x)
		\end{equation}
	\end{enumerate}
	(cf.\ Grohs et al.~\cite[Proposition~3.10]{GrohsHornungEtAl2023} (applied with $N \is N$, $ d \is d$, $a \is \mathfrak{r}$, $T \is T$, $t_0 \is \tau_0$, $t_1 \is \tau_1$, $\dots$, $t_N \is \tau_N$,  
	$\mathfrak{D} \is \mathfrak{D}$, $\varepsilon \is \varepsilon$, $q \is q$, $Y \is Y$
	in the notation of Grohs et al.~\cite[Proposition~3.10]{GrohsHornungEtAl2023})).
	Note that 	\eqref{ProofApproxOfEulerWithGronwall:ItemParams} ensures that  for all  $y\in (\R^d)^N$ it holds  that 
	\begin{equation}
	\begin{split}
	&\paramANN(\Psi_{y})
	\\&\le\tfrac{1}{2} \Big[  6d^2N^2  (\lengthANN(\Phi) - 1)
	+3N
	\big[d^2 \mathfrak{D}+(
	23
	+6N  (\lengthANN(\Phi) - 1)
	+  7d^2+N d^4 [4+\paramANN(\Phi)]^{2})^{2}\big]
	\Big]^{2}
		\\&\le\tfrac{1}{2} \Big[  6d^2N^2 (\lengthANN(\Phi) - 1)
	+3N
	\big[d^2 \mathfrak{D}+N^2 d^8(
	23
	+6(\lengthANN(\Phi) -1) + 7
	+   [4+\paramANN(\Phi)]^{2})^{2}\big]
	\Big]^{2} 
	\\&=\tfrac{1}{2} \Big[  6d^2N^2 (\lengthANN(\Phi) - 1)
	+3N
	\big[d^2 \mathfrak{D}+N^2 d^8(
	24
	+6\lengthANN(\Phi)
	+   [4+\paramANN(\Phi)]^{2})^{2}\big]
	\Big]^{2}.
	\end{split}
	\end{equation}
	Hence, we obtain that   for all  $y\in (\R^d)^N$ it holds that 
	\begin{equation}\label{ProofApproxOfEulerWithGronwall:EstimateParams}
	\begin{split}
	\paramANN(\Psi_{y})
	&\le\tfrac{1}{2} \Big[  6d^2N^2 (\lengthANN(\Phi) - 1)
	+3N^3 d^8
	\big[ \mathfrak{D}+(
24
	+6\lengthANN(\Phi)
	+   [4+\paramANN(\Phi)]^{2})^{2}\big]
	\Big]^{2}
	\\&\le\tfrac{9}{2}\, N^6 d^{16} \big[  2 (\lengthANN(\Phi) - 1)
	+
	\mathfrak{D}+(
	24
	+6\lengthANN(\Phi)
	+   [4+\paramANN(\Phi)]^{2})^{2}
	\big]^{2}.
	\end{split}
	\end{equation}
	In addition, observe that, e.g., Grohs et al.~\cite[Lemma~3.11]{GrohsHornungEtAl2023} (applied with $N \is N$, $d \is d$, $c \is c$, $C \is C$, $A_1 \is \frac{T}{N} \id_{\R^d} $, $A_2 \is \frac{T}{N} \id_{\R^d}$, $\dots$, $A_N \is \frac{T}{N} \id_{\R^d}$, $\mu \is \functionANN(\Phi)$, $Y_0 \is  (Y^{x,y }_{\tau_0})_{(x,y)\in \R^d\times(\R^d)^N}$, $Y_1 \is  (Y^{x,y }_{\tau_1})_{(x,y)\in \R^d\times(\R^d)^N}$, $\dots$, $Y_N \is  (Y^{x,y }_{\tau_N})_{(x,y)\in \R^d\times(\R^d)^N}$  in the notation of Grohs et al.~\cite[Lemma~3.11]{GrohsHornungEtAl2023}) and the hypothesis that  $\forall \,  n\in\{0,1,\dots,N\} \colon \tau_n=\tfrac{nT}{N}$ demonstrate that 	 for all $n\in\{0,1,\dots,N\}$,
	$ x \in \R^d $, $y=(y_1,y_2,\dots, y_N)\in (\R^d)^N$ it holds that 
	\begin{equation}\label{ApproxOfEuler:GronwallEstimate}
	\begin{split}
	\Vert \affineProcess^{x,y} _{\tau_n} \Vert
	&\le \bigg[\Vert x\Vert + \tfrac{C nT}{N}+
	\max_{m\in\{0,1,\dots,n\}}\big\Vert \smallsum_{k=1}^{m} y_{k}\big\Vert\bigg]
	\exp\!\left(\tfrac{c nT}{N}\right)
	\\&=\bigg[\Vert x\Vert + C \tau_n+
	\max_{m\in\{0,1,\dots,n\}}\big\Vert \smallsum_{k=1}^{m} y_{k}\big\Vert\bigg]
	\exp(c \tau_n)
	=g_n(x,y).
	\end{split}
	\end{equation}
	Combining this with  \eqref{ProofApproxOfEulerWithGronwall:Error} and \eqref{ProofApproxOfEulerWithGronwall:Growth} ensures that for all
	$n\in\{0,1,\dots,N-1\}$,
	$ t \in [\tau_{n},\tau_{n+1}]$,
	$x\in\R^d$, $y\in (\R^d)^N$ it holds that 
	\begin{equation}
	\begin{split}
	\Vert \affineProcess^{x,y }_t -(\functionANN (\Psi_{y}))(t,x)\Vert
	&\le\varepsilon \big[2\sqrt{d}+\|  Y_{\tau_{n}}^{x,y }\|^q+\|  Y_{\tau_{n+1}}^{x,y }\|^q\big]
	\\&\le\varepsilon \big[2\sqrt{d}+(g_{n}(x,y))^q+(g_{n+1}(x,y))^q\big]
	\end{split}
	\end{equation}
	and
	\begin{equation}
	\begin{split}
	\Vert (\functionANN (\Psi_{y}))(t,x)\Vert
	&\le 6\sqrt{d}+2\big(\Vert Y_{\tau_{n}}^{x,y }\Vert^2+\Vert Y_{\tau_{n+1}}^{x,y }\Vert^2\big)
	\\&\le 6\sqrt{d}+2\big[(g_{n}(x,y))^2+(g_{n+1}(x,y))^2\big].
	\end{split}
	\end{equation}
	This, \eqref{ProofApproxOfEulerWithGronwall:Function},  \eqref{ProofApproxOfEulerWithGronwall:Continuity}, \eqref{ProofApproxOfEulerWithGronwall:Adaptedness}, and \eqref{ProofApproxOfEulerWithGronwall:EstimateParams} establish items~\eqref{ApproxOfEulerWithGronwall:Function}--\eqref{ApproxOfEulerWithGronwall:Adaptedness}.
	The proof of \Cref{Thm:ApproxOfEulerWithGronwall} is thus completed.
\end{proof}

\begin{prop}\label{Cor:ApproxOfMCSum}
	Let $M,N, d,\mathfrak{d} \in \N$,  $\alpha, c, C, \mathfrak{C} \in[0,\infty)$, $T, \mathfrak{D} \in (0,\infty)$,
$q\in(2,\infty)$, $\varepsilon\in (0,1]$, $\drift, \initial \in \ANNs$ satisfy  that  $\inDimANN(\drift)=\outDimANN(\drift)= \inDimANN(\initial)=d$, $\outDimANN(\initial)=\mathfrak{d}$, and
		\begin{equation}\label{ApproxOfMCSum:paramBound}
		\mathfrak{D}= \big[\tfrac{720q}{q-2} \big] [\LogBin(\eps^{-1})+q+1]-504,
		\end{equation}	
assume for all $x,y\in\R^d$ that $\|(\functionANN(\drift))(x)\Vert\le C + c\Vert x\Vert$ and
\begin{align}
	\label{ApproxOfMCsum:Lipschitz}
	\Vert (\functionANN(\initial)) (x) - (\functionANN(\initial)) (y)\Vert \leq \mathfrak{C} (1 + \Vert x\Vert^{\alpha} + \Vert y \Vert^{\alpha})\Vert  x-y\Vert,
\end{align} 
	let $ ( \Omega, \mathcal{F}, \P ) $ be a probability space,
	let $W^m=(W_n^m)_{n\in\{0,1,\dots,N\}}\colon \{0,1,\dots,N\}\times  \Omega\to \R^d$, $m\in\{1,2,\dots,M\}$, be stochastic processes which satisfy for all $m\in\{1,2,\dots,M\}$ that $W_0^m=0$,
	 let 
	$Y^m= (Y^{m,x }_t (\omega))_{(t, x, \omega) \in [0,T] \times \R^d \times \Omega} \colon\allowbreak [0,T]\times \R^d\times\Omega \to \R^d $, $m\in\{1,2,\dots,\allowbreak M\}$, 
 satisfy for all 
	$m\in\{1,2,\dots,\allowbreak M\}$,
	$ x \in \R^d $,
	$n\in\{0,1,\dots,N-1\}$,
	$ t \in \big[\tfrac{nT}{N},\tfrac{(n+1)T}{N}\big]$
	that $\affineProcess^{m,x }_0=x$ and
	\begin{equation}
	\label{ApproxOfMCSum:Y_processes}
	\begin{split}
	&\affineProcess^{m,x }_t
	=
	\affineProcess^{m,x }_{\frac{nT}{N}}+ \left(\tfrac{tN}{T}-n\right)\left[\tfrac{T}{N}(\functionANN(\drift)) ( 
	\affineProcess^{m,x }_{\frac{nT}{N}} 
	)
	+
	W_{n+1}^{m}-W_{n}^{m}\right],
	\end{split}
	\end{equation}
			and let 
			$h_{m,r}\colon \R^d\times \Omega\to [0,\infty)$,  $m\in\{1,2,\dots, M\}$, $r\in (0,\infty)$, satisfy for all  $m\in\{1,2,\dots, M\}$, $r\in (0,\infty)$,
			$ x \in \R^d $ that 
			\begin{equation}\label{ApproxOfMCSum:AuxFunctionTwo}
			h_{m,r}(x)=1+ \left[\Vert x\Vert + C T+
			\max_{n\in\{0,1,\dots,N\}}\Vert  W_n^m\Vert\right]^r
			\exp(rc T)
			\end{equation}
	(cf.\ \Cref{Def:euclideanNorm,Def:ANN,Definition:ANNrealization,Definition:Relu1}).
	Then there exists $(\Psi_{\omega})_{\omega\in \Omega}\subseteq \ANNs$ such that
	\begin{enumerate}[(i)]
		\item \label{ApproxOfMCsum:Function} it holds for all   $\omega\in \Omega$ that $\functionANN (\Psi_{\omega})\in C(\R^{d+1},\R^\mathfrak{d})$,
		\item it holds for all  
		$ t \in [0,T]$,
		$x\in\R^d$, $\omega\in \Omega$ that 
	\begin{equation}
	\begin{split}
	&\left\Vert(\functionANN(\Psi_{\omega}))(t,x)-\frac{1}{M}\left[\smallsum\limits_{m=1}^M (\functionANN(\initial))(Y_t^{m,x}(\omega))\right]\right\Vert 
	\\&\le \frac{2\varepsilon \mathfrak{C} \sqrt{d}}{M}\bigg[\smallsum\limits_{m=1}^M   \left[1 +2 {d}^{\nicefrac{\alpha}{2}} 6^\alpha
	|h_{m,2}(x,\omega)|^\alpha \right]
	h_{m,q}(x,\omega) \bigg],
	\end{split}
	\end{equation}
		\item \label{ApproxOfMCsum:ItemParams}
		it holds for all  $\omega\in \Omega$ that 
				\begin{equation}
				\begin{split}
				\paramANN(\Psi_\omega)
				&\le 2 M^2 \paramANN(\initial) + 9 M^2 N^6 d^{16} \big[ 2 \lengthANN(\drift)
				+
				 \mathfrak{D}+(
				24
				+6\lengthANN(\drift)
				+   [4+\paramANN(\drift) ]^{2} )^{2}
				\big]^{2},
				\end{split}
				\end{equation}
		and
		\item \label{ApproxOfMCsum:Continuity}
		it holds  for all  $t\in [0,T]$, $x\in\R^d$ that $\Omega\ni \omega\mapsto (\functionANN (\Psi_{\omega}))(t,x)\in\R^\mathfrak{d}$
		is measurable.
	\end{enumerate}
\end{prop}

\begin{proof}[Proof of Proposition~\ref{Cor:ApproxOfMCSum}]	
	Throughout this proof  
	let $\tau_0, \tau_1, \ldots, \tau_N \in \R$
	satisfy 
	for all $n\in\{0,1, \allowbreak \dots,N\}$ that $\timeGrid_n=\tfrac{nT}{N}$,
	let 
	$g_{m}\colon \R^d\times \Omega\to [0,\infty)$,  $m\in\{1,2,\dots, M\}$,  satisfy for all  $m\in\{1,2,\dots, M\}$,
	$ x \in \R^d $ that 
	\begin{equation}\label{ApproxOfMCSum:AuxFunctionOne}
	g_{m}(x)=\left[\Vert x\Vert + CT+
	\max_{n\in\{0,1,\dots,N\}}\big\Vert  \smallsum_{l=1}^{n} \big(W_{l}^m-W_{l-1}^m\big)\big\Vert\right]
	\exp(c T),
	\end{equation}
	let $(\Psi_{\omega,m})_{\omega\in\Omega,m\in\{1,2,\dots, M\}}\allowbreak\subseteq \ANNs$ satisfy that
		\begin{enumerate}[(I)]
			\item \label{ApproxOfMCsumProof:Function} it holds for all   $m\in\{1,2,\dots, M\}$, $\omega\in \Omega$ that $\functionANN (\Psi_{\omega,m})\in C(\R^{d+1},\R^d)$,
			\item \label{ApproxOfMCsumProof:EstimateDifference}  it holds for all    $m\in\{1,2,\dots, M\}$, 
			$ t \in [0,T]$,
			$x\in\R^d$, $\omega\in \Omega$ that 
			\begin{equation}
			\Vert \affineProcess^{m, x }_t (\omega) -(\functionANN (\Psi_{\omega,m}))(t,x)\Vert
			\le 2\varepsilon \sqrt{d}\, \big[1+(g_{m}(x,\omega))^q\big],
			\end{equation}
			\item \label{ApproxOfMCsumProof:EstimateGrowth}  it holds for all   $m\in\{1,2,\dots, M\}$, 
			$ t \in [0,T]$,
			$x\in\R^d$, $\omega\in \Omega$
			that 
			\begin{equation}
			\Vert (\functionANN (\Psi_{\omega,m}))(t,x)\Vert
			\le  6 \sqrt{d}\, \big[1+(g_{m}(x,\omega))^2\big],
			\end{equation}
			\item \label{ApproxOfMCsumProof:ItemParams}
			 it holds
			 for all   $m\in\{1,2,\dots, M\}$, $\omega\in \Omega$  that 
						\begin{equation}
						\begin{split}
						\paramANN(\Psi_{\omega,m})
						&\le\tfrac{9}{2}\, N^6 d^{16} \big[  2 \lengthANN(\drift)
						+
						 \mathfrak{D}+(
						24
						+6\lengthANN(\drift)
						+   [4+\paramANN(\drift)]^{2})^{2}
						\big]^{2},
						\end{split}
						\end{equation}
			and
			\item \label{ApproxOfMCsumProof:Continuity}
			 it holds
		  for all $m\in\{1,2,\dots, M\}$,  $t\in [0,T]$, $x\in\R^d$ that $\Omega\ni \omega\mapsto (\functionANN (\Psi_{\omega,m}))(t,x)\in\R^d$
			is measurable
		\end{enumerate}
		(cf.\ \Cref{Thm:ApproxOfEulerWithGronwall} (applied with $N \is N$, $d \is d$, $c \is c$, $C \is C$, $T \is T$, $\mathfrak{D} \is \mathfrak{D}$, $ q \is q$, $\varepsilon \is \varepsilon$, $\tau_0 \is \tau_0$, $\tau_1 \is \tau_1, \ldots, \tau_N \is \tau_N$, $\Phi \is \drift$ in the notation of  \Cref{Thm:ApproxOfEulerWithGronwall})),
let $\mathbb{I}\in \ANNs$ satisfy for all $x\in\R^d$ that $\functionANN(\mathbb{I})\in C(\R^d,\R^d)$, $\dims(\mathbb{I})=(d,2d,d)$, and $(\functionANN(\mathbb{I}))(x)=x$ (cf.~\cite[Lemma~5.5]{JentzenSalimovaWelti2018}),
	let $(\psi_{\omega,m})_{\omega\in\Omega,m\in\{1,2,\dots, M\}}\subseteq \ANNs$ satisfy for all $m\in\{1,2,\dots, M\}$, $\omega\in\Omega$ that $\psi_{\omega,m}=\concANN{\initial}{\Psi_{\omega,m}}$ (cf.\ Definition~\ref{Definition:ANNconcatenation}),
	and let $(\Phi_{\omega})_{\omega\in\Omega}\subseteq \ANNs$ satisfy that
	\begin{enumerate}[(A)]
		\item \label{ApproxOfMCsum:MCsumItemOne}   it holds for all $\omega\in\Omega$  that $\functionANN(\Phi_{\omega})\in C(\R^{\inDimANN(\psi_{\omega,1})},\R^{\outDimANN(\psi_{\omega,1})})$,
		\item\label{ApproxOfMCsum:MCsumItemTwo}   it holds for all $\omega\in\Omega$ that $\paramANN(\Phi_\omega)\le M^2 \paramANN(\psi_{\omega,1})$, and 
		\item\label{ApproxOfMCsum:MCsumItemThree}  it holds for all  $t\in\R$, $x\in\R^{d}$, $\omega\in\Omega$  that
			\begin{equation}
			\begin{split}
			(\functionANN (\Phi_\omega))(t,x)=\frac{1}{M}\smallsum\limits_{m=1}^M (\functionANN (\psi_{\omega,m}))(t,x)
			\end{split}
			\end{equation}
	\end{enumerate}
	(cf.\ Grohs et al.~\cite[Proposition 2.25]{GrohsHornungEtAl2023}).
	Note that \eqref{ApproxOfMCsum:MCsumItemTwo}, \eqref{ApproxOfMCsumProof:ItemParams}, Grohs et al.~\cite[item~(iii) in Proposition 2.16]{GrohsHornungEtAl2023}, and the fact that $\dims(\mathbb{I})=(d,2d,d)$ demonstrate that 
	for all $\omega\in\Omega$ it holds that
		\begin{equation}\label{ApproxOfMCsum:ParamEstimate}
		\begin{split}
		\paramANN(\Phi_\omega) &\le M^2 \paramANN(\psi_{\omega,1})
		\le 2 M^2 [\paramANN(\initial)+\paramANN(\Psi_{\omega,1})]
		\\&\le  2 M^2 \paramANN(\initial) + 9 M^2 N^6 d^{16} \big[  2 \lengthANN(\drift)
		+
		 \mathfrak{D}+(
		24
		+6\lengthANN(\drift)
		+   [4+\paramANN(\drift)]^{2})^{2}
		\big]^{2}.
		\end{split}
		\end{equation}
	Moreover, observe that \eqref{ApproxOfMCsum:MCsumItemOne}, \eqref{ApproxOfMCsum:MCsumItemThree}, and Grohs et al.~\cite[item~(iv) in Proposition 2.16]{GrohsHornungEtAl2023}   imply that for all  $t\in\R$, $x\in\R^{d}$, $\omega\in\Omega$ it holds that $\functionANN (\Phi_\omega)\in C(\R^{d+1},\R^\mathfrak{d})$ and
	\begin{equation}\label{ApproxOfMCsum:PhiSum}
	\begin{split}
	(\functionANN (\Phi_\omega))(t,x)=\frac{1}{M}\smallsum\limits_{m=1}^M (\functionANN (\initial))((\functionANN(\Psi_{\omega,m}))(t,x)).
	\end{split}
	\end{equation} 
	Next note that the fact that  $ \forall \, m\in\{1,2,\dots, M\} \colon W_0^m=0$ ensures that for all $n\in\{1,2,\dots, N\}$, $m\in\{1,2,\dots, M\}$ it holds that
		\begin{equation}\label{telescopicSum}
			\begin{split}
			\smallsum\limits_{l=1}^{n} \big(W_{l}^m-W_{l-1}^m\big)=W_{n}^m-W_{0}^m=W_{n}^m.
			\end{split}
		\end{equation}
	Combining this with \eqref{ApproxOfMCSum:AuxFunctionOne} proves that 
	for all  $m\in\{1,2,\dots, M\}$,
	$ x \in \R^d $ it holds that 
	\begin{equation}\label{ApproxOfMCSum:AuxFunctionOneTilde}
	g_{m}(x)=\left[\Vert x\Vert + CT+
	\max_{n\in\{0,1,\dots,N\}}\Vert W_{n}^m \Vert\right]
	\exp(cT).
	\end{equation}
	In addition, observe that \eqref{ApproxOfMCSum:Y_processes} and the fact that
	$\forall \, n\in\{0,1,\dots,N\} \colon \timeGrid_n=\tfrac{nT}{N}$
	assure that for all    $m\in\{1,2,\dots, M\}$,  $n\in\{0, 1, \dots,N-1\}$,
	$ x \in \R^d $ it holds that
	\begin{equation}
	\begin{split}
 \affineProcess^{m,x} _{\timeGrid_{n+1}}  &= \affineProcess^{m,x} _{\timeGrid_{n}}  + \tfrac{T }{N} (\functionANN(\drift)) ( 
 \affineProcess^{m,x }_{\timeGrid_n}
 ) + W_{n+1}^{m}-W_{n}^{m}.
	\end{split}
	\end{equation}
	Induction and \eqref{telescopicSum} hence show
 that for all    $m\in\{1,2,\dots, M\}$,  $n\in\{0, 1, \dots,N-1\}$,
$ x \in \R^d $ it holds that
\begin{equation}
\begin{split}
\affineProcess^{m,x} _{\timeGrid_{n+1}}  &= \affineProcess^{m,x} _{\timeGrid_{0}}  + \tfrac{T }{N} \left[\smallsum\limits_{l=0}^n (\functionANN(\drift)) ( 
\affineProcess^{m,x }_{\timeGrid_l} 
) \right] + \left[\smallsum\limits_{l=0}^n ( W_{l+1}^{m}-W_{l}^{m} ) \right]\\
& = x + \tfrac{T }{N} \left[\smallsum\limits_{l=0}^n (\functionANN(\drift)) ( 
\affineProcess^{m,x }_{\timeGrid_l} 
) \right] +  W_{n+1}^{m}.
\end{split}
\end{equation}
This and the assumption that $ \forall \, x \in \R^d \colon \|(\functionANN(\drift))(x)\Vert\le C + c\Vert x\Vert$ establish that for all    $m\in\{1,2,\dots, M\}$,  $n\in\{0, 1, \dots,N\}$,
$ x \in \R^d $ it holds that
\begin{equation}
\begin{split}
\Vert  \affineProcess^{m,x} _{\timeGrid_{n}}  \Vert &\le \|x\| +  \tfrac{T }{N} \left[ \smallsum\limits_{l=0}^{n-1} \| (\functionANN(\drift)) ( 
\affineProcess^{m,x }_{\timeGrid_l} 
) \| \right] + \| W_{n}^{m} \|\\
& \le \|x\| +  \tfrac{ T}{N} \left[\smallsum\limits_{l=0}^{n-1} ( C + c\Vert  \affineProcess^{m,x} _{\timeGrid_{l}}  \Vert ) \right] + \| W_{n}^{m} \|\\
& \le \|x\| + C T + \left[\max_{k\in\{0,1,\dots,N\}}\Vert  W_{k}^m\Vert \right] + \tfrac{c T}{N} \left[ \smallsum\limits_{l=0}^{n-1} \Vert  \affineProcess^{m,x} _{\timeGrid_{l}}  \Vert \right].
\end{split}
\end{equation}
The time-discrete Gronwall inequality, e.g., in  Hutzenthaler et al.~\cite[Lemma~2.1]{HutzenthalerPricing2020} (applied with $N \is N$, $\alpha \is( \|x\| +C T + \max_{k\in\{0,1,\dots,N\}}\Vert  W_{k}^m(\omega)\Vert)$, $\beta_0 \is \frac{c T}{N}$, $\beta_1 \is \frac{cT}{N}$, $\dots$, $\beta_{N-1} \is \frac{c T}{N}$, $\epsilon_0 \is \|x\|$, $\epsilon_1 \is  \Vert  \affineProcess^{m,x} _{\timeGrid_{1}}(\omega)  \Vert$, $\dots$, $\epsilon_N \is  \Vert  \affineProcess^{m,x} _{\timeGrid_{N}}(\omega)  \Vert$ for  $m\in\{1,2,\dots, M\}$,  
$ x \in \R^d $, $\omega\in\Omega$ in the notation of  Hutzenthaler et al.~\cite[Lemma~2.1]{HutzenthalerPricing2020}) and \eqref{ApproxOfMCSum:AuxFunctionOneTilde} 
 hence demonstrate that 	 for all $m\in\{1,2,\dots, M\}$,  $n\in\{0,1,\dots,N\}$,
		$ x \in \R^d $ it holds that 
		\begin{equation}\label{Gronwall:GridpointEstimate}
		\begin{split}
		&\Vert \affineProcess^{m,x} _{\timeGrid_n} \Vert
\le\left[\Vert x\Vert + C T+
		\max_{k\in\{0,1,\dots,N\}}\Vert  W_{k}^m\Vert \right]
		\exp(c T)=g_{m}(x)
		.
		\end{split}
		\end{equation}
		In addition,
		note that \eqref{ApproxOfMCSum:Y_processes} and the fact that  $\forall \, n\in \{0,1,\dots, N-1\}, t\in [\timeGrid_n,\timeGrid_{n+1}] \colon \tfrac{t-\timeGrid_n}{\timeGrid_{n+1}-\timeGrid_n}=\tfrac{tN}{T}-n$  
		ensure that for all $m\in\{1,2,\dots, M\}$, $n\in \{0,1,\dots, N-1\}$, $t\in [\timeGrid_n,\timeGrid_{n+1}]$, $x\in\R^d$  it holds that
		\begin{align*}
		&
		\affineProcess^{m,x}_{\timeGrid_{n+1}} \big[\tfrac{t-\timeGrid_n}{\timeGrid_{n+1}-\timeGrid_n}\big]+\affineProcess^{m,x}_{\timeGrid_n}\big[1- \tfrac{t-\timeGrid_n}{\timeGrid_{n+1}-\timeGrid_n}\big]
				\\&= 
				\big( \affineProcess^{m,x }_{\timeGrid_n}+ \big[\tfrac{\timeGrid_{n+1}N}{T}-n\big]\big[\tfrac{T}{N}(\functionANN(\drift)) ( 
						\affineProcess^{m,x }_{\timeGrid_n} 
						)
						+
						W_{n+1}^{m}-W_{n}^{m}\big]\big)
				\tfrac{t-\timeGrid_n}{\timeGrid_{n+1}-\timeGrid_n}
+\affineProcess^{m,x}_{\timeGrid_n}\big[1- \tfrac{t-\timeGrid_n}{\timeGrid_{n+1}-\timeGrid_n}\big]
		\\&= 		\left(\affineProcess^{m,x}_{\timeGrid_n}+ \big[\tfrac{T}{N}(\functionANN(\drift))( 
		\affineProcess^{m,x}_{\timeGrid_n} 
		)
		+
		W_{n+1}^m-W_{n}^m\big]\right)
		\tfrac{t-\timeGrid_n}{\timeGrid_{n+1}-\timeGrid_n}
+ 
		\affineProcess^{m,x}_{\timeGrid_n}\big[1- \tfrac{t-\timeGrid_n}{\timeGrid_{n+1}-\timeGrid_n}\big]
		\\&= \affineProcess^{m,x}_{\timeGrid_n}+  \big[\tfrac{T}{N}(\functionANN(\drift))( 
		\affineProcess^{m,x}_{\timeGrid_n}
		)
		+
			W_{n+1}^m-W_{n}^m\big]
		\tfrac{t-\timeGrid_n}{\timeGrid_{n+1}-\timeGrid_n}
		\\&=Y_t^{m,x}. \numberthis
		\end{align*}
		Combining this with \eqref{Gronwall:GridpointEstimate} implies that for all $m\in\{1,2,\dots, M\}$, $n\in \{0,1,\dots, N-1\}$, $t\in [\timeGrid_n,\timeGrid_{n+1}]$, $x\in\R^d$ it holds that
		\begin{equation}\label{ApproxOfMCsum:EulerEstimate}
		\begin{split}
		&\Vert Y_t^{m,x}\Vert 
		\le \Vert \affineProcess^{m,x}_{\timeGrid_n}\Vert \big[1- \tfrac{t-\timeGrid_n}{\timeGrid_{n+1}-\timeGrid_n}\big]+
		\Vert \affineProcess^{m,x}_{\timeGrid_{n+1}}\Vert \big[\tfrac{t-\timeGrid_n}{\timeGrid_{n+1}-\timeGrid_n}\big]
		\\&\le \max\!\big\{\Vert \affineProcess^{m,x}_{\timeGrid_n} \Vert,\Vert \affineProcess^{m,x}_{\timeGrid_{n+1}} \Vert\big\} \le g_{m}(x)
\le 1+|g_{m}(x)|^2
		.
		\end{split}
		\end{equation}	
	This, \eqref{ApproxOfMCsum:Lipschitz}, \eqref{ApproxOfMCsumProof:EstimateDifference}, and \eqref{ApproxOfMCsumProof:EstimateGrowth} ensure that for all  $m\in\{1,2,\dots, M\}$,	 $t\in [0,T]$, $x\in\R^d$, $\omega\in\Omega$ it holds that 
	\begin{equation}
	\begin{split}
	&\norm{(\functionANN (\initial))((\functionANN(\Psi_{\omega,m}))(t,x))-(\functionANN (\initial))(Y_t^{m,x}(\omega))}
	\\&\le \mathfrak{C} \big(1 + \Vert(\functionANN(\Psi_{\omega,m}))(t,x)\Vert^{\alpha} + \Vert Y_t^{m,x}(\omega) \Vert^{\alpha}\big)\Vert(\functionANN(\Psi_{\omega,m}))(t,x)-Y_t^{m,x}(\omega)\Vert
	\\&\le  \mathfrak{C} \left[1 + 6^\alpha {d}^{\nicefrac{\alpha}{2}}
	(1+|g_{m}(x,\omega)|^2)^{\alpha} + (1+|g_{m}(x,\omega)|^2)^{\alpha}\right]
\, 2\varepsilon \sqrt{d}\left(1+|g_{m}(x,\omega)|^q\right)
\!.
	\end{split}
	\end{equation}
	Combining this, \eqref{ApproxOfMCSum:AuxFunctionTwo}, and \eqref{ApproxOfMCSum:AuxFunctionOneTilde} demonstrates that for all  $m\in\{1,2,\dots, M\}$, $t\in [0,T]$, $x\in\R^d$, $\omega\in\Omega$ it holds that 
		\begin{equation}\label{ApproxOfMCsum:FinalOne}
		\begin{split}
		&\norm{(\functionANN (\initial))((\functionANN(\Psi_{\omega,m}))(t,x))-(\functionANN (\initial))(Y_t^{m,x}(\omega))}
			\\&\le 2\varepsilon \mathfrak{C} \sqrt{d}  \left[1 +2 {d}^{\nicefrac{\alpha}{2}} 6^\alpha
			(1+|g_{m}(x,\omega)|^2)^{\alpha}\right]
(1+|g_{m}(x,\omega)|^q)
		\\&= 2\varepsilon \mathfrak{C} \sqrt{d}  \left[1 +2 {d}^{\nicefrac{\alpha}{2}} 6^\alpha
		|h_{m,2}(x,\omega)|^\alpha \right]
		 h_{m,q}(x,\omega).
		\end{split}
		\end{equation}
	This  and \eqref{ApproxOfMCsum:PhiSum} show that for all 	 $t\in [0,T]$, $x\in\R^d$, $\omega\in\Omega$ it holds that 
	\begin{equation}\label{ApproxOfMCsum:FinalTwo}
	\begin{split}
	&\left\Vert(\functionANN(\Phi_\omega))(t,x)-\frac{1}{M}\left[\smallsum\limits_{m=1}^M (\functionANN(\initial))(Y_t^{m,x}(\omega))\right]\right\Vert 
	\\&= \left\Vert\frac{1}{M}\left[\smallsum\limits_{m=1}^M (\functionANN (\initial))((\functionANN(\Psi_{\omega,m}))(t,x))\right]-\frac{1}{M}\left[\smallsum\limits_{m=1}^M (\functionANN(\initial))(Y_t^{m,x}(\omega))\right]\right\Vert
	\\&\le \frac{1}{M}\left[\smallsum\limits_{m=1}^M \norm{ (\functionANN (\initial))((\functionANN(\Psi_{\omega,m}))(t,x))-(\functionANN(\initial))(Y_t^{m,x}(\omega))} \right]
	\\&\le \frac{2\varepsilon \mathfrak{C} \sqrt{d}}{M}\bigg[\smallsum\limits_{m=1}^M   \left[1 +2 {d}^{\nicefrac{\alpha}{2}} 6^\alpha
	|h_{m,2}(x,\omega)|^\alpha \right]
	h_{m,q}(x,\omega) \bigg]\!.
	\end{split}
	\end{equation}
	Moreover, observe that \eqref{ApproxOfMCsumProof:Continuity}, 
	the fact that $\functionANN(\initial)$ is continuous,
	and \eqref{ApproxOfMCsum:PhiSum} ensure that 
	for all $t\in[0,T]$, $x\in\R^d$ it holds that $\Omega\ni \omega\mapsto (\functionANN (\Phi_{\omega}))(t,x)\in\R^{\mathfrak{d}}$ is measurable.
	Combining this with \eqref{ApproxOfMCsum:ParamEstimate}, \eqref{ApproxOfMCsum:FinalTwo}, and the fact that  $\forall \, \omega\in\Omega \colon \functionANN (\Phi_\omega)\in C(\R^{d+1},\R^{\mathfrak{d}})$
	establishes items~\eqref{ApproxOfMCsum:Function}--\eqref{ApproxOfMCsum:Continuity}.
	This completes the proof of Proposition~\ref{Cor:ApproxOfMCSum}.
\end{proof}

\begin{lemma}\label{Lemma:LpBoundsGronwall}
	Let $N, d\in \N$, $T\in (0,\infty)$, $\alpha, c,  C, \mathcal{C} \in[0,\infty)$,  $p\in (1,\infty)$, 
	let $\nu \colon \mathcal{B}([0,T]\times\R^d) \to [0,\infty)$ be a finite measure 	which  satisfies that
	\begin{equation}\label{LpBoundsGronwall:MomentBound}
	\mathcal{C}=\max\!\left\{\left[\int_{[0,T]\times\R^d}\Vert x\Vert^{\max\{4p\alpha,6p\}} \, \nu(dt,dx)\right]^{\!\nicefrac{1}{\max\{4p\alpha,6p\}}},1\right\},
	\end{equation}	
	let $ ( \Omega, \mathcal{F}, \P,(\mathbb{F}_n)_{n\in\{0,1,\dots, N\}} ) $ be a filtered probability space,
	let $\mathcal{M}=(\mathcal{M}_n)_{n\in\{0,1,\dots,N\}}\colon\allowbreak \{0,1,\dots,N\}\times  \Omega\to [0,\infty)$  be an $(\mathbb{F}_n)_{n\in\{0,1,\dots, N\}}$-submartingale which satisfies for all $n\in\{0,1,\dots,N\}$ that $ \E [ |\mathcal{M}_n|^{\max\{4p\alpha,6p\}}] < \infty$, 
	and let 
	$h_r\colon \R^d\times \Omega\to [0,\infty)$, $r\in (0,\infty)$,   satisfy for all $r\in (0,\infty)$,
	$ x \in \R^d $ that 
	\begin{equation}\label{LpBoundsGronwall:AuxFunctionOne}
	h_r(x)=1+e^{r c}\left[\Vert x\Vert + C+
	\max_{n\in\{0,1,\dots,N\}} \mathcal{M}_n\right]^{r} 
	\end{equation}
	(cf.\ Definition~\ref{Def:euclideanNorm}).
	Then 
	\begin{enumerate}[(i)]
		\item \label{LpBoundsGronwall:ItemOne}
	it holds for all $q,r\in (0,\infty)$ with  $ 1 < qr \leq \max\{4\alpha p,6p\}$ that
				\begin{align*}
				&\left[\int_{[0,T]\times\R^d}\E\big[\vert  h_r(x)\vert^{q}\big] \, \nu(dt,dx)\right]^{\!\nicefrac{1}{q}} \numberthis
				\\&\le 2e^{r c} \max\!\big\{2^{(\nicefrac{1}{q}) - 1},1\big\} \left[\mathcal{C}   +C+
				\tfrac{qr}{qr-1} \big|\E\big[\vert  \mathcal{M}_N\vert^{qr}\big]\big|^{\nicefrac{1}{qr}} \right]^{r}
				\max\!\big\{1,[\nu([0,T]\times\R^d)]^{\nicefrac{1}{q}}\big\}
				\end{align*}
	and
	\item \label{LpBoundsGronwall:ItemTwo} it holds that
				\begin{equation}
				\begin{split}
				&\left[\int_{[0,T]\times\R^d}\E\Big[\big\vert [h_{2}(x)]^\alpha  h_{3}(x)\big\vert^p\Big]\, \nu(dt,dx)\right]^{\!\nicefrac{1}{p}}
				\\&\le  \Big[\mathcal{C}   +C+
				\tfrac{\max\{4p\alpha,6p\}}{\max\{4p\alpha,6p\}-1} \big|\E\big[\vert  \mathcal{M}_N\vert^{\max\{4p\alpha,6p\}}\big]\big|^{\nicefrac{1}{\max\{4p\alpha,6p\}}} \Big]^{2\alpha+3}
				\\&\cdot 2^{\alpha+1}e^{(2\alpha+3)c } \max\!\big\{1,[\nu([0,T]\times\R^d)]^{\nicefrac{1}{p}}\big\}.
				\end{split}
				\end{equation}
	\end{enumerate}
\end{lemma}

\begin{proof}[Proof of Lemma~\ref{Lemma:LpBoundsGronwall}]	
	Throughout this proof let $\mathscr{C}_q \in \R$, $q 
\in (0,\max\{4p\alpha,6p\}]$, satisfy for all $q 
\in (0,\max\{4p\alpha,6p\}]$ that
	\begin{equation}\label{LpBoundsGronwall:MomentBoundq}
	\mathscr{C}_q=\max\!\left\{\left[\int_{[0,T]\times\R^d}\Vert x\Vert^q \, \nu(dt,dx)\right]^{\!\nicefrac{1}{q}},1\right\}
	\end{equation}
	and let  
			$g\colon \R^d\times \Omega\to [0,\infty)$ satisfy for all 
			$ x \in \R^d $ that 
			\begin{equation}\label{LpBoundsGronwall:AuxFunctionTwo}
			g(x)=\left[\Vert x\Vert + C+
			\max_{n\in\{0,1,\dots,N\}} \mathcal{M}_n\right]
			.
			\end{equation}
	Observe that  Doob's inequality (cf., e.g., Klenke \cite[Theorem 11.2]{Klenke14}), H\"older's inequality, the hypothesis that $\mathcal{M}$ is a submartingale, the hypothesis that  $\forall \, n\in\{0,1,\dots, N\} \colon \mathcal{M}_n\ge 0$, and  the hypothesis that  $\forall \, n\in\{0,1,\dots, N\} \colon \E [ |\mathcal{M}_n|^{\max\{4p\alpha,6p\}}] < \infty $ demonstrate that for all $q\in (1,\max\{4p\alpha,6p\}]$ it holds that 
	\begin{equation}\label{LpBoundsGronwall:Doob}
	\begin{split}
	\Big|\E\Big[\max_{n\in\{0,1,\dots,N\}}\vert\mathcal{M}_n\vert^q\Big]\Big|^{\nicefrac{1}{q}}
	&\le \tfrac{q}{q-1} \big|\E\big[\vert  \mathcal{M}_N\vert^q\big]\big|^{\nicefrac{1}{q}}
	\\&\le \tfrac{q}{q-1} \big|\E\big[\vert  \mathcal{M}_N\vert^{\max\{4p\alpha,6p\}}\big]\big|^{\nicefrac{1}{{\max\{4p\alpha,6p\}}}}
	<\infty.
	\end{split}
	\end{equation}
	Moreover, note that the triangle inequality and \eqref{LpBoundsGronwall:AuxFunctionTwo}   prove that for all $q, r\in (0,\infty)$ with  $ 1 < qr \leq \max\{4\alpha p,6p\}$  it holds that
	\begin{equation}
	\begin{split}
	&\left[\int_{[0,T]\times\R^d}\E\big[\vert  g(x)\vert^{qr}\big]\, \nu(dt,dx)\right]^{\!\nicefrac{1}{qr}}
	\\&\le \left[\int_{[0,T]\times\R^d}\E\big[\Vert  x\Vert^{qr}\big]\, \nu(dt,dx)\right]^{\!\nicefrac{1}{qr}}
 +\left[\int_{[0,T]\times\R^d}\E\bigg[\Big\vert C+ \max_{n\in\{0,1,\dots,N\}} \mathcal{M}_n\Big\vert^{qr}\bigg]\, \nu(dt,dx)\right]^{\!\nicefrac{1}{qr}}.		
	\end{split}
	\end{equation}
	The triangle inequality,   \eqref{LpBoundsGronwall:MomentBoundq}, and \eqref{LpBoundsGronwall:Doob}  hence show that for all $q, r \in (0,\infty)$ with  $ 1 < qr \leq \max\{4\alpha p,6p\}$   it holds that
		\begin{equation}\label{LpBoundsGronwall:g}
		\begin{split}
		&\left[\int_{[0,T]\times\R^d}\E\big[\vert  g(x)\vert^{qr}\big]\, \nu(dt,dx)\right]^{\!\nicefrac{1}{qr}}	
		\\&\le \left[\int_{[0,T]\times\R^d}\Vert x\Vert^{qr} \, \nu(dt,dx)\right]^{\!\nicefrac{1}{qr}}   +\bigg[C+
		\Big|\E\Big[\max_{n\in\{0,1,\dots,N\}}\vert  \mathcal{M}_n\vert^{qr}\Big]\Big|^{\nicefrac{1}{qr}}\bigg] [\nu([0,T]\times\R^d)]^{\nicefrac{1}{qr}}	
		\\&\le \mathscr{C}_{qr}   +\left[C+
		\tfrac{qr}{qr-1} \big|\E\big[\vert  \mathcal{M}_N\vert^{qr}\big]\big|^{\nicefrac{1}{qr}}\right] [\nu([0,T]\times\R^d)]^{\nicefrac{1}{qr}}
		.
		\end{split}
		\end{equation}
	Combining this with \eqref{LpBoundsGronwall:AuxFunctionOne}, \eqref{LpBoundsGronwall:AuxFunctionTwo}, and the fact that $\forall \, a, b \in  [0, \infty)$, $ q \in (0, \infty) \colon (a+b)^q \leq 2^{\max\{q-1, 0\}} (a^q + b^q)$ ensures that for all  $q, r \in (0,\infty)$ with  $ 1 < qr \leq \max\{4\alpha p,6p\}$ it holds that
		\begin{equation}\label{LpBoundsGronwall:hOne}
		\begin{split}
		&\left[\int_{[0,T]\times\R^d}\E\big[\vert  h_r(x)\vert^{q}\big]\, \nu(dt,dx)\right]^{\!\nicefrac{1}{q}}
		\\&\le \max\!\big\{2^{(\nicefrac{1}{q}) - 1},1\big\}  \left[\int_{[0,T]\times\R^d}\indicator{[0,T]\times\R^d}(t,x)\, \nu(dt,dx)\right]^{\!\nicefrac{1}{q}}\\
		&   +e^{rc} \max\!\big\{2^{(\nicefrac{1}{q}) - 1},1\big\} \left[\int_{[0,T]\times\R^d}\E\big[\vert  g(x)\vert^{rq}\big]\, \nu(dt,dx)\right]^{\!\nicefrac{1}{q}}
		\\&\le \max\!\big\{2^{(\nicefrac{1}{q}) - 1},1\big\} [\nu([0,T]\times\R^d)]^{\nicefrac{1}{q}}\\
& +e^{r c} \max\!\big\{2^{(\nicefrac{1}{q}) - 1},1\big\}\left[\mathscr{C}_{qr}   +\left[C+
		\tfrac{qr}{qr-1} \big|\E\big[\vert  \mathcal{M}_N\vert^{qr}\big]\big|^{\!\nicefrac{1}{qr}}\right] [\nu([0,T]\times\R^d)]^{\nicefrac{1}{qr}}\right]^{r}.
		\end{split}
		\end{equation}
	In addition, observe that  \eqref{LpBoundsGronwall:MomentBound} and H\"older's inequality show that for all  $q, r \in (0,\infty)$ with  $ 1 < qr \leq \max\{4\alpha p,6p\}$  it holds that
	\begin{equation}
	\begin{split}
			&\left[\int_{[0,T]\times\R^d}\Vert x\Vert^{qr} \, \nu(dt,dx)\right]^{\!\nicefrac{1}{qr}}
			\\&\le \left[\int_{[0,T]\times\R^d}\Vert x\Vert^{\max\{4\alpha p,6p\}} \, \nu(dt,dx)\right]^{\!\nicefrac{1}{\max\{4\alpha p,6p\}}} [\nu([0,T]\times\R^d)]^{[\frac{1}{qr}-\frac{1}{\max\{4\alpha p,6p\}}]}
						\\&\le \mathcal{C}
						\max\!\big\{1,[\nu([0,T]\times\R^d)]^{\nicefrac{1}{qr}}\big\}.
	\end{split}
	\end{equation}
	This, the fact that $\mathcal{C} \geq 1$, and  \eqref{LpBoundsGronwall:MomentBoundq} prove that 
	 for all  $q, r \in (0,\infty)$ with  $ 1 < qr \leq \max\{4\alpha p,6p\}$  it holds that
	\begin{equation}
	\begin{split}
	\mathscr{C}_{qr} \le \mathcal{C}
	\max\!\big\{1,[\nu([0,T]\times\R^d)]^{\nicefrac{1}{qr}}\big\}.
	\end{split}
	\end{equation}
Combining this with \eqref{LpBoundsGronwall:hOne} implies that for all $q, r \in (0,\infty)$ with  $ 1 < qr \leq \max\{4\alpha p,6p\}$   it holds that
		\begin{equation}
		\begin{split}
		&\left[\int_{[0,T]\times\R^d}\E\big[\vert  h_r(x)\vert^{q}\big]\, \nu(dt,dx)\right]^{\!\nicefrac{1}{q}}
		\\&\le \max\!\big\{2^{(\nicefrac{1}{q}) - 1},1\big\} [\nu([0,T]\times\R^d)]^{\nicefrac{1}{q}}\\
&  +e^{r c} \max\!\big\{2^{(\nicefrac{1}{q}) - 1},1\big\} \left[\mathcal{C}   +C+
		\tfrac{qr}{qr-1} \big|\E\big[\vert  \mathcal{M}_N\vert^{qr}\big]\big|^{\nicefrac{1}{qr}} \right]^{r}
		\max\!\big\{1,[\nu([0,T]\times\R^d)]^{\nicefrac{1}{q}}\big\}.
		\end{split}
		\end{equation}	
	Therefore, we obtain that for all $q, r \in (0,\infty)$ with  $ 1 < qr \leq \max\{4\alpha p,6p\}$  it holds that
				\begin{equation}\label{LpBoundsGronwall:h}
				\begin{split}
				&\left[\int_{[0,T]\times\R^d}\E\big[\vert  h_r(x)\vert^{q}\big]\, \nu(dt,dx)\right]^{\!\nicefrac{1}{q}}
				\\&\le 2e^{r c} \max\!\big\{2^{(\nicefrac{1}{q}) - 1},1\big\} \left[\mathcal{C}   +C+
				\tfrac{qr}{qr-1} \big|\E\big[\vert  \mathcal{M}_N\vert^{qr}\big]\big|^{\nicefrac{1}{qr}} \right]^{r}
				\max\!\big\{1,[\nu([0,T]\times\R^d)]^{\nicefrac{1}{q}}\big\}.
				\end{split}
				\end{equation}
		This establishes item~\eqref{LpBoundsGronwall:ItemOne}.
	Next observe that H\"older's inequality assures that
		\begin{equation}\label{LpBoundsGronwall:itemBstart}
		\begin{split}
		&\left[\int_{[0,T]\times\R^d}\E\Big[\big\vert [h_{2}(x)]^\alpha  h_{3}(x)\big\vert^p\Big]\, \nu(dt,dx)\right]^{\!\nicefrac{1}{p}}
		\\&\le \left[\int_{[0,T]\times\R^d}\E\big[\vert h_{2}(x)\vert^{2p\alpha}  \big]\, \nu(dt,dx)\right]^{\!\nicefrac{1}{2p}}
		\left[\int_{[0,T]\times\R^d}\E\big[\vert h_{3}(x)\vert^{2p}\big]\, \nu(dt,dx)\right]^{\!\nicefrac{1}{2p}}.
		\end{split}
		\end{equation}
	Moreover, note that H\"older's inequality demonstrates that 	
	\begin{equation}
	\begin{split}
	&\left[\int_{[0,T]\times\R^d}\E\big[\vert h_{2}(x)\vert^{2p\alpha}  \big]\, \nu(dt,dx)\right]^{\!\nicefrac{1}{2p\alpha}}
	\\&\le \left[\int_{[0,T]\times\R^d}\E\big[\vert h_{2}(x)\vert^{\max\{2p\alpha,3p\}}  \big]\, \nu(dt,dx)\right]^{\!\nicefrac{1}{\max\{2p\alpha,3p\}}} [\nu([0,T]\times\R^d)]^{[\frac{1}{2p\alpha}-\frac{1}{\max\{2p\alpha,3p\}}]}
	\end{split}
	\end{equation}
	and 
	\begin{align*}
	&\left[\int_{[0,T]\times\R^d}\E\big[\vert h_{3}(x)\vert^{2p}  \big]\, \nu(dt,dx)\right]^{\!\nicefrac{1}{2p}} \numberthis
	\\&\le \left[\int_{[0,T]\times\R^d}\E\big[\vert h_{3}(x)\vert^{\max\{(\nicefrac{4}{3})p\alpha,2p\}}  \big]\, \nu(dt,dx)\right]^{\!\nicefrac{1}{\max\{(\nicefrac{4}{3})p\alpha,2p\}}} [\nu([0,T]\times\R^d)]^{[\frac{1}{2p}-\frac{1}{\max\{(\nicefrac{4}{3})p\alpha,2p\}}]}.
	\end{align*}
	Combining this with \eqref{LpBoundsGronwall:h} implies that 
\begin{align*}
	&\left[\int_{[0,T]\times\R^d}\E\big[\vert h_{2}(x)\vert^{2p\alpha}  \big]\, \nu(dt,dx)\right]^{\!\nicefrac{1}{2p\alpha}}
		\\&\le  2e^{2 c} \Big[\mathcal{C}   +C+
		\tfrac{\max\{4p\alpha,6p\}}{\max\{4p\alpha,6p\}-1} \big|\E\big[\vert  \mathcal{M}_N\vert^{\max\{4p\alpha,6p\}}\big]\big|^{\nicefrac{1}{\max\{4p\alpha,6p\}}} \Big]^{2} \numberthis
		\\&\cdot \max\!\big\{1,[\nu([0,T]\times\R^d)]^{\nicefrac{1}{\max\{2p\alpha,3p\}}}\big\}\ [\nu([0,T]\times\R^d)]^{[\frac{1}{2p\alpha}-\frac{1}{\max\{2p\alpha,3p\}}]}
		\\&\le  2e^{2 c} \Big[\mathcal{C}   +C+
		\tfrac{\max\{4p\alpha,6p\}}{\max\{4p\alpha,6p\}-1} \big|\E\big[\vert  \mathcal{M}_N\vert^{\max\{4p\alpha,6p\}}\big]\big|^{\nicefrac{1}{\max\{4p\alpha,6p\}}} \Big]^{2}
\max\!\big\{1,[\nu([0,T]\times\R^d)]^{\nicefrac{1}{2p\alpha}}\big\}
\end{align*}
and 
\begin{align*}
&\left[\int_{[0,T]\times\R^d}\E\big[\vert h_{3}(x)\vert^{2p}  \big]\, \nu(dt,dx)\right]^{\!\nicefrac{1}{2p}}
\\&\le 2e^{3 c} \Big[\mathcal{C}   +C+
\tfrac{\max\{4p\alpha,6p\}}{\max\{4p\alpha,6p\}-1} \big|\E\big[\vert  \mathcal{M}_N\vert^{\max\{4p\alpha,6p\}}\big]\big|^{\nicefrac{1}{\max\{4p\alpha,6p\}}} \Big]^{3} \numberthis
\\&\cdot \max\!\big\{1,[\nu([0,T]\times\R^d)]^{\nicefrac{1}{\max\{(\nicefrac{4}{3})p\alpha,2p\}}}\big\} [\nu([0,T]\times\R^d)]^{[\frac{1}{2p}-\frac{1}{\max\{(\nicefrac{4}{3})p\alpha,2p\}}]}
\\&\le  2e^{3 c} \Big[\mathcal{C}   +C+
\tfrac{\max\{4p\alpha,6p\}}{\max\{4p\alpha,6p\}-1} \big|\E\big[\vert  \mathcal{M}_N\vert^{\max\{4p\alpha,6p\}}\big]\big|^{\nicefrac{1}{\max\{4p\alpha,6p\}}} \Big]^{3}
\max\!\big\{1,[\nu([0,T]\times\R^d)]^{\nicefrac{1}{2p}}\big\}. 
\end{align*}
This and \eqref{LpBoundsGronwall:itemBstart} show that 
\begin{equation}
\begin{split}
&\left[\int_{[0,T]\times\R^d}\E\Big[\big\vert [h_{2}(x)]^\alpha  h_{3}(x)\big\vert^p\Big]\, \nu(dt,dx)\right]^{\!\nicefrac{1}{p}}
\\&\le \left[\left[\int_{[0,T]\times\R^d}\E\big[\vert h_{2}(x)\vert^{2p\alpha}  \big]\, \nu(dt,dx)\right]^{\!\nicefrac{1}{2p\alpha}}\right]^{\!\alpha}
\left[\int_{[0,T]\times\R^d}\E\big[\vert h_{3}(x)\vert^{2p}\big]\, \nu(dt,dx)\right]^{\!\nicefrac{1}{2p}}
				\\&\le  2^{\alpha+1}e^{(2\alpha+3)c} \Big[\mathcal{C}   +C+
				\tfrac{\max\{4p\alpha,6p\}}{\max\{4p\alpha,6p\}-1} \big|\E\big[\vert  \mathcal{M}_N\vert^{\max\{4p\alpha,6p\}}\big]\big|^{\nicefrac{1}{\max\{4p\alpha,6p\}}} \Big]^{2\alpha+3}
				\\&\cdot \max\!\big\{1,[\nu([0,T]\times\R^d)]^{\nicefrac{1}{2p}}\big\} \left[\max\!\big\{1,[\nu([0,T]\times\R^d)]^{\nicefrac{1}{2p\alpha}}\big\}\right]^\alpha. 
\end{split}
\end{equation}
Hence, we obtain that 
\begin{equation}
\begin{split}
&\left[\int_{[0,T]\times\R^d}\E\Big[\big\vert [h_{2}(x)]^\alpha  h_{3}(x)\big\vert^p\Big]\, \nu(dt,dx)\right]^{\!\nicefrac{1}{p}} 
\\&\le 2^{\alpha+1}e^{(2\alpha+3)c} \Big[\mathcal{C}   +C+
\tfrac{\max\{4p\alpha,6p\}}{\max\{4p\alpha,6p\}-1} \big|\E\big[\vert  \mathcal{M}_N\vert^{\max\{4p\alpha,6p\}}\big]\big|^{\nicefrac{1}{\max\{4p\alpha,6p\}}} \Big]^{2\alpha+3}
\\&\cdot \max\!\big\{1,[\nu([0,T]\times\R^d)]^{\nicefrac{1}{p}}\big\}.
\end{split}
\end{equation}
This establishes item~\eqref{LpBoundsGronwall:ItemTwo}.
	This completes the proof of Lemma~\ref{Lemma:LpBoundsGronwall}.
\end{proof}

\begin{cor}\label{Cor:MCeuler}
	Let $M,N, d,\mathfrak{d},k \in \N$, $p\in [2, \infty)$, $\alpha, c, C,\mathcal{C}, \mathfrak{C} \in[0,\infty)$, $T, \mathfrak{D} \in (0,\infty)$,  $B\in \R^{d\times k}$, $\varepsilon\in (0,1]$, $\drift, \initial \in \ANNs$ satisfy  that  $\inDimANN(\drift)=\outDimANN(\drift)= \inDimANN(\initial)=d$, $\outDimANN(\initial)=\mathfrak{d}$, and
		\begin{equation}\label{MCeuler:paramBound}
		\mathfrak{D}= 2160 [\LogBin(\eps^{-1})+4]-504,
		\end{equation}	
	let $\nu \colon \mathcal{B}([0,T]\times\R^d) \to [0,\infty)$ be a finite measure which satisfies that
		\begin{equation}\label{MCeuler:MomentBound}
		\mathcal{C}=\max\!\left\{\left[\int_{[0,T]\times\R^d}\Vert x\Vert^{\max\{4p\alpha,6p\}} \, \nu(dt,dx)\right]^{\!\nicefrac{1}{\max\{4p\alpha,6p\}}},1\right\},
		\end{equation}	
assume for all $x,y\in\R^d$ that $\Vert (\functionANN(\drift))(x)\Vert\le C+ c\Vert x\Vert$ and
\begin{align}
\label{MCeuler:Lipschitz}
\Vert (\functionANN(\initial)) (x) - (\functionANN(\initial)) (y)\Vert \leq \mathfrak{C} (1 + \Vert x\Vert^{\alpha} + \Vert y \Vert^{\alpha})\Vert  x-y\Vert,
\end{align} 
	let $ ( \Omega, \mathcal{F}, \P, (\mathbb{F}_t)_{t\in [0,T]} ) $ be a filtered probability space which satisfies the usual conditions\footnote{Note that we say that a filtered probability space $(\Omega,\mathcal F,\P,( \mathbb{F}_t)_{t\in [0,T]})$ satisfies the usual conditions if and only if it holds for all 
		$ t \in [0,T) $ 
		that 
		$ \{ A \in \mathcal F\colon \P(A) = 0 \} \subseteq \mathbb{F}_t = (\cap_{s\in (t,T]} \mathbb{F}_s) $; cf., e.g., Liu \& R\"ockner~\cite[Definition 2.1.11]{LiuRoeckner2015SPDEs}.},
	let $ W^{  m } \colon [0,T] \times \Omega \to \R^k $, $  m \in \{1,2,\dots,M\} $, 
	be standard $(\mathbb{F}_t)_{t\in [0,T]}$--Brownian motions, 
	and let 
	$Y^m= (Y^{m,x }_t(\omega))_{(t, x, \omega) \in [0,T] \times \R^d \times \Omega} \colon\allowbreak [0,T]\times \R^d\times\Omega \to \R^d $, $m\in\{1,2,\dots,\allowbreak M\}$, 
satisfy for all 
	$m\in\{1,2,\dots,\allowbreak M\}$,
	$ x \in \R^d $,
	$n\in\{0,1,\dots,N-1\}$,
	$ t \in \big[\tfrac{nT}{N},\tfrac{(n+1)T}{N}\big]$
	that $\affineProcess^{m,x }_0=x$ and
	\begin{equation}
	\label{MCeuler:Y_processes}
	\begin{split}
	&\affineProcess^{m,x }_t
	=
	\affineProcess^{m,x }_{\frac{nT}{N}}+ \left(\tfrac{tN}{T}-n\right) \! \left[\tfrac{T}{N}(\functionANN(\drift)) ( 
	\affineProcess^{m,x }_{\frac{nT}{N}} 
	)
	+
	B\big(W_{\frac{(n+1)T}{N}}^{m}-W_{\frac{nT}{N}}^{m}\big)\right]
	\end{split}
	\end{equation} 
	(cf.\ \Cref{Def:euclideanNorm,Def:ANN,Definition:ANNrealization,Definition:Relu1}).
	Then there exists $(\Psi_{\omega})_{\omega\in\Omega}\subseteq \ANNs$ such that
	\begin{enumerate}[(i)]
		\item \label{MCeuler:Function} it holds for all   $\omega\in \Omega$ that $\functionANN (\Psi_{\omega})\in C(\R^{d+1},\R^\mathfrak{d})$,
				\item \label{MCeuler:Continuity}
				it holds  for all  $t\in [0,T]$, $x\in\R^d$ that $\Omega\ni \omega\mapsto (\functionANN (\Psi_{\omega}))(t,x)\in\R^\mathfrak{d}$
				is measurable,
		\item it holds  that 
\begin{equation}\label{MCeuler:Estimate}
\begin{split}
&\left[\int_{[0,T]\times\R^d}\int_{\Omega}
\left\Vert(\functionANN(\Psi_{\omega}))(t,x)-\frac{1}{M}\left[\smallsum\limits_{m=1}^M (\functionANN(\initial))(Y_t^{m,x}(\omega))\right]\right\Vert^p 
\P(d\omega) \, \nu(dt,dx)\right]^{\!\nicefrac{1}{p}}
\\&\le  \big[2p\max\{C,\mathcal{C}\}\max\{T,1\}\max\{\alpha,1\}\big]^{2\alpha+3}\left[1+
\sqrt{ \operatorname{Trace}(B^*B)} \right]^{2\alpha+3}
\\&\cdot \varepsilon {d}^{\nicefrac{(\alpha+1)}{2}} \mathfrak{C}\,e^{(2\alpha+3)cT} 2^{2\alpha+4} 3^{\alpha}
\max\!\big\{1,[\nu([0,T]\times\R^d)]^{\nicefrac{1}{p}}\big\},	
\end{split}
\end{equation}
		and
		\item \label{MCeuler:ItemParams}
		it holds for all  $\omega\in \Omega$ that 
								\begin{equation}
								\begin{split}
								\paramANN(\Psi_\omega)
								&\le 2 M^2 \paramANN(\initial) + 9 M^2 N^6 d^{16} \big[  2 \lengthANN(\drift)
								+
								 \mathfrak{D}+(
								24
								+6\lengthANN(\drift)
								+   [4+\paramANN(\drift)]^{2})^{2}
								\big]^{2}.
								\end{split}
								\end{equation}
	\end{enumerate}
\end{cor}

\begin{proof}[Proof of Corollary~\ref{Cor:MCeuler}]	
	Throughout this proof let 
	$h_{m,r}\colon \R^d\times \Omega\to [0,\infty)$,  $m\in\{1,2,\dots, M\}$, $r\in \R$, satisfy for all $m\in\{1,2,\dots, M\}$, $r\in \R$,
	$ x \in \R^d$ that 
	\begin{equation}\label{MCeuler:AuxFunctionTwo}
	h_{m,r}(x)=1+\Big[\Vert x\Vert + CT+
	\max_{n\in\{0,1,\dots,N\}}\big\Vert B  W_{\frac{nT}{N}}^m\big\Vert\Big]^r
	\exp(rc T),
	\end{equation}
	let $(\Psi_{\omega})_{\omega\in\Omega}\subseteq \ANNs$ satisfy that 
	\begin{enumerate}[(I)]
		\item \label{MCeulerProof:Function} it holds for all  $\omega\in \Omega$  that $\functionANN (\Psi_{\omega})\in C(\R^{d+1},\R^\mathfrak{d})$,
		\item  it holds  for all  $\omega\in \Omega$,
		$ t \in [0,T]$,
		$x\in\R^d$ that 
		\begin{equation}\label{MCeulerProof:FinalTwo}
		\begin{split}
		&\left\Vert(\functionANN(\Psi_{\omega}))(t,x)-\frac{1}{M}\left[\smallsum\limits_{m=1}^M (\functionANN(\initial))(Y_t^{m,x}(\omega))\right]\right\Vert 
		\\&\le \frac{2\varepsilon \mathfrak{C} \sqrt{d}}{M}\bigg[\smallsum\limits_{m=1}^M   \left[1 +2 {d}^{\nicefrac{\alpha}{2}}  6^\alpha
		|h_{m,2}(x,\omega)|^\alpha \right]
		h_{m,3}(x,\omega) \bigg],
		\end{split}
		\end{equation}
		\item \label{MCeulerProof:ItemParams}
		it holds 
		for all  $\omega\in \Omega$ that 
						\begin{equation}
						\begin{split}
						&\paramANN(\Psi_\omega) \le 2 M^2 \paramANN(\initial) +  9 M^2 N^6 d^{16} \big[  2 \lengthANN(\drift)
						+
						 \mathfrak{D}+(
						24
						+6\lengthANN(\drift)
						+   [4+\paramANN(\drift)]^{2})^{2}
						\big]^{2},
						\end{split}
						\end{equation}
		and
		\item \label{MCeulerProof:Continuity}
		it holds 
		 for all  $t\in [0,T]$, $x\in\R^d$  that $\Omega\ni \omega\mapsto (\functionANN (\Psi_{\omega}))(t,x)\in\R^\mathfrak{d}$
		is measurable
	\end{enumerate}			
	(cf.\ Proposition~\ref{Cor:ApproxOfMCSum}), let 
		$Z= (Z^{x,y }_t)_{(t,x,y)\in [0,T]\times \R^d\times(\R^d)^N} \colon\allowbreak [0,T]\times \R^d\times (\R^d)^N \to \R^d $ 
		satisfy for all 
		$n\in\{0,1,\dots,N-1\}$,
		$ t \in [\tfrac{nT}{N},\tfrac{(n+1)T}{N}]$,
		$ x \in \R^d $, $y=(y_1, y_2, \dots, y_N)\in (\R^d)^N$
		that $Z^{x,y }_0=x$ and
		\begin{equation}
		\label{MCeulerProof:Z_processes}
		\begin{split}
Z^{x,y }_{t}=
		Z^{x,y }_{\frac{nT}{N}}+ \left(\tfrac{tN}{T}-n\right) \!\left[\tfrac{T}{N}(\functionANN(\drift)) \big( 
		Z^{x,y }_{\frac{nT}{N}}
		\big)
		+
		y_{n+1}\right]\!,
		\end{split}
		\end{equation}
		and let $\mathcal{W}^m\colon [0,T]\times \R^d\times \Omega\to [0,T]\times \R^d\times (\R^d)^N$, $m\in\{1,2,\dots, M\}$,  satisfy for all $m\in\{1,2,\dots, M\}$, $t\in [0,T]$, $x\in\R^d$  that
		\begin{equation}\label{MCeuler:increments}
			\mathcal{W}^m(t,x)=\Big(t,x, B\big(W_{\frac{T}{N}}^{m}-W_{0}^{m}\big), B\big(W_{\frac{2T}{N}}^{m}-W_{\frac{T}{N}}^{m}\big), \ldots, B\big(W_{\frac{NT}{N}}^{m}-W_{\frac{(N-1)T}{N}}^{m}\big) \Big).
		\end{equation}
		Note that \eqref{MCeulerProof:Function}, \eqref{MCeulerProof:Continuity}, and Beck et al.~\cite[Lemma~2.4]{Kolmogorov} demonstrate that $ [0,T]\times\R^{d}\times \Omega \ni(t,x,\omega)\mapsto (\functionANN(\Psi_{\omega}))(t,x)\in \R^{\mathfrak{d}}$ is 
		measurable.
In addition, observe that \eqref{MCeuler:Y_processes}, \eqref{MCeulerProof:Z_processes}, and \eqref{MCeuler:increments} ensure that for all $m\in\{1,2,\dots,M\}$, $t\in [0,T]$, $x\in\R^d$ it holds that
\begin{equation}\label{MCeuler:Y_processes:Factorization}
Y_t^{m,x}=(Z\circ \mathcal{W}^m)(t,x).
\end{equation}
Next note that Grohs et al.~\cite[Lemma 3.8]{GrohsHornungEtAl2023}
(applied with $N\is N$, $d\is d$, $\mu\is \functionANN(\drift)$, $T\is T$, $
(\{-1,0,1,\dots,N+1\} \ni n \mapsto t_n \in \R) \is (\{-1,0,1,\dots,N+1\} \ni n \mapsto \tfrac{nT}{N} \in \R)$,  $Y\is Z$ in the notation of Grohs et al.~\cite[Lemma 3.8]{GrohsHornungEtAl2023}) proves that $Z\in C([0,T]\times \R^d\times (\R^d)^N,\R^d)$.
Combining this with \eqref{MCeuler:Y_processes:Factorization} and the fact that for all $m\in\{1,2,\dots,M\}$ it holds that $\mathcal{W}^m$ is measurable shows that for all $m\in\{1,2,\dots,M\}$ it holds that $Y^m$ is measurable.
The fact that $\functionANN(\initial)\in C(\R^d,\R^{\mathfrak{d}})$ hence ensures that 
\begin{equation}
	[0,T]\times \R^d\times \Omega \ni (t,x,\omega)\mapsto \frac{1}{M}\bigg[\smallsum\limits_{m=1}^M (\functionANN(\initial))(Y_t^{m,x}(\omega))\bigg]\in \R^{\mathfrak{d}}
\end{equation}
is measurable.
Combining this with \eqref{MCeulerProof:FinalTwo} and the fact that $ [0,T]\times\R^{d}\times \Omega \ni(t,x,\omega)\mapsto (\functionANN(\Psi_{\omega}))(t,x)\in \R^{\mathfrak{d}}$ is measurable proves that 
		\begin{align*}
		&\left[\int_{[0,T]\times\R^d}\int_{\Omega}
		\left\Vert(\functionANN(\Psi_{\omega}))(t,x)-\frac{1}{M}\bigg[\smallsum\limits_{m=1}^M (\functionANN(\initial))(Y_t^{m,x}(\omega))\bigg]\right\Vert^p 
		\P(d\omega) \, \nu(dt,dx)\right]^{\!\nicefrac{1}{p}} \numberthis
		\\&\le \frac{2\varepsilon \mathfrak{C} \sqrt{d}}{M}
		\left[\int_{[0,T]\times\R^d}\int_{\Omega}\bigg\vert\smallsum\limits_{m=1}^M   \left[1 +2 {d}^{\nicefrac{\alpha}{2}} 6^\alpha
		|h_{m,2}(x,\omega)|^\alpha \right]
		h_{m,3}(x,\omega) \bigg\vert^p 
		\,\P(d\omega) \, \nu(dt,dx)\right]^{\!\nicefrac{1}{p}}.													
		\end{align*}
The triangle inequality
therefore implies that 
		\begin{equation}\label{MCeulerProof:Start}
		\begin{split}
		&\left[\int_{[0,T]\times\R^d}\int_{\Omega}
		\left\Vert(\functionANN(\Psi_{\omega}))(t,x)-\frac{1}{M}\bigg[\smallsum\limits_{m=1}^M (\functionANN(\initial))(Y_t^{m,x}(\omega))\bigg]\right\Vert^p 
		\P(d\omega) \, \nu(dt,dx)\right]^{\!\nicefrac{1}{p}}
				\\&\le  \frac{2\varepsilon \mathfrak{C} \sqrt{d}}{M} \sum_{m=1}^M
				\left[\int_{[0,T]\times\R^d}\int_{\Omega}\vert  
				h_{m,3}(x,\omega) \vert^p 
				\,\P(d\omega) \, \nu(dt,dx)\right]^{\!\nicefrac{1}{p}}
				\\&+\frac{4  {d}^{\nicefrac{\alpha}{2}} 6^\alpha \varepsilon \mathfrak{C} \sqrt{d}}{M} \sum_{m=1}^M
				\left[\int_{[0,T]\times\R^d}\int_{\Omega}|
				[h_{m,2}(x,\omega)]^\alpha 
				h_{m,3}(x,\omega) |^p 
				\,\P(d\omega) \, \nu(dt,dx)\right]^{\!\nicefrac{1}{p}}.														
		\end{split}
		\end{equation}
Next note that \eqref{MCeuler:AuxFunctionTwo}, Lemma~\ref{Lemma:LpBoundsGronwall}, and the fact that for all $m\in\{1,2,\dots, M\}$ it holds that 
$(\|  B W_{\frac{nT}{N}}^m\|)_{n \in \{0, 1, \ldots, N\}}$ is a nonnegative $(\mathbb{F}_{\frac{nT}{N}})_{n \in \{0, 1, \ldots, N\}}$-submartingale
 demonstrate that for all $m\in\{1,2,\dots, M\}$ it holds that
	\begin{equation}\label{MCeulerProof:H}
	\begin{split}
	&\left[\int_{[0,T]\times\R^d}\E\big[\vert  h_{m,3}(x)\vert^{p}\big]\, \nu(dt,dx)\right]^{\!\nicefrac{1}{p}}
	\\&\le 2e^{3cT} \left[\mathcal{C}  +CT+
	\tfrac{3p}{3p-1} \big|\E\big[\Vert  B W_T^m\Vert^{3p}\big]\big|^{\nicefrac{1}{3p}} \right]^{3}
	\max\!\big\{1,[\nu([0,T]\times\R^d)]^{\nicefrac{1}{p}}\big\}
	\end{split}
	\end{equation}
	and
	\begin{equation}\label{MCeulerProof:productH}
	\begin{split}
	&\left[\int_{[0,T]\times\R^d}\E\big[| [h_{m,2}(x)]^\alpha  h_{m,3}(x)|^p\big]\, \nu(dt,dx)\right]^{\!\nicefrac{1}{p}}
	\\&\le  \Big[\mathcal{C}  +CT+
	\tfrac{\max\{4p\alpha,6p\}}{\max\{4p\alpha,6p\}-1} \big|\E\big[\Vert  B W_T^m\Vert^{\max\{4p\alpha,6p\}}\big]\big|^{\nicefrac{1}{\max\{4p\alpha,6p\}}} \Big]^{2\alpha+3}
	\\&\cdot 2^{\alpha+1}e^{(2\alpha+3)cT} \max\!\big\{1,[\nu([0,T]\times\R^d)]^{\nicefrac{1}{p}}\big\}.
	\end{split}
	\end{equation}
	Moreover, observe that 
	Lemma~\ref{lem:momentsGauss},
	 the fact that for all $m\in\{1,2,\dots, M\}$ it holds that  $B W_T^m$ is a  Gaussian random variable, and  the fact that for all $m\in\{1,2,\dots, M\}$ it holds that  $ \covariance\,(B W_T^m)= T B^*B$
	ensure that  for all $m\in\{1,2,\dots, M\}$, $q\in [2,\infty)$ it holds that 
	\begin{equation}
	\label{MCeulerProof:Moments}
	\begin{split}
	\tfrac{q}{q-1}\big|\E\big[\Vert  B W_T^m\Vert^{q}\big]\big|^{\nicefrac{1}{q}}
	&\leq \tfrac{q}{q-1}
	\sqrt{ \max\{1,q-1\} T \operatorname{Trace}(B^*B)} 
	\\&= \tfrac{q}{q-1}
	\sqrt{ (q-1) T \operatorname{Trace}(B^*B)} 
	\\&= \tfrac{q}{\sqrt{q-1}}
	\sqrt{T \operatorname{Trace}(B^*B)} 
	.
	\end{split}
	\end{equation}	
	Combining this with \eqref{MCeulerProof:H} and \eqref{MCeulerProof:productH} assures that
	\begin{equation}\label{MCeulerProof:Htwo}
	\begin{split}
	&\frac{1}{M} \sum_{m=1}^M\left[\int_{[0,T]\times\R^d}\E\big[\vert  h_{m,3}(x)\vert^{p}\big]\, \nu(dt,dx)\right]^{\!\nicefrac{1}{p}}
	\\&\le 2e^{3cT} \left[\mathcal{C}  +CT+
	\tfrac{3p}{\sqrt{3p-1}}
	\sqrt{T \operatorname{Trace}(B^*B)} \right]^{3}
	\max\!\big\{1,[\nu([0,T]\times\R^d)]^{\nicefrac{1}{p}}\big\}
	\end{split}
	\end{equation}	
	and 
	\begin{equation}\label{MCeulerProof:productHtwo}
	\begin{split}
	&\frac{1}{M} \sum_{m=1}^M\left[\int_{[0,T]\times\R^d}\E\big[\big\vert [h_{m,2}(x)]^\alpha  h_{m,3}(x)\big\vert^p\big]\, \nu(dt,dx)\right]^{\!\nicefrac{1}{p}}
	\\&\le  \Big[\mathcal{C}  +CT+
	\tfrac{\max\{4p\alpha,6p\}}{\sqrt{\max\{4p\alpha,6p\}-1}} 
	\sqrt{T \operatorname{Trace}(B^*B)} \Big]^{2\alpha+3}
	\\&\cdot 2^{\alpha+1}e^{(2\alpha+3)cT} \max\!\big\{1,[\nu([0,T]\times\R^d)]^{\!\nicefrac{1}{p}}\big\}.
	\end{split}
	\end{equation}	
	This and \eqref{MCeulerProof:Start}	demonstrate that 
\begin{equation}\label{MCeulerProof:PreFinal}
\begin{split}
&\left[\int_{[0,T]\times\R^d}\int_{\Omega}
\left\Vert(\functionANN(\Psi_{\omega}))(t,x)-\frac{1}{M}\left[\smallsum\limits_{m=1}^M (\functionANN(\initial))(Y_t^{m,x}(\omega))\right]\right\Vert^p 
\P(d\omega) \, \nu(dt,dx)\right]^{\!\nicefrac{1}{p}}
\\&\le  4\varepsilon \mathfrak{C} e^{3cT} \sqrt{d} \max\!\big\{1,[\nu([0,T]\times\R^d)]^{\nicefrac{1}{p}}\big\}   \left[\mathcal{C}  +CT+
\tfrac{3p}{\sqrt{3p-1}}
\sqrt{T \operatorname{Trace}(B^*B)} \right]^{3} 
\\&+ \Big[\mathcal{C}  +CT+
\tfrac{\max\{4p\alpha,6p\}}{\sqrt{\max\{4p\alpha,6p\}-1}} 
\sqrt{T \operatorname{Trace}(B^*B)} \Big]^{2\alpha+3}
\\&\cdot 4\varepsilon \mathfrak{C} 6^\alpha {d}^{\nicefrac{\alpha}{2}} 2^{\alpha+1} e^{(2\alpha+3)cT} \sqrt{d} \max\!\big\{1,[\nu([0,T]\times\R^d)]^{\nicefrac{1}{p}}\big\}.					
\end{split}
\end{equation}	
The fact that  $[2,\infty)\ni x\mapsto \nicefrac{x}{\sqrt{x-1}}\in\R$ is non-decreasing
hence implies that 
\begin{equation}
\begin{split}
&\left[\int_{[0,T]\times\R^d}\int_{\Omega}
\left\Vert(\functionANN(\Psi_{\omega}))(t,x)-\frac{1}{M}\left[\smallsum\limits_{m=1}^M (\functionANN(\initial))(Y_t^{m,x}(\omega))\right]\right\Vert^p 
\P(d\omega) \, \nu(dt,dx)\right]^{\!\nicefrac{1}{p}}
\\&\le   \left[\mathcal{C}  +CT+
\tfrac{\max\{4p\alpha,6p\}}{\sqrt{\max\{4p\alpha,6p\}-1}}
\sqrt{T \operatorname{Trace}(B^*B)} \right]^{2\alpha+3}
\\&\cdot \Big[4\varepsilon \mathfrak{C} e^{3cT} \sqrt{d} +  4\varepsilon \mathfrak{C} 6^\alpha {d}^{\nicefrac{\alpha}{2}} 2^{\alpha+1} e^{(2\alpha+3)cT} \sqrt{d} \Big]
\max\!\big\{1,[\nu([0,T]\times\R^d)]^{\nicefrac{1}{p}}\big\}
\\&\le  |\!\max\{C,\mathcal{C}\}|^{2\alpha+3} |\!\max\{T,1\}|^{2\alpha+3} \left[2+
\tfrac{\max\{4p\alpha,6p\}}{\sqrt{\max\{4p\alpha,6p\}-1}}
\sqrt{ \operatorname{Trace}(B^*B)} \right]^{2\alpha+3}
\\&\cdot \varepsilon {d}^{\nicefrac{(\alpha+1)}{2}} \mathfrak{C}\,e^{(2\alpha+3) c T} \big[4  + 6^\alpha  2^{\alpha+1} 4\big]
\max\!\big\{1,[\nu([0,T]\times\R^d)]^{\nicefrac{1}{p}}\big\}.		
\end{split}
\end{equation}	
The fact that $\sqrt{\max\{4p\alpha,6p\}-1}\ge \sqrt{6p-1}\ge \sqrt{11}\ge 3$ therefore ensures that 
\begin{equation}\label{MCeulerProof:Final}
\begin{split}
&\left[\int_{[0,T]\times\R^d}\int_{\Omega}
\left\Vert(\functionANN(\Psi_{\omega}))(t,x)-\frac{1}{M}\left[\smallsum\limits_{m=1}^M (\functionANN(\initial))(Y_t^{m,x}(\omega))\right]\right\Vert^p 
\P(d\omega) \, \nu(dt,dx)\right]^{\!\nicefrac{1}{p}}
\\&\le  |\!\max\{C,\mathcal{C}\}|^{2\alpha+3} |\!\max\{T,1\}|^{2\alpha+3} \left[2+
\max\{(\nicefrac{4}{3})p\alpha,2p\}
\sqrt{ \operatorname{Trace}(B^*B)} \right]^{2\alpha+3}
\\&\cdot \varepsilon {d}^{\nicefrac{(\alpha+1)}{2}} \mathfrak{C}\,e^{(2\alpha+3)cT} \big[4+2^{2\alpha+3} 3^{\alpha}\big]
\max\!\big\{1,[\nu([0,T]\times\R^d)]^{\nicefrac{1}{p}}\big\}
\\&\le  |\!\max\{C,\mathcal{C}\}|^{2\alpha+3} |\!\max\{T,1\} |^{2\alpha+3} 
[2p\max\{\alpha,1\}]^{2\alpha+3}\left[1+
\sqrt{ \operatorname{Trace}(B^*B)} \right]^{2\alpha+3}
\\&\cdot \varepsilon {d}^{\nicefrac{(\alpha+1)}{2}} \mathfrak{C}\,e^{(2\alpha+3)cT} 2^{2\alpha+4} 3^{\alpha}
\max\!\big\{1,[\nu([0,T]\times\R^d)]^{\nicefrac{1}{p}}\big\}.		
\end{split}
\end{equation}
Combining this with \eqref{MCeulerProof:Function}, \eqref{MCeulerProof:ItemParams}, and \eqref{MCeulerProof:Continuity} establishes items~\eqref{MCeuler:Function}--\eqref{MCeuler:ItemParams}.
	This completes the proof of Corollary~\ref{Cor:MCeuler}.
\end{proof}

%% file: DNNapproximationERRORnew.tex
\subsection{Approximation error estimates for deep ANNs}
\label{subsec:Appr_DNNs}

\begin{prop}
	\label{prop:DNNerrorEstimate}
	 Let 
	$ T, \kappa \in (0,\infty) $, $\eta \in [1, \infty)$,  $p \in [2, \infty)$, 
	let 
	$
	A_d = ( a_{ d, i, j } )_{ (i, j) \in \{ 1, 2,\dots, d \}^2 } $ $ \in \R^{ d \times d }
	$,
	$ d \in \N $,
	be symmetric positive semidefinite matrices, 
 let $\nu_d  \colon \mathcal{B}([0,T]\times\R^d) \to [0,\infty)$, $d\in\N$, be finite measures which satisfy for all $d\in\N$ that 
	\begin{equation}\label{DNNerrorEstimate:MeasureAssumption}
	\left[\int_{[0,T]\times \R^d} 
\|x \|^{2p \max\{2\kappa, 3\}}
\, \nu_d (d t, d x) \right]^{\!\nicefrac{1}{p}}\leq \eta d^{\eta},
	\end{equation}
	let
	$f^m_d \colon \R^d \to \R^{md-m+1}$, $m\in\{0,1\}$, $d \in \N $,
	be functions,
		let 
	$
	( \mathfrak{f}^m_{d, \varepsilon})_{ 
		(m, d, \varepsilon) \in \{ 0, 1 \} \times \N \times (0,1] 
	} 
	\subseteq \ANNs
	$,
	assume for all
		$m\in\{0,1\}$,
	$ d \in \N $, 
	$ \varepsilon \in (0,1] $, 
	$ 
	x, y \in \R^d
	$
	that 
\vspace{-1ex}
	\begin{gather}
		\label{eq:prop:new}
	\functionANN( \mathfrak{f}^{ 0 }_{d, \varepsilon } )
	\in 
	C( \R^d,  \R), \qquad 	\functionANN( \mathfrak{f}^{ 1 }_{d, \varepsilon } )
	\in 
	C( \R^d,  \R^{ d }),  \qquad 	\mathcal{P}( \mathfrak{f}^{ m }_{d, \varepsilon } ) 
	\leq \kappa d^{ \kappa } \varepsilon^{ - \kappa }, \\
	\label{DNNerrorEstimate:growthPhiOne}
	\| 
	f^1_d( x ) 
	- 
	f^1_d( y )
	\|
	\leq 
	\kappa 
	\| x - y \|, \qquad \|
	(\functionANN(\mathfrak{f}^1_{d, \varepsilon}))(x)    
	\|	
	\leq 
	\kappa ( d^{ \kappa } + \| x \| ),	\\
	\label{DNNerrorEstimate:approximationLocallyLipschitz}\left\vert (\functionANN(\mathfrak{f}^0_{d, \varepsilon})) (x) - (\functionANN(\mathfrak{f}^0_{d, \varepsilon})) (y)\right\vert \leq \kappa d^{\kappa} (1 + \|x\|^{\kappa} + \|y \|^{\kappa})\|x-y\|,
	\\
	\label{DNNerrorEstimate:appr:coef}
	\left\| 
	f^m_d(x) 
	- 
	(\functionANN(\mathfrak{f}^m_{d, \varepsilon})) (x)
	\right\|
	\leq 
	\varepsilon \kappa d^{ \kappa }
	(
	1 + \| x \|^{ \kappa }
	),
		\\ \label{DNNerrorEstimate:growthMatrix}
	\text{and} \qquad
 	|
	f^0_d( x )
	| 
	+
	\operatorname{Trace}(A_d)
	\leq 
	\kappa d^{ \kappa }
	( 1 + \| x \|^{ \kappa } )
	,
	\end{gather}
and	for every $d \in \N$	let
	$u_d \in  \{v \in C([0, T] \times \R^d, \R) \colon \allowbreak \inf_{ q \in (0,\infty) }
	\allowbreak \sup_{ (t,y) \in [0,T] \times \R^d } \allowbreak
	\frac{ | v(t,y) | }{
		1 + \| y \|^q
	}
	< \infty \}$ be a viscosity solution of
		\begin{equation}
		\begin{split}
		( \tfrac{ \partial }{\partial t} u_d )( t, x ) 
		& = 
		( \tfrac{ \partial }{\partial x} u_d )( t, x )
		\,
		f^1_d( x )
		+
		\sum_{ i, j = 1 }^d
		a_{ d, i, j }
		\,
		( \tfrac{ \partial^2 }{ \partial x_i \partial x_j } u_d )( t, x )
		\end{split}
		\end{equation}
		with $ u_d( 0, x ) = f^0_d( x )$
		for $ ( t, x ) \in (0,T) \times \R^d $ (cf.\ \Cref{Def:euclideanNorm,Def:ANN,Definition:ANNrealization,Definition:Relu1}).
	Then 
		there exist $\mathcal{C} \in \R$   and $(\mathfrak{u}_{d,N,M,\delta})_{(d,N,M, \delta) \in \N^3 \times (0,1] }\subseteq \ANNs$   such that
		\begin{enumerate}[(i)]
				\item \label{DNNerrorEstimate:func}it holds for all 
			$d,N,M\in\N$, $\delta\in (0,1]$
			that $\functionANN(\mathfrak{u}_{d,N,M,\delta})\in C(\R^{d+1},\R)$,
			\item \label{DNNerrorEstimate:Estimate}it holds for all 
			$d,N,M\in\N$, $\delta\in (0,1]$
			that 
			\begin{equation}
			\begin{split}
			& \left[
			\int_{[0,T]\times\R^d}
			\left\vert
			u_d(y) 
			- 
			(\functionANN(\mathfrak{u}_{d,N,M,\delta}))(y)
			\right\vert^p
			\nu_d (d y)
			\right]^{\!\nicefrac{1}{p}} 
			\\&\le
			\mathcal{C} \! \left[\max\!\big\{1,\nu_d([0,T]\times\R^d)\big\}\right]^{\!\nicefrac{1}{p}}
			\\&\cdot \left[ \frac{d^{\kappa(\kappa+4)+\max\{\eta,\kappa(2\kappa+1)\}}}{N^{\nicefrac{1}{2}}}
			+ \frac{d^{\kappa+\max\{\eta,\kappa^2\}}}{M^{\nicefrac{1}{2}}}
			+\delta d^{(2\kappa+3)\max\{\eta,\kappa\}+\kappa^2+\nicefrac{(7\kappa+1)}{2}} \right],
			\end{split}
			\end{equation}
			and
			\item \label{DNNerrorEstimate:Params}  it holds for all 
			$d,N,M\in\N$, $\delta\in (0,1]$
			that
			\begin{equation}
			\begin{split}
			&\paramANN(\mathfrak{u}_{d,N,M,\delta})
			\le \mathcal{C} M^2 N^{6+4\kappa} [\LogBin(\delta^{-1})+1]^{2} d^{16+8\kappa}
			.
			\end{split}
			\end{equation}
		\end{enumerate}
\end{prop}

\begin{proof}[Proof of Proposition~\ref{prop:DNNerrorEstimate}]
Throughout this proof		
	let $\mathfrak{D}_\delta \in \R$, $\delta\in (0,1]$, satisfy for all $\delta\in (0,1]$ that 
	\begin{equation}\label{DNNerrorEstimate:paramBound}
	\mathfrak{D}_\delta= 2160 [\LogBin(\delta^{-1})+4]-504,
	\end{equation}
	let $C_d \in \R$, $d \in \N$,  satisfy for all $d\in\N$ that 
	\begin{equation}\label{DNNerrorEstimate:MomentBound}
	C_d=\max\!\left\{\left[\int_{[0,T]\times\R^d}\Vert x\Vert^{\max\{4p\kappa,6p\}} \, \nu_d(dt,dx)\right]^{\!\nicefrac{1}{\max\{4p\kappa,6p\}}},1\right\},
	\end{equation}
	let $\mathcal{C}_1 \in(0,\infty)$ satisfy that
	\begin{equation}\label{DNNerrorEstimateProof:DefinitionCone}
	\mathcal{C}_1=\big[2p\max\{\eta,\kappa\} \max\{T,1\}\max\{\kappa,1\}\big]^{2\kappa+3} [1+(2\kappa)^{\nicefrac{1}{2}}]^{2\kappa+3}
	\kappa e^{(2\kappa+3)\kappa  T} 2^{2\kappa + 4}  \, 3^{\kappa},
	\end{equation}
	let $\mathcal{C}_2\in (0,\infty)$ satisfy that
	\begin{equation}\label{DNNerrorEstimateProof:DefinitionCtwo}
		\mathcal{C}_2=	
		2^{57} [\max\{\kappa,1\}]^8  [\max\{T^{-\nicefrac{\kappa}{2}},1\}]^8,
	\end{equation}
	let $ ( \Omega, \mathcal{F}, \P, (\mathbb{F}_t)_{t\in [0,T]} ) $ be a filtered probability space which satisfies the usual conditions,
	let $ W^{ d, m } \colon [0,T] \times \Omega \to \R^d $, $ d, m \in \N $, 
	be independent standard $(\mathbb{F}_t)_{t\in [0,T]}$-Brownian motions, 
	let 
	$ \affineProcess^{ N, d, m, x } \colon [0,T] \times \Omega \to \R^d $, 	$ x \in \R^d $,
	$ N,d, m \in \N $,
	be stochastic processes 
	which satisfy for all 
	$N, d, m \in \N $,
	$ x \in \R^d $,
	$n\in\{0,1,\dots,N-1\}$,
	$ t \in [\tfrac{nT}{N},\tfrac{(n+1)T}{N}]$  
	that $\affineProcess^{ N, d, m, x }_0=x$ and
	\begin{equation}
	\begin{split}
	&\affineProcess^{ N, d, m, x }_t 
	=
	\\&\affineProcess^{ N, d, m, x }_{\frac{nT}{N}}+ \left(\tfrac{tN}{T}-n\right)\Big[\tfrac{T}{N}\big(\functionANN\big(\mathfrak{f}^1_{d, \deltaIndex}\big)\big) \big( 
	\affineProcess^{ N, d, m, x }_{\frac{nT}{N}} 
	\big)
	+\sqrt{ 2 A_d }
	\big(W^{ d, m }_{\frac{(n+1)T}{N}}-W^{ d, m }_{\frac{nT}{N}}\big)\Big]
	\!,
	\end{split}
	\end{equation}
	let 	
	$(\psi_{d,N,M,\delta,\omega})_{(d,N,M, \delta, \omega) \in \N^3 \times (0,1] \times \Omega}\subseteq \ANNs$   satisfy that
	\begin{enumerate}[(I)]
		\item \label{DNNerrorEstimateProof:Function} it holds for all $d,N,M\in\N$, $\delta\in (0,1]$,  $\omega\in \Omega$ that $\functionANN (\psi_{d,N,M,\delta,\omega})\in C(\R^{d+1},\R)$,
		\item \label{DNNerrorEstimateProof:Continuity}
		it holds  for all $d,N,M\in\N$, $\delta\in (0,1]$,  $t\in [0,T]$, $x\in\R^d$ that $\Omega\ni \omega\mapsto\allowbreak (\functionANN (\psi_{d,N,M,\delta,\omega}))(t,x)\in\R$
		is measurable,
		\item \label{DNNerrorEstimateProof:ItemEstimate} it holds  for all $d,N,M\in\N$, $\delta\in (0,1]$ that 
		\begin{equation}\label{DNNerrorEstimateProof:Estimate}
		\begin{split}
		&\bigg[\int_{[0,T]\times\R^d}\int_{\Omega}
		\bigg\vert(\functionANN(\psi_{d,N,M,\delta,\omega}))(t,x)
		\\&-\frac{1}{M}\bigg[\smallsum\limits_{m=1}^M \big(\functionANN\big(\mathfrak{f}^0_{d, \deltaIndex}\big)\big)\big(\affineProcess^{N, d, m, x }_t(\omega)\big)\bigg]\bigg\vert^p 
		\,\P(d\omega) \, \nu_d(dt,dx)\bigg]^{\!\nicefrac{1}{p}}
		\\&\le  \big[2p\max\{C_d,\kappa d^\kappa\}\max\{T,1\}\max\{\kappa,1\}\big]^{2\kappa+3}\left[1+
		\sqrt{ \operatorname{Trace}(2A_d)} \right]^{2\kappa+3}
		\\&\cdot \delta {d}^{\nicefrac{(\kappa+1)}{2}} \kappa d^\kappa  e^{(2\kappa+3)\kappa  T} 2^{2\kappa + 4} \, 3^{\kappa}
	\max\!\big\{1,[\nu_d([0,T]\times\R^d)]^{\nicefrac{1}{p}}\big\},	
		\end{split}
		\end{equation}
		and
		\item \label{DNNerrorEstimateProof:ItemParams}
		it holds for all $d,N,M\in\N$, $\delta\in (0,1]$, $\omega\in \Omega$ that 
		\begin{equation}
		\begin{split}
							\paramANN(\psi_{d,N,M,\delta,\omega})
							&\le 2 M^2 \paramANN\big(\mathfrak{f}^0_{d, \deltaIndex}\big)
+9 M^2 N^6 d^{16} \Big[  2 \lengthANN\big(\mathfrak{f}^1_{d, \deltaIndex}\big)+
							\mathfrak{D}_\delta
							\\& 
							+\big(
							24
							+6\lengthANN\big(\mathfrak{f}^1_{d, \deltaIndex}\big)
							+   \big[4+\paramANN\big(\mathfrak{f}^1_{d, \deltaIndex}\big)\big]^{\!2}\big)^{\!2}
							\Big]^{2}
		\end{split}
		\end{equation}
	\end{enumerate}
	(cf.\ Corollary~\ref{Cor:MCeuler} (applied with $M \is M$, $N \is N$, $d \is d$, $\mathfrak{d} \is 1$, $k \is d$, $p \is p$, $\alpha \is \kappa$, $c \is \kappa$, $C \is \kappa d^\kappa$, $\mathcal{C} \is C_d$, $\mathfrak{C} \is \kappa d^\kappa$, $T \is T$, $\mathfrak{D} \is \mathfrak{D}_\delta$,  $B \is \sqrt{2 A_d}$, $\varepsilon \is \delta$,  $\drift \is \mathfrak{f}^1_{d, \deltaIndex}$, $\initial \is \mathfrak{f}^0_{d, \deltaIndex}$, $\nu \is \nu_d$,
	$ ( \Omega, \mathcal{F}, \P, (\mathbb{F}_t)_{t\in [0,T]} ) \is  ( \Omega, \mathcal{F}, \P, (\mathbb{F}_t)_{t\in [0,T]} )  $, $W^m \is W^{d,m}$, $Y^m \is (Y^{N,d,m,x})_{x\in\R^d}$
for $d,N,M\in\N$, $m\in\{1,2,\dots, M\}$, $\delta\in (0,1]$ in the notation of Corollary~\ref{Cor:MCeuler})),
	let $Z_{d,N,M,\delta}\colon \Omega\to [0,\infty]$, $d,N,M\in\N$, $\delta\in (0,1]$, satisfy for all $d,N,M\in\N$, $\delta\in (0,1]$, $\omega\in\Omega$ that 
	\begin{equation}\label{DNNerrorEstimateProof:DefRV}
	Z_{d,N,M,\delta}(\omega)=\int_{ [0,T]\times\R^d }\left\vert u_d(y)-(\functionANN(\psi_{d,N,M,\delta,\omega}))(y)\right\vert^p
	\nu_d (d y)
	\end{equation}
	(cf.\ \eqref{DNNerrorEstimateProof:Function}), and
 let $\mathcal{C}_3\in (0,\infty)$ satisfy that for all 
$d,N,M\in\N$
it holds that 
\begin{equation}\label{DNNerrorEstimateProof:ErrorAnalysis}
\begin{split}
& \left(
\E\bigg[\int_{[0,T]\times\R^d}
\big|
u_d(t,x) 
- 
\tfrac{ 1 }{ M } 
\big[ 
\textstyle
\sum_{ m = 1 }^{M}
\big(\functionANN\big(\mathfrak{f}^0_{d, \deltaIndex}\big)\big)\big(
\affineProcess^{N, d, m, x }_t
\big)
\big]
\big|^p
\,
\nu_d (d t, d x)
\bigg]
\right)^{\!\nicefrac{1}{p}} 
\\&\le 
\mathcal{C}_3 \!
\left[ \frac{d^{\kappa(\kappa+4)+\max\{\eta,\kappa(2\kappa+1)\}}}{N^{\nicefrac{1}{2}}}
+ \frac{d^{\kappa+\max\{\eta,\kappa^2\}}}{M^{\nicefrac{1}{2}}} \right] \! \left[\max\!\big\{1,\nu_d([0,T]\times\R^d)\big\}\right]^{\!\nicefrac{1}{p}}
\end{split}
\end{equation}
(cf.\ item~\eqref{item:PDE_approxMainStatement} in Proposition~\ref{prop:PDE_approx_Lp} (applied with 
$T \is T$, $\kappa \is \kappa$, $\eta \is \eta$, $p \is p$, $A_d \is A_d$, $\nu_d \is \nu_d$, 
 $f^0_d \is f^0_d$, $f^1_d \is f^1_d$ $\dnnFunction^0_{d, \varepsilon} \is \functionANN(\mathfrak{f}^0_{d, \varepsilon})$, $\dnnFunction^1_{d, \varepsilon} \is \functionANN(\mathfrak{f}^1_{d, \varepsilon})$, $ ( \Omega, \mathcal{F}, \P ) \is  ( \Omega, \mathcal{F}, \P)  $, $W^{d,m} \is W^{d,m}$, $Y^{N,d,m,x} \is Y^{N,d,m,x}$
for $d,N,M\in\N$, $m\in\{1,2,\dots, M\}$, $\varepsilon\in (0,1]$, $x\in\R^d$ in the notation of Proposition~\ref{prop:PDE_approx_Lp})).
Observe that \eqref{DNNerrorEstimate:MeasureAssumption} and \eqref{DNNerrorEstimate:MomentBound} demonstrate that for all $d\in\N$ it holds that
\begin{equation}\label{DNNerrorEstimate:MomentBoundTwo}
\begin{split}
	\max\{C_d,\kappa d^\kappa\}&=\max\!\left\{\left[\int_{[0,T]\times\R^d}\Vert x\Vert^{2p\max\{2\kappa,3\}} \, \nu_d(dt,dx)\right]^{\!\nicefrac{1}{2p\max\{2\kappa,3\}}},1,\kappa d^\kappa\right\}
	\\&\le \max\!\left\{(\eta d^\eta)^{\nicefrac{1}{(2\max\{2\kappa,3\})}},1,\kappa d^\kappa\right\}
	\le \max\!\left\{\eta d^\eta,1,\kappa d^\kappa\right\}
	\\&\le \max\{\eta,1,\kappa\}\, d^{\max\{\eta,\kappa\}}
	= \max\{\eta,\kappa\}\, d^{\max\{\eta,\kappa\}}.
\end{split}
\end{equation}
Next note that \eqref{DNNerrorEstimate:growthMatrix} proves that 
\begin{equation}
\begin{split}
	1+\sqrt{ \operatorname{Trace}(2A_d)}
	\le 1+(2\kappa d^\kappa)^{\nicefrac{1}{2}} \le d^{\nicefrac{\kappa}{2}} [1+(2\kappa)^{\nicefrac{1}{2}}].
\end{split}
\end{equation}
Combining this with \eqref{DNNerrorEstimateProof:Estimate} and \eqref{DNNerrorEstimate:MomentBoundTwo} ensures   that for all $d,N,M\in\N$, $\delta\in (0,1]$ it holds that 
		\begin{equation}
		\begin{split}
		&\bigg[\int_{[0,T]\times\R^d}\int_{\Omega}
		\bigg\vert(\functionANN(\psi_{d,N,M,\delta,\omega}))(t,x)
		\\&-\frac{1}{M}\bigg[\smallsum\limits_{m=1}^M \big(\functionANN\big(\mathfrak{f}^0_{d, \deltaIndex}\big)\big)\big(\affineProcess^{N, d, m, x }_t(\omega)\big)\bigg]\bigg\vert^p 
		\,\P(d\omega) \, \nu_d(dt,dx)\bigg]^{\!\nicefrac{1}{p}}
		\\&\le  \big[2p\max\{\eta,\kappa\} d^{\max\{\eta,\kappa\}}\max\{T,1\}\max\{\kappa,1\}\big]^{2\kappa+3}\left[d^{\nicefrac{\kappa}{2}} [1+(2\kappa)^{\nicefrac{1}{2}}]\right]^{2\kappa+3}
		\\&\cdot \delta {d}^{\nicefrac{(\kappa+1)}{2}} \kappa d^\kappa e^{(2\kappa+3)\kappa  T} 2^{2\kappa + 4} \, 3^{\kappa}
\max\!\big\{1,[\nu_d([0,T]\times\R^d)]^{\nicefrac{1}{p}}\big\}
		\\&\le \delta
		 \big[2p\max\{\eta,\kappa\} \max\{T,1\}\max\{\kappa,1\}\big]^{2\kappa+3} [1+(2\kappa)^{\nicefrac{1}{2}}]^{2\kappa+3}
		   \kappa e^{(2\kappa+3)\kappa  T} 2^{2\kappa + 4} \, 3^{\kappa}
		\\&\cdot d^{(2\kappa+3)\max\{\eta,\kappa\}+\kappa(\kappa+2)+\kappa+\nicefrac{(\kappa+1)}{2}} 
\max\!\big\{1,[\nu_d([0,T]\times\R^d)]^{\nicefrac{1}{p}}\big\}.	
		\end{split}
		\end{equation}
This and \eqref{DNNerrorEstimateProof:DefinitionCone} imply   that for all $d,N,M\in\N$, $\delta\in (0,1]$ it holds that
		\begin{equation}\label{DNNerrorEstimateProof:EstimateTwo}
		\begin{split}
		&\bigg[\int_{[0,T]\times\R^d}\int_{\Omega}
		\bigg\vert(\functionANN(\psi_{d,N,M,\delta,\omega}))(t,x)
		\\&-\frac{1}{M}\bigg[\smallsum\limits_{m=1}^M \big(\functionANN\big(\mathfrak{f}^0_{d, \deltaIndex}\big)\big)\big(\affineProcess^{N, d, m, x }_t(\omega)\big)\bigg]\bigg\vert^p 
		\,\P(d\omega) \, \nu_d(dt,dx)\bigg]^{\!\nicefrac{1}{p}}
						\\&\le \mathcal{C}_1 \delta\,
						 d^{(2\kappa+3)\max\{\eta,\kappa\}+\kappa(\kappa+2)+\kappa+\nicefrac{(\kappa+1)}{2}} 
						\max\!\big\{1,[\nu_d([0,T]\times\R^d)]^{\nicefrac{1}{p}}\big\}
				.	
		\end{split}
		\end{equation}
Furthermore, observe that \eqref{eq:prop:new} shows that for all $ d, N\in\N$, $m\in\{0,1\}$ it holds that
\begin{equation}\label{DNNerrorEstimateProof:AuxParams}
\begin{split}
	\lengthANN\big( \mathfrak{f}^m_{d, \deltaIndex} \big)
	&\le \mathcal{P}\big( \mathfrak{f}^m_{d, \deltaIndex} \big) 
	\leq \kappa d^{ \kappa } [\deltaIndex]^{ - \kappa }
	\\&=\kappa d^{ \kappa } N^{\nicefrac{\kappa}{2}}[\min\{T^{\nicefrac{1}{2}},N^{\nicefrac{1}{2}}\}]^{-\kappa}
	\le\kappa d^{ \kappa } N^{\nicefrac{\kappa}{2}}[\min\{T^{\nicefrac{1}{2}},1\}]^{-\kappa}
	\\&=\kappa d^{ \kappa } N^{\nicefrac{\kappa}{2}} \max\{T^{-\nicefrac{\kappa}{2}},1\}.
\end{split}
\end{equation}
Hence, we obtain  that for all $d,N\in\N$, $\delta\in (0,1]$ it holds that 
\begin{equation}
	\begin{split}
							&24
							+6\lengthANN\big(\mathfrak{f}^1_{d, \deltaIndex}\big)
							+   \big[4+\paramANN\big(\mathfrak{f}^1_{d, \deltaIndex}\big)\big]^{\!2}
\\&\le 							24
+6\paramANN\big(\mathfrak{f}^1_{d, \deltaIndex}\big)
+  25 \big[\paramANN\big(\mathfrak{f}^1_{d, \deltaIndex}\big)\big]^{\!2}
\\&\le 							24
+  31 \big[\paramANN\big(\mathfrak{f}^1_{d, \deltaIndex}\big)\big]^{\!2}
	\\&\le 55 \big[\paramANN\big(\mathfrak{f}^1_{d, \deltaIndex}\big)\big]^{\!2}
		\le 	
		 2^6 \kappa^2 d^{2 \kappa } N^{\kappa} |\!\max\{T^{-\nicefrac{\kappa}{2}},1\}|^2.  
	\end{split}
\end{equation}
This, \eqref{DNNerrorEstimateProof:ItemParams}, and \eqref{DNNerrorEstimateProof:AuxParams}
 establish that  for all $d,N,M\in\N$, $\delta\in (0,1]$, $\omega\in \Omega$ it holds that 
		\begin{equation}\label{DNNerrorEstimate:paramBoundPreFinal}
		\begin{split}
	&\paramANN(\psi_{d,N,M,\delta,\omega})
		\le 2 M^2 \kappa d^{ \kappa } N^{\nicefrac{\kappa}{2}} \max\{T^{-\nicefrac{\kappa}{2}},1\}
		\\&   +9 M^2 N^6 d^{16} \big[  2 \kappa d^{ \kappa } N^{\nicefrac{\kappa}{2}} \max\{T^{-\nicefrac{\kappa}{2}},1\}+
		\mathfrak{D}_\delta
+2^{12}\kappa^4 d^{4 \kappa } N^{2\kappa} |\!\max\{T^{-\nicefrac{\kappa}{2}},1\}|^4
		\big]^{2}
				\\&\le 2 M^2 \kappa d^{ \kappa } N^{\nicefrac{\kappa}{2}} \max\{T^{-\nicefrac{\kappa}{2}},1\}
				\\&+9 M^2 N^6 d^{16} \big[
				(2^{12} + 3) \mathfrak{D}_\delta |\!\max\{\kappa,1\}|^4 d^{4 \kappa } N^{2\kappa} |\!\max\{T^{-\nicefrac{\kappa}{2}},1\}|^4
				\big]^{2}
				\\&\le 2 M^2 \kappa d^{ \kappa } N^{\nicefrac{\kappa}{2}} \max\{T^{-\nicefrac{\kappa}{2}},1\}
				\\&+2^{26
				} 3^2M^2 N^6 d^{16} \big[
				 \mathfrak{D}_\delta |\!\max\{\kappa,1\}|^4 d^{4 \kappa } N^{2\kappa} |\!\max\{T^{-\nicefrac{\kappa}{2}},1\}|^4
				\big]^{2}.
		\end{split}
		\end{equation}
Moreover, observe that \eqref{DNNerrorEstimate:paramBound} proves for all $\delta\in (0,1]$ that 
			\begin{equation}
			\begin{split}
						\mathfrak{D}_\delta = 2160 \LogBin(\delta^{-1}) + 8136 \leq 8136 [\LogBin(\delta^{-1}) + 1]  \leq 2^{13} [\LogBin(\delta^{-1}) + 1].
			\end{split}
			\end{equation}
Combining this with \eqref{DNNerrorEstimate:paramBoundPreFinal} ensures  that for all $d,N,M\in\N$, $\delta\in (0,1]$, $\omega\in \Omega$ it holds that 
		\begin{equation}
		\begin{split}
		&\paramANN(\psi_{d,N,M,\delta,\omega})
	\le 2 M^2 \kappa d^{ \kappa } N^{\nicefrac{\kappa}{2}}\max\{T^{-\nicefrac{\kappa}{2}},1\}
		\\&+2^{26} 3^2 M^2 N^{6+4\kappa} d^{16+8\kappa} 
	 (\mathfrak{D}_\delta)^2[\max\{\kappa,1\}]^8  [\max\{T^{-\nicefrac{\kappa}{2}},1\}]^8
				\\&\le 2^{27} 3^2  M^2 N^{6+4\kappa} d^{16+8\kappa} 
			(\mathfrak{D}_\delta)^2[\max\{\kappa,1\}]^8  [\max\{T^{-\nicefrac{\kappa}{2}},1\}]^8
							\\&\le 2^{31} M^2 N^{6+4\kappa} d^{16+8\kappa} 
			 (\mathfrak{D}_\delta)^2[\max\{\kappa,1\}]^8  [\max\{T^{-\nicefrac{\kappa}{2}},1\}]^8
				\\&\le 2^{57} [\max\{\kappa,1\}]^8  [\max\{T^{-\nicefrac{\kappa}{2}},1\}]^8
				 M^2 N^{6+4\kappa} [\LogBin(\delta^{-1})+1]^{2} d^{16+8\kappa}.
		\end{split}
		\end{equation}
This and \eqref{DNNerrorEstimateProof:DefinitionCtwo} prove that for all $d,N,M\in\N$, $\delta\in (0,1]$, $\omega\in \Omega$ it holds that 
		\begin{equation}\label{DNNerrorEstimateProof:FinalParam}
		\begin{split}
		&\paramANN(\psi_{d,N,M,\delta,\omega})
		\le \mathcal{C}_2 M^2 N^{6+4\kappa} [\LogBin(\delta^{-1})+1 ]^{2} d^{16+8\kappa}.
		\end{split}
		\end{equation}
Next note that \eqref{DNNerrorEstimateProof:Function}, \eqref{DNNerrorEstimateProof:Continuity}, and, e.g., Beck et al.~\cite[Lemma~2.4]{Kolmogorov} show that for all $d,N,M\in\N$, $\delta\in (0,1]$ it holds  that 
$[0,T]\times\R^{d}\times \Omega \ni(t,x,\omega)\mapsto (\functionANN(\psi_{d,N,M,\delta,\omega}))(t,x)\allowbreak\in \R$ is 
measurable. The triangle inequality and Fubini's theorem hence establish that for all $d,N,M\in\N$, $\delta\in (0,1]$
		it holds that
\begin{align*}
&\bigg[\int_\Omega \int_{ [0,T]\times\R^d }\left\vert u_d(y)-(\functionANN(\psi_{d,N,M,\delta,\omega}))(y)\right\vert^p
\nu_d (d y) \, \P(d\omega)\bigg]^{\!\nicefrac{1}{p}} \numberthis
\\&\le
\left(
\E\bigg[\int_{[0,T]\times\R^d}
\bigg|
u_d(t,x) 
- 
\frac{ 1 }{ M } 
\bigg[ 
\textstyle
\sum\limits_{ m = 1 }^{M}
\big(\functionANN\big(\mathfrak{f}^0_{d, \deltaIndex}\big)\big)\big(
\affineProcess^{N, d, m, x }_t
\big)
\bigg]
\bigg|^p
\,
\nu_d (d t, d x)
\bigg]
\right)^{\!\nicefrac{1}{p}}
\\&+
\bigg[ \int_{ [0,T]\times\R^d } \int_\Omega\bigg|(\functionANN(\psi_{d,N,M,\delta,\omega}))(t,x)
\\& -\frac{ 1 }{ M }\bigg[ 
\textstyle
\sum\limits_{ m = 1 }^{M}
\big(\functionANN\big(\mathfrak{f}^0_{d, \deltaIndex}\big)\big)\big(
\affineProcess^{N, d, m, x }_t(\omega)
\big)
\bigg]\bigg|^p \, \P(d\omega) \, \nu_d (d t, d x) \bigg]^{\!\nicefrac{1}{p}}.
\end{align*}
Combining this with  \eqref{DNNerrorEstimateProof:ErrorAnalysis} and \eqref{DNNerrorEstimateProof:EstimateTwo} ensures that for all $d,N,M\in\N$, $\delta\in (0,1]$
it holds that
\begin{equation}
\begin{split}
&\bigg[\int_\Omega \int_{ [0,T]\times\R^d }\left\vert u_d(y)-(\functionANN(\psi_{d,N,M,\delta,\omega}))(y)\right\vert^p
\nu_d (d y) \, \P(d\omega)\bigg]^{\!\nicefrac{1}{p}}
\\&\le
\mathcal{C}_3 \!
\left[ \frac{d^{\kappa(\kappa+4)+\max\{\eta,\kappa(2\kappa+1)\}}}{N^{\nicefrac{1}{2}}}
+ \frac{d^{\kappa+\max\{\eta,\kappa^2\}}}{M^{\nicefrac{1}{2}}} \right]\! \left[\max\!\big\{1,\nu_d([0,T]\times\R^d)\big\}\right]^{\!\nicefrac{1}{p}}
\\&+\mathcal{C}_1 \delta\,
d^{(2\kappa+3)\max\{\eta,\kappa\}+\kappa(\kappa+2)+\kappa+\nicefrac{(\kappa+1)}{2}} 
\max\!\big\{1,[\nu_d([0,T]\times\R^d)]^{\nicefrac{1}{p}}\big\}.
\end{split}
\end{equation}
Hence, we obtain that for all $d,N,M\in\N$, $\delta\in (0,1]$
it holds that
\begin{equation}\label{DNNerrorEstimateProof:FinalExpectation}
\begin{split}
&\bigg[\int_\Omega \int_{ [0,T]\times\R^d }\left\vert u_d(y)-(\functionANN(\psi_{d,N,M,\delta,\omega}))(y)\right\vert^p
\nu_d (d y) \, \P(d\omega)\bigg]^{\!\nicefrac{1}{p}}
\\&\le
\max\{\mathcal{C}_1, \mathcal{C}_3 \} \!\left[\max\!\big\{1,\nu_d([0,T]\times\R^d)\big\}\right]^{\!\nicefrac{1}{p}}
\\&\cdot \left[ \frac{d^{\kappa(\kappa+4)+\max\{\eta,\kappa(2\kappa+1)\}}}{N^{\nicefrac{1}{2}}}
+ \frac{d^{\kappa+\max\{\eta,\kappa^2\}}}{M^{\nicefrac{1}{2}}}
+\delta d^{(2\kappa+3)\max\{\eta,\kappa\}+\kappa(\kappa+2)+\kappa+\nicefrac{(\kappa+1)}{2}} \right]
\!.
\end{split}
\end{equation}
Next note that \eqref{DNNerrorEstimateProof:Function}, \eqref{DNNerrorEstimateProof:Continuity}, and, e.g., Beck et al.~\cite[Lemma~2.4]{Kolmogorov} demonstrate that for all $d,N,M\in\N$, $\delta\in (0,1]$ it holds  that 
$[0,T]\times\R^{d}\times \Omega \ni(t,x,\omega)\mapsto (\functionANN(\psi_{d,N,M,\delta,\omega}))(t,x)\allowbreak\in \R$ is 
measurable. Combining this with Fubini's theorem, \eqref{DNNerrorEstimateProof:DefRV}, \eqref{DNNerrorEstimateProof:FinalExpectation}, and the fact that  $\forall \, d\in\N \colon u_d\in C([0,T]\times \R^d,\R)$  proves that
\begin{enumerate}[A)]
	\item it holds 
	for all $d,N,M\in\N$, $\delta\in (0,1]$  that  $Z_{d,N,M,\delta}$ 
	is a random variable and
	\item it holds 
	for all $d,N,M\in\N$, $\delta\in (0,1]$  that 
\begin{equation}
\begin{split}
&\E \big[\vert Z_{d,N,M,\delta}\vert \big] \le \left[\max\{\mathcal{C}_1, \mathcal{C}_3\}\right]^{p} \max\!\big\{1,\nu_d([0,T]\times\R^d)\big\}
\\&\cdot \left[ \frac{d^{\kappa(\kappa+4)+\max\{\eta,\kappa(2\kappa+1)\}}}{N^{\nicefrac{1}{2}}}
+ \frac{d^{\kappa+\max\{\eta,\kappa^2\}}}{M^{\nicefrac{1}{2}}}
+\delta  d^{(2\kappa+3)\max\{\eta,\kappa\}+\kappa(\kappa+2)+\kappa+\nicefrac{(\kappa+1)}{2}} \right]^p
\!.
\end{split}
\end{equation}
\end{enumerate}
This and, e.g., \cite[Lemma~2.1]{JentzenSalimovaWelti2018}
 prove that there exist $\mathfrak{w}_{d,N,M,\delta} \in \Omega$, $d,N,M\in\N$, $\delta\in (0,1]$, which satisfy that for all $d,N,M\in\N$, $\delta\in (0,1]$ it holds that
\begin{equation}\label{finalEstimateError}
\begin{split}
&\int_{ [0,T]\times\R^d }\left\vert u_d(y)-(\functionANN(\psi_{d,N,M,\delta,\mathfrak{w}_{d,N,M,\delta}}))(y)\right\vert^p
\nu_d (d y)
=Z_{d,N,M,\delta}(\mathfrak{w}_{d,N,M,\delta})
\\&\le \left[\max\{\mathcal{C}_1, \mathcal{C}_3  \}\right]^{p} \max\!\big\{1,\nu_d([0,T]\times\R^d)\big\}
\\&\cdot \left[ \frac{d^{\kappa(\kappa+4)+\max\{\eta,\kappa(2\kappa+1)\}}}{N^{\nicefrac{1}{2}}}
+ \frac{d^{\kappa+\max\{\eta,\kappa^2\}}}{M^{\nicefrac{1}{2}}}
+\delta d^{(2\kappa+3)\max\{\eta,\kappa\}+\kappa(\kappa+2)+\kappa+\nicefrac{(\kappa+1)}{2}} \right]^p
\\&= \left[\max\{\mathcal{C}_1, \mathcal{C}_3  \}\right]^{p} \max\!\big\{1,\nu_d([0,T]\times\R^d)\big\}
\\&\cdot \left[ \frac{d^{\kappa(\kappa+4)+\max\{\eta,\kappa(2\kappa+1)\}}}{N^{\nicefrac{1}{2}}}
+ \frac{d^{\kappa+\max\{\eta,\kappa^2\}}}{M^{\nicefrac{1}{2}}}
+\delta d^{(2\kappa+3)\max\{\eta,\kappa\}+\kappa^2+\nicefrac{(7\kappa+1)}{2}} \right]^p
\!.
\end{split}
\end{equation}
Combining this, \eqref{DNNerrorEstimateProof:Function}, and \eqref{DNNerrorEstimateProof:FinalParam} establishes items~\eqref{DNNerrorEstimate:func}--\eqref{DNNerrorEstimate:Params}.
	This completes the proof of Proposition~\ref{prop:DNNerrorEstimate}.
\end{proof}

%% file: DNNcostEstimates.tex
\subsection{Cost estimates for deep ANNs}
\label{sub:cost}

\begin{theorem}
	\label{theorem:DNNerrorEstimate}
	Let 
	$ T, \kappa,\eta,c \in (0,\infty) $,   $p \in [2, \infty)$ satisfy that 
		\begin{equation}\label{DNNerrorEstimate:DimExponent}
		c=18+12\kappa+4\max\{\eta,\kappa^2\}+4\eta+[2\kappa(\kappa+4)+2\max\{\eta,\kappa(2\kappa+1)\}+2\eta](6+4\kappa),
		\end{equation}
	for every $d \in \N$
	let 
	$
	A_d = ( a_{ d, i, j } )_{ (i, j) \in \{ 1, 2,\dots, d \}^2 } $ $ \in \R^{ d \times d }
	$
	be a symmetric positive semidefinite matrix, 
 for every $d \in \N$ let $\nu_d  \colon \mathcal{B}([0,T]\times\R^d) \to [0,\infty)$ be a finite measure which satisfies that 
	\begin{equation}\label{theoremDNNerrorEstimate:MeasureAssumption}
	\left[\max\!\left\{1,\nu_d([0,T]\times\R^d),\int_{[0,T]\times \R^d} 
	\|x \|^{2p \max\{2\kappa, 3\}}
	\, \nu_d (d t, d x) \right\}\right]^{\!\nicefrac{1}{p}}\leq \eta d^{\eta},
	\end{equation}
for every  $m\in\{0,1\}$, $ d \in \N $
		let
		$ f^m_d \colon \R^d \to \R^{ m d - m + 1 } $
		be a function,
		let 
		$
		( \mathfrak{f}^m_{d, \varepsilon})_{ 
			(m, d, \varepsilon) \in \{ 0, 1 \} \times \N \times (0,1] 
		} 
		\allowbreak\subseteq \ANNs
		$, 
		assume for all
		$ d \in \N $, 
		$ \varepsilon \in (0,1] $, 
		$m\in\{0,1\}$,
		$ 
		x, y \in \R^d
		$
		that 
			\vspace{-1ex}
		\begin{gather}
		\label{eq:theorem:new}
		\functionANN( \mathfrak{f}^{ 0 }_{d, \varepsilon } )
		\in 
		C( \R^d,  \R), \qquad 	\functionANN( \mathfrak{f}^{ 1 }_{d, \varepsilon } )
		\in 
		C( \R^d,  \R^{ d }),  \qquad 	\mathcal{P}( \mathfrak{f}^{ m }_{d, \varepsilon } ) 
		\leq \kappa d^{ \kappa } \varepsilon^{ - \kappa }, \\
		\label{theoremDNNerrorEstimate:growthPhiOne}
		\| 
		f^1_d( x ) 
		- 
		f^1_d( y )
		\|
		\leq 
		\kappa 
		\| x - y \|, \qquad \|
		(\functionANN(\mathfrak{f}^1_{d, \varepsilon}))(x)    
		\|	
		\leq 
		\kappa ( d^{ \kappa } + \| x \| ),
				\\
		\label{theoremDNNerrorEstimate:approximationLocallyLipschitz}\left\vert (\functionANN(\mathfrak{f}^0_{d, \varepsilon})) (x) - (\functionANN(\mathfrak{f}^0_{d, \varepsilon})) (y)\right\vert \leq \kappa d^{\kappa} (1 + \|x\|^{\kappa} + \|y \|^{\kappa})\|x-y\|,
		\\
		\label{theoremDNNerrorEstimate:appr:coef}
		\left\| 
		f^m_d(x) 
		- 
		(\functionANN(\mathfrak{f}^m_{d, \varepsilon})) (x)
		\right\|
		\leq 
		\varepsilon \kappa d^{ \kappa }
		(
		1 + \| x \|^{ \kappa }
		),
			\\ \label{theoremDNNerrorEstimate:growthMatrix}
	\text{and} \qquad	|
		f^0_d( x )
		| 
		+
		\operatorname{Trace}(A_d)
		\leq 
		\kappa d^{ \kappa }
		( 1 + \| x \|^{ \kappa } )
		,
		\end{gather}
and	for every $d \in \N$	let
$u_d \in  \{v \in C([0, T] \times \R^d, \R) \colon \allowbreak \inf_{ q \in (0,\infty) }
\allowbreak \sup_{ (t,y) \in [0,T] \times \R^d } \allowbreak
\frac{ | v(t,y) | }{
	1 + \| y \|^q
}
< \infty \}$ be a viscosity solution of
	\begin{equation}
	\begin{split}
	( \tfrac{ \partial }{\partial t} u_d )( t, x ) 
	& = 
	( \tfrac{ \partial }{\partial x} u_d )( t, x )
	\,
	f^1_d( x )
	+
	\textstyle{\sum\limits_{ i, j = 1 }^d}
	a_{ d, i, j }
	\,
	( \tfrac{ \partial^2 }{ \partial x_i \partial x_j } u_d )( t, x )
	\end{split}
	\end{equation}
	with $ u_d( 0, x ) = f^0_d( x )$
	for $ ( t, x ) \in (0,T) \times \R^d $ (cf.\ \Cref{Def:euclideanNorm,Def:ANN,Definition:ANNrealization,Definition:Relu1}).
	Then there exist $\mathfrak{C}\in \R$ and $(\mathfrak{u}_{d,\varepsilon})_{d\in\N,\varepsilon\in (0,1]}\subseteq \ANNs$ such that for all 
$d\in\N$, $\varepsilon\in (0,1]$
				it holds  that $\functionANN(\mathfrak{u}_{d,\varepsilon})\in C(\R^{d+1},\R)$,
		$\paramANN(\mathfrak{u}_{d,\varepsilon})
		\le \mathfrak{C} \varepsilon^{-(18+8\kappa)}  d^{c}$,
		and
		\begin{equation}
		\label{eq:thm:cost}
		\begin{split}
		& \left[
		\int_{[0,T]\times\R^d}
		\left\vert
		u_d(y) 
		- 
		(\functionANN(\mathfrak{u}_{d,\varepsilon}))(y)
		\right\vert^p
		\nu_d (d y)
		\right]^{\!\nicefrac{1}{p}} 
		\le
		\varepsilon.
		\end{split}
		\end{equation}
\end{theorem}

\begin{proof}[Proof of Theorem~\ref{theorem:DNNerrorEstimate}]
Note that 	\eqref{theoremDNNerrorEstimate:MeasureAssumption} implies that
	\begin{equation}\
\left[\max\!\left\{1,\nu_1([0,T]\times\R^d),\int_{[0,T]\times \R} 
\|x \|^{2p \max\{2\kappa, 3\}}
\, \nu_1 (d t, d x) \right\}\right]^{\!\nicefrac{1}{p}}\leq \eta.
\end{equation}
This proves that $\eta \in [1, \infty)$. 	
 Proposition~\ref{prop:DNNerrorEstimate} hence ensures that there exist $\mathcal{C}_1 \in (0, \infty)$ and 
 $(\Phi_{d,N,M,\delta})_{(d,N,M, \delta) \in \N^3 \times (0,1]}\subseteq \ANNs$ which  satisfy that
		\begin{enumerate}[(I)]
			\item\label{thmDNNerrorEstimate:LoadContinuity} it holds for all 
			$d,N,M\in\N$, $\delta\in (0,1]$
		 that $\functionANN(\Phi_{d,N,M,\delta})\in C(\R^{d+1},\R)$,
			\item it holds for all 
			$d,N,M\in\N$, $\delta\in (0,1]$
			 that 
			\begin{equation}\label{thmDNNerrorEstimate:BasisEstimate}
			\begin{split}
			& \left[
			\int_{[0,T]\times\R^d}
			\left\vert
			u_d(y) 
			- 
			(\functionANN(\Phi_{d,N,M,\delta}))(y)
			\right\vert^p
			\nu_d (d y)
			\right]^{\!\nicefrac{1}{p}} 
			\\&\le
			\mathcal{C}_1 \! \left[\max\!\big\{1,\nu_d([0,T]\times\R^d)\big\}\right]^{\!\nicefrac{1}{p}}
			\\&\cdot \left[ \frac{d^{\kappa(\kappa+4)+\max\{\eta,\kappa(2\kappa+1)\}}}{N^{\nicefrac{1}{2}}}
			+ \frac{d^{\kappa+\max\{\eta,\kappa^2\}}}{M^{\nicefrac{1}{2}}}
			+\delta d^{(2\kappa+3)\max\{\eta,\kappa\}+\kappa^2+\nicefrac{(7\kappa+1)}{2}} \right]
			\!,
			\end{split}
			\end{equation}
			and
			\item \label{thmDNNerrorEstimate:itemParam}
			it holds for all 
			$d,N,M\in\N$, $\delta\in (0,1]$
		 that
			\begin{equation}
			\begin{split}
			&\paramANN(\Phi_{d,N,M,\delta})
			\le \mathcal{C}_1 M^2 N^{6+4\kappa} [\LogBin(\delta^{-1})+1 ]^{2} d^{16+8\kappa}.
			\end{split}
			\end{equation}
		\end{enumerate}	
Next
let	$\mathcal{C}_2, \mathfrak{C} \in (0,\infty)$  satisfy that
\begin{equation}\label{thmDNNerrorEstimate:thirdConstant}
\mathcal{C}_2=\max\!\left\{0,\log_2(3\mathcal{C}_1\eta)\right\}+\tfrac{1}{\ln(2)}+\tfrac{1}{\ln(2)}\,\big[2(\kappa+2)\max\{\eta,\kappa\}+\kappa^2+\nicefrac{(7\kappa+1)}{2}\big]
\end{equation}	
and
\begin{equation}\label{thmDNNerrorEstimate:totalConstant}
	\mathfrak{C}=\mathcal{C}_1 2^{8+4\kappa}  (3\mathcal{C}_1\eta)^{16+8\kappa} [\mathcal{C}_2+1]^2,
\end{equation}
let $\mathcal{D}_{d,\varepsilon} \in  (0,1]$, $\varepsilon\in (0,1]$, $d \in \N$, 
satisfy for all $d\in\N$, $\varepsilon\in (0,1]$ that 
\begin{equation}\label{thmDNNerrorEstimate:ProductAccuracy}
\mathcal{D}_{d,\varepsilon}=\min\!\big\{1,(3\mathcal{C}_1\eta)^{-1}\varepsilon d^{-2(\kappa+2)\max\{\eta,\kappa\}-\kappa^2-\nicefrac{(7\kappa+1)}{2}}\big\},
\end{equation}
let $\mathfrak{M}_{d,\varepsilon} \in \N$, $\varepsilon\in (0,1]$, $d \in \N$, 
satisfy for all $d\in\N$, $\varepsilon\in (0,1]$ that 
\begin{equation}\label{thmDNNerrorEstimate:MCAccuracy}
\mathfrak{M}_{d,\varepsilon}=\min\!\big[\N\cap \!\big[(3\mathcal{C}_1\eta)^2 \varepsilon^{-2}d^{2\kappa+2\max\{\eta,\kappa^2\}+2\eta},\infty\big)\big],
\end{equation} 
let  $\mathfrak{N}_{d,\varepsilon} \in \N$, $\varepsilon\in (0,1]$, $d \in \N$, 
satisfy for all $d\in\N$, $\varepsilon\in (0,1]$ that 
	\begin{equation}\label{thmDNNerrorEstimate:EulerAccuracy}
\mathfrak{N}_{d,\varepsilon}=\min\!\left[\N\cap \!\big[(3\mathcal{C}_1\eta)^2 \varepsilon^{-2}d^{2\kappa(\kappa+4)+2\max\{\eta,\kappa(2\kappa+1)\}+2\eta},\infty\big)\right]\!,
\end{equation}
and let  $\mathfrak{u}_{d,\varepsilon}\in \ANNs$, $\varepsilon\in (0,1]$, $d \in \N$, 
satisfy for all $d\in\N$, $\varepsilon\in (0,1]$ that 
	\begin{equation}\label{thmDNNerrorEstimate:ChoiceANN}
	\mathfrak{u}_{d,\varepsilon}=\Phi_{d,\mathfrak{N}_{d,\varepsilon},\mathfrak{M}_{d,\varepsilon},\mathcal{D}_{d,\varepsilon}}.
	\end{equation}	
Observe that \eqref{theoremDNNerrorEstimate:MeasureAssumption} and \eqref{thmDNNerrorEstimate:BasisEstimate} ensure that for all 
$d,N,M\in\N$, $\delta\in (0,1]$
it holds that 
\begin{equation}
\begin{split}
& \left[
\int_{[0,T]\times\R^d}
\left\vert
u_d(y) 
- 
(\functionANN(\Phi_{d,N,M,\delta}))(y)
\right\vert^p
\nu_d (d y)
\right]^{\!\nicefrac{1}{p}} 
\\&\le
\mathcal{C}_1 \eta \!
 \left[ \frac{d^{\kappa(\kappa+4)+\max\{\eta,\kappa(2\kappa+1)\}+\eta}}{N^{\nicefrac{1}{2}}}
+ \frac{d^{\kappa+\max\{\eta,\kappa^2\}+\eta}}{M^{\nicefrac{1}{2}}}
+\delta d^{2(\kappa+2)\max\{\eta,\kappa\}+\kappa^2+\nicefrac{(7\kappa+1)}{2}} \right]
\!.
\end{split}
\end{equation}	
Combining this, \eqref{thmDNNerrorEstimate:ChoiceANN}, \eqref{thmDNNerrorEstimate:ProductAccuracy}, \eqref{thmDNNerrorEstimate:MCAccuracy}, and  \eqref{thmDNNerrorEstimate:EulerAccuracy} implies that for all $d\in\N$, $\varepsilon\in (0,1]$ it holds that 
\begin{equation}\label{thmDNNerrorEstimate:errorFinal}
\begin{split}
& \left[
\int_{[0,T]\times\R^d}
\left\vert
u_d(y) 
- 
(\functionANN(\mathfrak{u}_{d,\varepsilon}))(y)
\right\vert^p
\nu_d (d y)
\right]^{\!\nicefrac{1}{p}} 
\\&= \left[
\int_{[0,T]\times\R^d}
\left\vert
u_d(y) 
- 
(\functionANN(\Phi_{d,\mathfrak{N}_{d,\varepsilon},\mathfrak{M}_{d,\varepsilon},\mathcal{D}_{d,\varepsilon}}))(y)
\right\vert^p
\nu_d (d y)
\right]^{\!\nicefrac{1}{p}} 
\\&\le
\mathcal{C}_1 \eta \!
\left[ \frac{d^{\kappa(\kappa+4)+\max\{\eta,\kappa(2\kappa+1)\}+\eta}}{\vert\mathfrak{N}_{d,\varepsilon}\vert^{\nicefrac{1}{2}}}
+ \frac{d^{\kappa+\max\{\eta,\kappa^2\}+\eta}}{\vert\mathfrak{M}_{d,\varepsilon}\vert^{\nicefrac{1}{2}}}
+\mathcal{D}_{d,\varepsilon}  d^{2(\kappa+2)\max\{\eta,\kappa\}+\kappa^2+\nicefrac{(7\kappa+1)}{2}} \right]
\\&\le
\mathcal{C}_1 \eta
\Bigg[ \frac{d^{\kappa(\kappa+4)+\max\{\eta,\kappa(2\kappa+1)\}+\eta}}{\left\vert(3\mathcal{C}_1\eta)^2 \varepsilon^{-2}d^{2\kappa(\kappa+4)+2\max\{\eta,\kappa(2\kappa+1)\}+2\eta}\right\vert^{\nicefrac{1}{2}}}
+ \frac{d^{\kappa+\max\{\eta,\kappa^2\}+\eta}}{\left\vert(3\mathcal{C}_1\eta)^2 \varepsilon^{-2}d^{2\kappa+2\max\{\eta,\kappa^2\}+2\eta}\right\vert^{\nicefrac{1}{2}}}
\\&+(3\mathcal{C}_1\eta)^{-1}\varepsilon d^{-2(\kappa+2)\max\{\eta,\kappa\}-\kappa^2-\nicefrac{(7\kappa+1)}{2}}  d^{2(\kappa+2)\max\{\eta,\kappa\}+\kappa^2+\nicefrac{(7\kappa+1)}{2}} \Bigg]
\\&=\mathcal{C}_1\eta \! \left(\frac{1}{3\mathcal{C}_1\eta \varepsilon^{-1}}+\frac{1}{3\mathcal{C}_1\eta \varepsilon^{-1}}+\frac{\varepsilon}{3\mathcal{C}_1\eta }\right)
=\varepsilon.
\end{split}
\end{equation}
In addition, observe that \eqref{thmDNNerrorEstimate:thirdConstant}, \eqref{thmDNNerrorEstimate:ProductAccuracy}, and the fact that  $\forall \, x\in [1,\infty) \colon \log_2(x) = \nicefrac{\ln(x)}{\ln(2)} \le \nicefrac{x}{\ln(2)}$  demonstrate that for all $d\in\N$, $\varepsilon\in (0,1]$ it holds that
\begin{equation}
\begin{split}
&\log_2((\mathcal{D}_{d,\varepsilon})^{-1})
=\max\!\big\{0,\log_2\!\big(3\mathcal{C}_1\eta\,\varepsilon^{-1}\, d^{2(\kappa+2)\max\{\eta,\kappa\}+\kappa^2+\nicefrac{(7\kappa+1)}{2}}\big)\big\}
\\&\le \max\!\left\{0,\log_2(3\mathcal{C}_1\eta)+\log_2(\varepsilon^{-1})+\big[2(\kappa+2)\max\{\eta,\kappa\}+\kappa^2+\nicefrac{(7\kappa+1)}{2}\big]\log_2(d)\right\}
\\&\le \max\!\left\{0,\log_2(3\mathcal{C}_1\eta)\right\}+\tfrac{1}{\ln(2)}\, \varepsilon^{-1}+\tfrac{1}{\ln(2)}\,\big[2(\kappa+2)\max\{\eta,\kappa\}+\kappa^2+\nicefrac{(7\kappa+1)}{2}\big]d
\\&\le  \varepsilon^{-1}d \big(\max\!\left\{0,\log_2(3\mathcal{C}_1\eta)\right\}+\tfrac{1}{\ln(2)}+\tfrac{1}{\ln(2)}\,\big[2(\kappa+2)\max\{\eta,\kappa\}+\kappa^2+\nicefrac{(7\kappa+1)}{2}\big]\big)
\\&= \varepsilon^{-1}d\, \mathcal{C}_2.
\end{split}
\end{equation}	
Combining this with \eqref{thmDNNerrorEstimate:itemParam},  \eqref{thmDNNerrorEstimate:MCAccuracy}, \eqref{thmDNNerrorEstimate:EulerAccuracy}, and \eqref{thmDNNerrorEstimate:ChoiceANN} proves that  for all $d\in\N$, $\varepsilon\in (0,1]$ it holds that
\begin{equation}
\begin{split}
\paramANN(\mathfrak{u}_{d,\varepsilon}) &= \paramANN(\Phi_{d,\mathfrak{N}_{d,\varepsilon},\mathfrak{M}_{d,\varepsilon},\mathcal{D}_{d,\varepsilon}})
\le \mathcal{C}_1 (\mathfrak{M}_{d,\varepsilon})^2 d^{16+8\kappa}(\mathfrak{N}_{d,\varepsilon})^{6+4\kappa} [\log_2((\mathcal{D}_{d,\varepsilon})^{-1})+1]^2
\\&\le \mathcal{C}_1 \big[(3\mathcal{C}_1\eta)^2 \varepsilon^{-2}d^{2\kappa+2\max\{\eta,\kappa^2\}+2\eta}+1\big]^{\!2} d^{16+8\kappa} [\varepsilon^{-1}d\, \mathcal{C}_2+1]^2
\\&\cdot\big[(3\mathcal{C}_1\eta)^2 \varepsilon^{-2}d^{2\kappa(\kappa+4)+2\max\{\eta,\kappa(2\kappa+1)\}+2\eta}+1\big]^{\!6+4\kappa} 
\\&\le \mathcal{C}_1 2^{8+4\kappa} (3\mathcal{C}_1\eta)^4 \varepsilon^{-4}d^{4\kappa+4\max\{\eta,\kappa^2\}+4\eta} d^{16+8\kappa} \varepsilon^{-2}d^2 [\mathcal{C}_2+1]^2
\\&\cdot(3\mathcal{C}_1\eta)^{12+8\kappa} \varepsilon^{-(12+8\kappa)}d^{[2\kappa(\kappa+4)+2\max\{\eta,\kappa(2\kappa+1)\}+2\eta](6+4\kappa)}. 
\end{split}
\end{equation}
This, \eqref{DNNerrorEstimate:DimExponent}, and \eqref{thmDNNerrorEstimate:totalConstant} ensure that for all $d\in\N$, $\varepsilon\in (0,1]$ it holds that
\begin{equation}
\begin{split}
\paramANN(\mathfrak{u}_{d,\varepsilon})
&\le \mathcal{C}_1 2^{8+4\kappa}  (3\mathcal{C}_1\eta)^{16+8\kappa} \varepsilon^{-(18+8\kappa)} [\mathcal{C}_2+1]^2
\\&\cdot d^{2+4\kappa+4\max\{\eta,\kappa^2\}+4\eta+16+8\kappa+[2\kappa(\kappa+4)+2\max\{\eta,\kappa(2\kappa+1)\}+2\eta](6+4\kappa)}
\\&= \mathfrak{C} \varepsilon^{-(18+8\kappa)} 
d^{18+12\kappa+4\max\{\eta,\kappa^2\}+4\eta+[2\kappa(\kappa+4)+2\max\{\eta,\kappa(2\kappa+1)\}+2\eta](6+4\kappa)}
\\&= \mathfrak{C} \varepsilon^{-(18+8\kappa)} 
d^{c}
.  
\end{split}
\end{equation}
Combining this, \eqref{thmDNNerrorEstimate:LoadContinuity}, and \eqref{thmDNNerrorEstimate:errorFinal}
establishes \eqref{eq:thm:cost}. This
completes
the proof of Theorem~\ref{theorem:DNNerrorEstimate}.
\end{proof}

\begin{cor}
	\label{cor:DNNerrorEstimate}
	Let 
	$ T, \kappa,\eta \in (0,\infty)$,   $p \in [2, \infty)$,
	let 
	$
	A_d = ( a_{ d, i, j } )_{ (i, j) \in \{ 1, 2,\dots, d \}^2 } $ $ \in \R^{ d \times d }
	$,
	$ d \in \N $,
	be symmetric positive semidefinite matrices, 
 let $\nu_d  \colon \mathcal{B}([0,T]\times\R^d) \to [0,\infty)$, $d\in\N$, be finite measures which satisfy for all $d\in\N$ that 
	\begin{equation}\label{corDNNerrorEstimate:MeasureAssumption}
	\left[\max\!\left\{1,\nu_d([0,T]\times\R^d),\int_{[0,T]\times \R^d} 
	\|x \|^{2p \max\{6\kappa, 2\kappa +2, 3\}}
	\, \nu_d (d t, d x) \right\}\right]^{\!\nicefrac{1}{p}}\leq \eta d^{\eta},
	\end{equation}
		let
		$ f^m_d \colon \R^d \to \R^{ m d - m + 1 } $, $m\in\{0,1\}$, $ d \in \N $, 
		be functions,
		let 
		$
		( \mathfrak{f}^m_{d, \varepsilon})_{ 
			(m, d, \varepsilon) \in \{ 0, 1 \} \times \N \times (0,1] 
		} 
		\allowbreak\subseteq \ANNs
		$,
		assume for all
		$m\in\{0,1\}$,
		$ d \in \N $, 
		$ \varepsilon \in (0,1] $, 
		 $i \in\{1,2,\dots, d\}$,
		$ 
		x,y \in \R^d
		$
		that 
		\vspace{-1ex}
		\begin{gather}
			\label{eq:corol:new}
		\functionANN( \mathfrak{f}^{ 0 }_{d, \varepsilon } )
		\in 
		C( \R^d,  \R), \qquad 	\functionANN( \mathfrak{f}^{ 1 }_{d, \varepsilon } )
		\in 
		C( \R^d,  \R^{ d }),  \qquad 	\mathcal{P}( \mathfrak{f}^{ m }_{d, \varepsilon } ) 
		\leq \kappa d^{ \kappa } \varepsilon^{ - \kappa }, \\
		\label{corDNNerrorEstimate:growthfOne}
		\| 
		f^1_d( x ) 
		- 
		f^1_d( y )
		\|
		\leq 
		\kappa 
		\| x - y \|, \qquad \|
		(\functionANN(\mathfrak{f}^1_{d, \varepsilon}))(x)    
		\|	
		\leq 
		\kappa ( d^{ \kappa } + \| x \| ),
		\\
		\label{corDNNerrorEstimate:appr:coef}
		\varepsilon|
		f^0_d( x )
		| 
		+
		\varepsilon | a_{ d, i, i } |
		+\| f^m_d(x) 
		- 
		(\functionANN(\mathfrak{f}^m_{d, \varepsilon})) (x)
		\|
		\leq 
		\varepsilon \kappa d^{ \kappa }
		(
		1 + \| x \|^{ \kappa }
		),
		\\
		\label{corDNNerrorEstimate:approximationLocallyLipschitz}
		\text{and} \qquad | (\functionANN(\mathfrak{f}^0_{d, \varepsilon})) (x) - (\functionANN(\mathfrak{f}^0_{d, \varepsilon})) (y)| \leq \kappa d^{\kappa} (1 + \|x\|^{\kappa} + \|y \|^{\kappa})\|x-y\|,
		\end{gather}
and	for every $d \in \N$	let
$u_d \in  \{v \in C([0, T] \times \R^d, \R) \colon \allowbreak \inf_{ q \in (0,\infty) }
\allowbreak \sup_{ (t,y) \in [0,T] \times \R^d } \allowbreak
\frac{ | v(t,y) | }{
	1 + \| y \|^q
}
< \infty \}$ be a viscosity solution of
	\begin{equation}
	\begin{split}
	( \tfrac{ \partial }{\partial t} u_d )( t, x ) 
	& = 
	( \tfrac{ \partial }{\partial x} u_d )( t, x )
	\,
	f^1_d( x )
	+
	\textstyle{\sum\limits_{ i, j = 1 }^d}
	a_{ d, i, j }
	\,
	( \tfrac{ \partial^2 }{ \partial x_i \partial x_j } u_d )( t, x )
	\end{split}
	\end{equation}
	with $ u_d( 0, x ) = f^0_d( x )$
	for $ ( t, x ) \in (0,T) \times \R^d $ (cf.\ \Cref{Def:euclideanNorm,Def:ANN,Definition:ANNrealization,Definition:Relu1}).
	Then there exist $\mathfrak{C}\in \R$ and $(\mathfrak{u}_{d,\varepsilon})_{d\in\N,\varepsilon\in (0,1]}\subseteq \ANNs$ such that
	for all 
	$d\in\N$, $\varepsilon\in (0,1]$
	it holds 
	that
	$\functionANN(\mathfrak{u}_{d,\varepsilon})\in C(\R^{d+1},\R)$,
	$\paramANN(\mathfrak{u}_{d,\varepsilon})\le \mathfrak{C} \varepsilon^{-\mathfrak{C}}  d^{\mathfrak{C}}$, and 
			\begin{equation}
			\label{eq:cor:dnns}
			\begin{split}
			& \left[
			\int_{[0,T]\times\R^d}
			\left\vert
			u_d(y) 
			- 
			(\functionANN(\mathfrak{u}_{d,\varepsilon}))(y)
			\right\vert^p
			\nu_d (d y)
			\right]^{\!\nicefrac{1}{p}} 
			\le
			\varepsilon.
			\end{split}
			\end{equation}
\end{cor}

\begin{proof}[Proof of Corollary~\ref{cor:DNNerrorEstimate}]
	Throughout this proof let $\iota=\max\{3\kappa,\kappa+1\}$.
	 Observe that \eqref{corDNNerrorEstimate:growthfOne} and the fact that $\iota\ge \kappa$ prove that for all $d\in\N$, $\varepsilon\in (0,1]$, $x,y\in\R^d$ it holds that 
	\begin{equation}\label{corDNNerrorEstimate:HypoOne}
			\| 
			f^1_d( x ) 
			- 
			f^1_d( y )
			\|
			\leq 
			\iota
			\| x - y \|
			\qandq
			\|(\functionANN(\mathfrak{f}^1_{d, \varepsilon}))(x)    
			\|	
			\leq 
			\iota ( d^{ \iota } + \| x \| ).
	\end{equation}
	Next note that \eqref{corDNNerrorEstimate:approximationLocallyLipschitz}  and the fact that $\iota\ge 3\kappa$ ensure that for all $d\in\N$, $\varepsilon\in (0,1]$, $x,y\in\R^d$ it holds that 
	\begin{equation}\label{corDNNerrorEstimate:HypoTwo}
	\begin{split}
			| (\functionANN(\mathfrak{f}^0_{d, \varepsilon})) (x) - (\functionANN(\mathfrak{f}^0_{d, \varepsilon})) (y)| 
			&\leq \kappa d^{\iota} (3 + \|x\|^{\iota} + \|y \|^{\iota})\|x-y\|
			\\& \leq \iota d^{\iota} (1 + \|x\|^{\iota} + \|y \|^{\iota})\|x-y\|.
	\end{split}
	\end{equation}
	In addition, observe that \eqref{corDNNerrorEstimate:appr:coef} and the fact that $\iota\ge 2\kappa$ imply that for all $m\in\{0,1\}$, $d\in\N$, $\varepsilon\in (0,1]$, $x\in\R^d$ it holds that 
	\begin{equation}\label{corDNNerrorEstimate:HypoThree}
	\begin{split}
			\| f^m_d(x) 
			- 
			(\functionANN(\mathfrak{f}^m_{d, \varepsilon})) (x)
			\|
			&\leq 
			\varepsilon \kappa d^{ \kappa }
			(
			1 + \| x \|^{ \kappa }
			)
			\\&\leq 
			\varepsilon \kappa d^{ \iota }
			(
			2 + \| x \|^{\iota }
			)
			\leq 
			\varepsilon \iota d^{ \iota }
			(
			1 + \| x \|^{\iota }
			).
	\end{split}
	\end{equation}
Moreover, note that \eqref{corDNNerrorEstimate:appr:coef} and the fact that $\iota\ge \max\{2\kappa, \kappa+1\}$ demonstrate that for all  $d\in\N$, $\varepsilon\in (0,1]$, $x\in\R^d$ it holds that 
\begin{equation}\label{corDNNerrorEstimate:HypoFour}
\begin{split}
	|
	f^0_d( x )
	| 
	+
\operatorname{Trace}(A_d)
	&\leq 
	\kappa d^{ \kappa +1}
	( 1 + \| x \|^{ \kappa } )
	\leq 
	\kappa d^{ \kappa +1}
	( 2 + \| x \|^{ \iota } )
	\\&\leq 
	2\kappa d^{ \kappa +1}
	( 1 + \| x \|^{ \iota } )
	\leq 
	\iota d^{ \iota}
	( 1 + \| x \|^{ \iota } )
	.
\end{split}
\end{equation}	
Furthermore, observe that \eqref{eq:corol:new} implies that   for all
$m\in\{0,1\}$,
$ d \in \N $, 
$ \varepsilon \in (0,1] $
it holds that
$\mathcal{P}( \mathfrak{f}^m_{d, \varepsilon}) 
\leq \iota d^{\iota } \varepsilon^{ - \iota }$.
Combining this, \eqref{corDNNerrorEstimate:MeasureAssumption},  \eqref{corDNNerrorEstimate:HypoOne}, \eqref{corDNNerrorEstimate:HypoTwo}, \eqref{corDNNerrorEstimate:HypoThree}, \eqref{corDNNerrorEstimate:HypoFour},  and Theorem~\ref{theorem:DNNerrorEstimate} establishes that there exist	
 $\mathcal{C}\in \R$ and $(\mathfrak{u}_{d,\varepsilon})_{d\in\N,\varepsilon\in (0,1]}\subseteq \ANNs$ which satisfy that
\begin{enumerate}[(I)]
	\item 	\label{corDNNerrorEstimateProof:Continuity} it holds 
	 for all 
	$d\in\N$, $\varepsilon\in (0,1]$
 that $\functionANN(\mathfrak{u}_{d,\varepsilon})\in C(\R^{d+1},\R)$,
	\item
	\label{corDNNerrorEstimateProof:Estimate}
	it holds for all 
	$d\in\N$, $\varepsilon\in (0,1]$
 that 
	\begin{equation}
	\begin{split}
	& \left[
	\int_{[0,T]\times\R^d}
	\left\vert
	u_d(y) 
	- 
	(\functionANN(\mathfrak{u}_{d,\varepsilon}))(y)
	\right\vert^p
	\nu_d (d y)
	\right]^{\!\nicefrac{1}{p}} 
	\le
	\varepsilon,
	\end{split}
	\end{equation}
	and
	\item \label{corDNNerrorEstimateProof:Param}
	it holds for all 
	$d\in\N$, $\varepsilon\in (0,1]$
 that
	\begin{equation}
	\begin{split}
	&\paramANN(\mathfrak{u}_{d,\varepsilon})
\le \mathcal{C} \varepsilon^{-(18+8\iota)}  d^{18+12\iota+4\max\{\eta,\iota^2\}+4\eta+[2\iota(\iota+4)+2\max\{\eta,\iota(2\iota+1)\}+2\eta](6+4\iota)}. 
	\end{split}
	\end{equation}
\end{enumerate}
This proves that for all $\mathfrak{C} \in [\max\{\mathcal{C}, 18+12\iota+4\max\{\eta,\iota^2\}+4\eta+[2\iota(\iota+4)+2\max\{\eta,\iota(2\iota+1)\}+2\eta](6+4\iota) \}, \infty)$, 	$d\in\N$, $\varepsilon\in (0,1]$ it holds that 	$\paramANN(\mathfrak{u}_{d,\varepsilon})\le \mathfrak{C} \varepsilon^{-\mathfrak{C}}  d^{\mathfrak{C}}$. Combining this, \eqref{corDNNerrorEstimateProof:Continuity},  and \eqref{corDNNerrorEstimateProof:Estimate} 
 establishes \eqref{eq:cor:dnns}. 
This  completes the proof of	Corollary~\ref{cor:DNNerrorEstimate}.
\end{proof}

\begin{cor}
	\label{cor:DNNerrorEstimateLebesgueMeasure}
	Let 
	$ T, \kappa \in (0,\infty) $,   $\alpha \in \R$, $\beta \in (\alpha, \infty)$,
	let 
	$
	A_d = ( a_{ d, i, j } )_{ (i, j) \in \{ 1,2, \dots, d \}^2 } $ $ \in \R^{ d \times d }
	$,
	$ d \in \N $,
	be symmetric positive semidefinite matrices, 
	let
	$ f^m_d \colon \R^d \to \R^{ m d - m + 1 } $, $m\in\{0,1\}$, $ d \in \N $, 
	be functions,
	let 
	$
	( \mathfrak{f}^m_{d, \varepsilon})_{ 
		(m, d, \varepsilon) \in \{ 0, 1 \} \times \N \times (0,1] 
	} 
	\subseteq \ANNs
	$, 
	assume for all
	$m\in\{0,1\}$,
	$ d \in \N $, 
	$ \varepsilon \in (0,1] $, 
	$i \in\{1,2, \dots, d\}$,
	$ 
	x,y \in \R^d
	$
	that 
	\vspace{-1ex}
	\begin{gather}
	\functionANN( \mathfrak{f}^{ 0 }_{d, \varepsilon } )
	\in 
	C( \R^d,  \R), \qquad 	\functionANN( \mathfrak{f}^{ 1 }_{d, \varepsilon } )
	\in 
	C( \R^d,  \R^{ d }),  \qquad 	\mathcal{P}( \mathfrak{f}^{ m }_{d, \varepsilon } ) 
	\leq \kappa d^{ \kappa } \varepsilon^{ - \kappa }, \\
	\label{corTwoDNNerrorEstimate:growthPhiOne}
	\| 
	f^1_d( x ) 
	- 
	f^1_d( y )
	\|
	\leq 
	\kappa 
	\| x - y \|, \qquad \|
	(\functionANN(\mathfrak{f}^1_{d, \varepsilon}))(x)    
	\|	
	\leq 
	\kappa ( d^{ \kappa } + \| x \| ),
	\\
	\label{corTwoDNNerrorEstimate:appr:coef}
	\varepsilon|
	f^0_d( x )
	| 
	+
	\varepsilon | a_{ d, i, i } |
	+\| f^m_d(x) 
	- 
	(\functionANN(\mathfrak{f}^m_{d, \varepsilon})) (x)
	\|
	\leq 
	\varepsilon \kappa d^{ \kappa }
	(
	1 + \| x \|^{ \kappa }
	),
	\\
	\label{corTwoDNNerrorEstimate:approximationLocallyLipschitz}
	\text{and} \qquad | (\functionANN(\mathfrak{f}^0_{d, \varepsilon})) (x) - (\functionANN(\mathfrak{f}^0_{d, \varepsilon})) (y)| \leq \kappa d^{\kappa} (1 + \|x\|^{\kappa} + \|y \|^{\kappa})\|x-y\|,
	\end{gather}
	and	for every $d \in \N$	let
	$u_d \in  \{v \in C([0, T] \times \R^d, \R) \colon \allowbreak \inf_{ q \in (0,\infty) }
	\allowbreak \sup_{ (t,y) \in [0,T] \times \R^d } \allowbreak
	\frac{ | v(t,y) | }{
		1 + \| y \|^q
	}
	< \infty \}$ be a viscosity solution of
	\begin{equation}
	\begin{split}
	( \tfrac{ \partial }{\partial t} u_d )( t, x ) 
	& = 
	( \tfrac{ \partial }{\partial x} u_d )( t, x )
	\,
	f^1_d( x )
	+
	\textstyle{\sum\limits_{ i, j = 1 }^d}
	a_{ d, i, j }
	\,
	( \tfrac{ \partial^2 }{ \partial x_i \partial x_j } u_d )( t, x )
	\end{split}
	\end{equation}
		with $ u_d( 0, x ) = f^0_d( x )$
	for $ ( t, x ) \in (0,T) \times \R^d $ (cf.\ \Cref{Def:euclideanNorm,Def:ANN,Definition:ANNrealization,Definition:Relu1}).
	Then for every $p \in (0, \infty)$ there exist $c\in \R$ and $(\mathfrak{u}_{d,\varepsilon})_{d\in\N,\varepsilon\in (0,1]}\subseteq \ANNs$ such that
	for all 
	$d\in\N$, $\varepsilon\in (0,1]$ it holds
	that
	$\functionANN(\mathfrak{u}_{d,\varepsilon})\in C(\R^{d+1},\R)$,
	$\paramANN(\mathfrak{u}_{d,\varepsilon})\le c \varepsilon^{-c}  d^{c}$, and 
	\begin{equation}
	\begin{split}
	& \left[
	\int_{[0,T]\times[\alpha, \beta]^d} 
	\frac{	\left\vert
		u_d(y) 
		- 
		(\functionANN(\mathfrak{u}_{d,\varepsilon}))(y)
		\right\vert^p\!}{|\beta - \alpha|^d} \, 
	d y
	\right]^{\!\nicefrac{1}{p}} 
	\le
	\varepsilon.
	\end{split}
	\end{equation}
\end{cor}

\begin{proof}[Proof of Corollary~\ref{cor:DNNerrorEstimateLebesgueMeasure}]
	Throughout this proof let $p,q,\eta\in (0,\infty)$ satisfy that $q=\max\{p,2\}$ and $\eta= \max\{6\kappa, 2\kappa +2, 3\}+[\max\{1,T\}]^{\nicefrac{1}{q}}  \max\{1, |\alpha|^{2\max\{6\kappa, 2\kappa +2, 3\}}, |\beta|^{2\max\{6\kappa, 2\kappa +2, 3\}}\} $, for every $d\in\N$ let $\mu_d \colon \mathcal{B}([0,T]\times \R^d)\allowbreak \to [0,\infty]$ be the  Lebesgue-Borel measure on $[0,T]\times \R^d$, for every $d\in\N$ let $\nu_d \colon \mathcal{B}([0,T]\times \R^d)\allowbreak \to [0,\infty]$ be the measure which satisfies for all $d\in\N$, $B_1\in \mathcal{B}([0,T])$, $B_2\in \mathcal{B}(\R^d)$ that 
	\begin{equation}\label{corDNNerrorEstimate:Defnu}
	\nu_d(B_1\times B_2)= \frac{\mu_d(B_1\times ([\alpha,\beta]^d\cap B_2))}{|\beta-\alpha|^d},
	\end{equation}
	let $\delta\colon (0,1]\to (0,1]$ satisfy for all $\varepsilon\in (0,1]$ that $\delta(\varepsilon)=\varepsilon [\max\{T,1\}]^{\nicefrac{1}{q}-\nicefrac{1}{p}}$,
	and let $r\colon (0,\infty)\to (0,\infty)$ satisfy for all $z\in (0,\infty)$ that $r(z)=z  [\max\{T,1\}]^{z(\nicefrac{1}{p}-\nicefrac{1}{q})}$.
	Observe that \eqref{corDNNerrorEstimate:Defnu},  Fubini's theorem, and, e.g.,  Grohs et al.~\cite[Lemma~3.15]{GrohsWurstemberger2023} prove that for all $d\in\N$ it holds that 
	\begin{equation}
	\begin{split}
	&\int_{[0,T]\times \R^d} 
	\|x \|^{2q \max\{6\kappa, 2\kappa +2, 3\}}
	\, \nu_d (d t, d x)
	=\int_{[0,T]\times [\alpha, \beta]^d} 
	\frac{	\|x \|^{2q \max\{6\kappa, 2\kappa +2, 3\}}}{|\beta-\alpha|^d}
		\, \mu_d (d t, d x)
	\\&= T  \int_{[\alpha,\beta]^d} 
	\frac{	\|x \|^{2q \max\{6\kappa, 2\kappa +2, 3\}}}{|\beta-\alpha|^d}
	\, dx\\
	&\le T d^{q \max\{6\kappa, 2\kappa +2, 3\}} \max\!\big\{|\alpha|^{2q\max\{6\kappa, 2\kappa +2, 3\}}, |\beta|^{2q\max\{6\kappa, 2\kappa +2, 3\}}\big\}.
	\end{split}
	\end{equation}
	Therefore, we obtain for all $d\in\N$  that 
	\begin{equation}
	\begin{split}
	&\left[\max\!\left\{1,\nu_d([0,T]\times\R^d),\int_{[0,T]\times \R^d} 
	\|x \|^{2q \max\{6\kappa, 2\kappa +2, 3\}}
	\, \nu_d (d t, d x) \right\}\right]^{\!\nicefrac{1}{q}}
\\&\le \left[\max\!\left\{1,T,T d^{q \max\{6\kappa, 2\kappa +2, 3\}} \max\!\big\{|\alpha|^{2q\max\{6\kappa, 2\kappa +2, 3\}}, |\beta|^{2q\max\{6\kappa, 2\kappa +2, 3\}}\big\} \right\}\right]^{\!\nicefrac{1}{q}}\\
&\le d^{ \max\{6\kappa, 2\kappa +2, 3\}} \left[\max\!\left\{1,T\right\}\right]^{\nicefrac{1}{q}} \max\!\big\{1, |\alpha|^{2\max\{6\kappa, 2\kappa +2, 3\}}, |\beta|^{2\max\{6\kappa, 2\kappa +2, 3\}}\big\} \le \eta d^\eta.
	\end{split}
	\end{equation}
	Corollary~\ref{cor:DNNerrorEstimate} hence ensures that  there exist $\mathfrak{C}\in (0, \infty)$ and $(\Phi_{d,\varepsilon})_{d\in\N,\varepsilon\in (0,1]}\subseteq \ANNs$ which satisfy for all 
	$d\in\N$, $\varepsilon\in (0,1]$ 
	that
	$\functionANN(\Phi_{d,\varepsilon})\in C(\R^{d+1},\R)$,
	$\paramANN(\Phi_{d,\varepsilon})\le \mathfrak{C} \varepsilon^{-\mathfrak{C}}  d^{\mathfrak{C}}$, and 
	\begin{equation}
	\begin{split}
	& \left[
	\int_{[0,T]\times\R^d}
	\left\vert
	u_d(y) 
	- 
	(\functionANN(\Phi_{d,\varepsilon}))(y)
	\right\vert^q
	\,
	\nu_d (d y)
	\right]^{\!\nicefrac{1}{q}} 
	\le
	\varepsilon.
	\end{split}
	\end{equation}
	Combining this with \eqref{corDNNerrorEstimate:Defnu} and  H\"older's inequality proves that 
	for all 
	$d\in\N$, $\varepsilon\in (0,1]$ it holds 
	that
	\begin{equation}\label{DNNerrorEstimateLebesgueError}
	\begin{split}
	& \left[
	\int_{[0,T]\times[\alpha, \beta]^d} 
	\frac{	\left\vert
		u_d(y) 
		- 
		(\functionANN(\Phi_{d,\delta(\varepsilon)}))(y)
		\right\vert^p\!}{|\beta - \alpha|^d} \, 
	d y
	\right]^{\!\nicefrac{1}{p}}  \\
&\le \left[
	\int_{[0,T]\times[\alpha, \beta]^d}
	\frac{	\left\vert
	u_d(y) 
	- 
	(\functionANN(\Phi_{d,\delta(\varepsilon)}))(y)
	\right\vert^q\!}{|\beta - \alpha|^d} \, 
	dy
	\right]^{\!\nicefrac{1}{q}} T^{\nicefrac{1}{p}-\nicefrac{1}{q}}
	\\&=
	\left[
	\int_{[0,T]\times\R^d}
	\left\vert
	u_d(y) 
	- 
	(\functionANN(\Phi_{d,\delta(\varepsilon)}))(y)
	\right\vert^q
	\,
	\nu_d (d y)
	\right]^{\!\nicefrac{1}{q}} T^{\nicefrac{1}{p}-\nicefrac{1}{q}}
	\\&\le
	\delta(\varepsilon)\, [\max\{T,1\}]^{\nicefrac{1}{p}-\nicefrac{1}{q}}
	=\varepsilon [\max\{T,1\}]^{\nicefrac{1}{q}-\nicefrac{1}{p}} [\max\{T,1\}]^{\nicefrac{1}{p}-\nicefrac{1}{q}}=\varepsilon.
	\end{split}
	\end{equation}
	In addition, observe that for all $z\in(0,\infty)$ it holds that 
	\begin{equation}
	z\le z  [\max\{T,1\}]^{z(\nicefrac{1}{p}-\nicefrac{1}{q})}=r(z).
	\end{equation}
	The fact that 
	$\forall \, d\in\N$, $\varepsilon\in (0,1]\colon \functionANN(\Phi_{d,\varepsilon})\in C(\R^{d+1},\R)$ and 	the fact that 
	$\forall \, d\in\N$, $\varepsilon\in (0,1]\colon \paramANN(\Phi_{d,\varepsilon})\le \mathfrak{C} \varepsilon^{-\mathfrak{C}}  d^{\mathfrak{C}}$ hence show that for all 
	$d\in\N$, $\varepsilon\in (0,1]$ it holds 
	that
	$\functionANN(\Phi_{d,\delta(\varepsilon)})\in C(\R^{d+1},\R)$ and 
	\begin{equation}
	\begin{split}
	\paramANN(\Phi_{d,\delta(\varepsilon)})&\le \mathfrak{C} [\delta(\varepsilon)]^{-\mathfrak{C}}  d^{\mathfrak{C}}
	=\mathfrak{C}  [\max\{T,1\}]^{-\mathfrak{C}(\nicefrac{1}{q}-\nicefrac{1}{p})} \varepsilon^{-\mathfrak{C}}  d^{\mathfrak{C}}
	\\&=\mathfrak{C}  [\max\{T,1\}]^{\mathfrak{C}(\nicefrac{1}{p}-\nicefrac{1}{q})} \varepsilon^{-\mathfrak{C}}  d^{\mathfrak{C}}
	=r(\mathfrak{C}) \varepsilon^{-\mathfrak{C}}  d^{\mathfrak{C}}\le r(\mathfrak{C}) \varepsilon^{-r(\mathfrak{C})}  d^{r(\mathfrak{C})}.
	\end{split}
	\end{equation}
	This and \eqref{DNNerrorEstimateLebesgueError} establish that there exist $c\in \R$ and $(\mathfrak{u}_{d,\varepsilon})_{d\in\N,\varepsilon\in (0,1]}\subseteq \ANNs$ such that
	for all 
	$d\in\N$, $\varepsilon\in (0,1]$ it holds
	that
	$\functionANN(\mathfrak{u}_{d,\varepsilon})\in C(\R^{d+1},\R)$,
	$\paramANN(\mathfrak{u}_{d,\varepsilon})\le c \varepsilon^{-c}  d^{c}$, and 
	\begin{equation}
	\begin{split}
	& \left[
\int_{[0,T]\times[\alpha, \beta]^d} 
\frac{	\left\vert
	u_d(y) 
	- 
	(\functionANN(\mathfrak{u}_{d,\varepsilon}))(y)
	\right\vert^p\!}{|\beta - \alpha|^d} \, 
d y
\right]^{\!\nicefrac{1}{p}} 
\le
\varepsilon.
	\end{split}
	\end{equation}	
	This completes the proof of	Corollary~\ref{cor:DNNerrorEstimateLebesgueMeasure}.
\end{proof}

\begin{cor}
	\label{cor:DNNerrorEstimateLaplace}
	Let 
	$ f^{m}_d \colon \R^d \to \R^{md-m+1} $,  $ m \in \{0, 1\}$, $d \in \N $,  
	be functions,
	let 
	$ T, \kappa, p \in (0,\infty) $,   $a \in \R$, $b \in (a, \infty)$,
	$
	( \mathfrak{f}^{ m }_{d, \varepsilon } )_{ 
		(m, d, \varepsilon) \in \{ 0, 1 \} \times \N \times (0,1] 
	} 
	\subseteq \ANNs
	$, 
	assume for all
	$m\in\{0,1\}$, 
	$ d \in \N $, 
	$ \varepsilon \in (0,1] $, 
	$ 
	x,y \in \R^d
	$
	that 
	\vspace{-1ex}
	\begin{gather}
	\functionANN( \mathfrak{f}^{ 0 }_{d, \varepsilon } )
	\in 
	C( \R^d,  \R), \qquad 	\functionANN( \mathfrak{f}^{ 1 }_{d, \varepsilon } )
	\in 
	C( \R^d,  \R^{ d }),  \qquad 	\mathcal{P}( \mathfrak{f}^{ m }_{d, \varepsilon } ) 
	\leq \kappa d^{ \kappa } \varepsilon^{ - \kappa }, \\
	\label{corLaplaceDNNerrorEstimate:fOne}
	\| 
	f^{ 1 }_d( x ) 
	- 
	f^{ 1}_d( y )
	\|
	\leq 
	\kappa 
	\| x - y \|, \qquad \|
	(\functionANN(\mathfrak{f}^{ 1 }_{d,  \varepsilon }))(x)    
	\|	
	\leq 
	\kappa ( d^{ \kappa } + \| x \| ),
	\\
	\label{corLaplaceDNNerrorEstimate:approximationLocallyLipschitz}
	| (\functionANN(\mathfrak{f}^{ 0 }_{d, \varepsilon })) (x) - (\functionANN(\mathfrak{f}^{ 0 }_{d,  \varepsilon })) (y)| \leq \kappa d^{\kappa} (1 + \|x\|^{\kappa} + \|y \|^{\kappa})\|x-y\|, \\
	\label{corLaplaceDNNerrorEstimate:appr:coef}
	\text{and} \qquad		\varepsilon|
	f^{ 0 }_d( x )
	| 
	+\| f^{ m }_d(x) 
	- 
	(\functionANN(\mathfrak{f}^{ m }_{d, \varepsilon })) (x)
	\|
	\leq 
	\varepsilon \kappa d^{ \kappa }
	(
	1 + \| x \|^{ \kappa }
	),
	\end{gather}
	and	for every $d \in \N$	let
	$u_d \in  \{v \in C([0, T] \times \R^d, \R) \colon \allowbreak \inf_{ q \in (0,\infty) }
	\allowbreak \sup_{ (t,y) \in [0,T] \times \R^d } \allowbreak
	\frac{ | v(t,y) | }{
		1 + \| y \|^q
	}
	< \infty \}$ be a viscosity solution of
	\begin{equation}
	\begin{split}
	( \tfrac{ \partial }{\partial t} u_d )( t, x ) 
	& = 
	(\Delta_x u_d )( t, x ) +
	( \tfrac{ \partial }{\partial x} u_d )( t, x )
	\,
	f^{ 1}_d( x )
	\end{split}
	\end{equation}
	with 	$ u_d( 0, x ) = f^{ 0 }_d( x )$
	for $ ( t, x ) \in (0,T) \times \R^d $ (cf.\ \Cref{Def:euclideanNorm,Def:ANN,Definition:ANNrealization,Definition:Relu1}).
	Then  there exist $c\in \R$ and $(\mathfrak{u}_{d,\varepsilon})_{(d, \varepsilon) \in\N\times (0,1]}\subseteq \ANNs$ such that
	for all 
	$d\in\N$, $\varepsilon\in (0,1]$ it holds
	that
	$\functionANN(\mathfrak{u}_{d,\varepsilon})\in C(\R^{d+1},\R)$,
	$\paramANN(\mathfrak{u}_{d,\varepsilon})\le c \varepsilon^{-c}  d^{c}$, and 
	\begin{equation}
	\begin{split}
	& \left[
\int_{[0,T]\times[a, b]^d} 
\frac{	\left\vert
	u_d(y) 
	- 
	(\functionANN(\mathfrak{u}_{d,\varepsilon}))(y)
	\right\vert^p\!}{|b - a|^d} \, 
d y
\right]^{\!\nicefrac{1}{p}} 
\le
\varepsilon.
	\end{split}
	\end{equation}
\end{cor}

\begin{proof}[Proof of Corolllary~\ref{cor:DNNerrorEstimateLaplace}]
	Throughout this proof let $\iota=\max\{3\kappa,2(\kappa+1)\}$.
	Observe that \eqref{corLaplaceDNNerrorEstimate:fOne} and the fact that $\iota\ge \kappa$ prove that for all $d\in\N$, $\varepsilon\in (0,1]$, $x,y\in\R^d$ it holds that 
	\begin{equation}\label{corDNNerrorEstimateLaplace:HypoOne}
	\| 
	f^1_d( x ) 
	- 
	f^1_d( y )
	\|
	\leq 
	\iota
	\| x - y \|
	\qandq
	\|(\functionANN(\mathfrak{f}^1_{d, \varepsilon}))(x)    
	\|	
	\leq 
	\iota ( d^{ \iota } + \| x \| ).
	\end{equation}
	Next note that \eqref{corLaplaceDNNerrorEstimate:approximationLocallyLipschitz}  and the fact that $\iota\ge 3\kappa$ ensure that for all $d\in\N$, $\varepsilon\in (0,1]$, $x,y\in\R^d$ it holds that 
	\begin{equation}\label{corDNNerrorEstimateLaplace:HypoTwo}
	\begin{split}
	| (\functionANN(\mathfrak{f}^0_{d, \varepsilon})) (x) - (\functionANN(\mathfrak{f}^0_{d, \varepsilon})) (y)| 
	&\leq \kappa d^{\iota} (3 + \|x\|^{\iota} + \|y \|^{\iota})\|x-y\|
	\\& \leq \iota d^{\iota} (1 + \|x\|^{\iota} + \|y \|^{\iota})\|x-y\|.
	\end{split}
	\end{equation}
	In addition, observe that \eqref{corLaplaceDNNerrorEstimate:appr:coef} and the fact that $\iota \geq2 (\kappa+1)$ show that  for all
	$ d \in \N $, 
	$ \varepsilon \in (0,1] $, 
	$m\in\{0,1\}$, 
	$ 
	x \in \R^d
	$ it holds that 
	\begin{equation}
	\begin{split}
	&	\varepsilon|
	f^0_d( x )
	|  + \varepsilon
	+\| f^m_d(x) 
	- 
	(\functionANN(\mathfrak{f}^m_{d, \varepsilon})) (x)
	\|
	\leq 
	\varepsilon \kappa d^{ \kappa }
	(
	1 + \| x \|^{ \kappa }
	) + \varepsilon\\
	& \leq \varepsilon (\kappa + 1) d^{ \kappa }
	(
	1 + \| x \|^{ \kappa }
	)  \leq \varepsilon (\kappa + 1) d^{ \kappa }
	(
	2 + \| x \|^{ \iota }
	) \leq \varepsilon \iota d^{\iota} (1 + \|x\|^{\iota}).
	\end{split}
	\end{equation}
	Combining this, \eqref{corDNNerrorEstimateLaplace:HypoOne}, \eqref{corDNNerrorEstimateLaplace:HypoTwo}, and Corollary~\ref{cor:DNNerrorEstimateLebesgueMeasure}  implies that  there exist $c\in  \R$ and \linebreak \allowbreak $(\mathfrak{u}_{d,\varepsilon})_{(d, \varepsilon)\in\N\times (0,1]}\subseteq \ANNs$ such that
	for all 
	$d\in\N$, $\varepsilon\in (0,1]$ it holds
	that
	$\functionANN(\mathfrak{u}_{d,\varepsilon})\in C(\R^{d+1},\R)$,
	$\paramANN(\mathfrak{u}_{d,\varepsilon})\le c \varepsilon^{-c}  d^{c}$, and 
	\begin{equation}
	\begin{split}
	& \left[
\int_{[0,T]\times[a, b]^d} 
\frac{	\left\vert
	u_d(y) 
	- 
	(\functionANN(\mathfrak{u}_{d,\varepsilon}))(y)
	\right\vert^p\!}{|b - a|^d} \, 
d y
\right]^{\!\nicefrac{1}{p}} 
\le
\varepsilon.
	\end{split}
	\end{equation}
	This completes the proof of	Corollary~\ref{cor:DNNerrorEstimateLaplace}.
\end{proof}